\documentclass[a4paper,11pt,reqno]{amsart}
\usepackage{amsfonts}
\usepackage{amsmath, amssymb, amsthm}
\usepackage[foot]{amsaddr}
\usepackage{graphicx}
\usepackage{subfigure}
\usepackage{color}
\usepackage{xspace}
\usepackage{framed}
\usepackage{algorithm}
\usepackage[noend]{algpseudocode}
\usepackage{multirow}
\usepackage[a4paper]{geometry}
\usepackage{amsmath}
\usepackage{amssymb}
\usepackage{xr}

\newtheorem{theorem}{Theorem}

\newtheorem{example}{Example}
\newtheorem{lemma}{Lemma}

\newtheorem{proposition}{Proposition}
\newtheorem{remark}{Remark}

\renewcommand{\Pr}{\mathbb{P}} 
\newcommand{\E}{\mathbb{E}} 

\begin{document}

\title[RWPI and applications to Machine
Learning]{
  \hspace{40pt}
  \large \lowercase{
    \uppercase{R}obust \uppercase{W}asserstein \uppercase{P}rofile \uppercase{I}nference
    and \newline 
   \uppercase{A}pplications to \uppercase{M}achine \uppercase{L}earning}
}

\author[Blanchet]{Jose Blanchet}
\address[A1]{Management Science and Engineering, Stanford University}

\author[Kang]{Yang Kang}
\address[A2]{Department of Statistics, Columbia University}

\author[Murthy]{Karthyek Murthy}
\address[A3]{Engineering Systems \& Design, Singapore University of
  Technology \& Design} 

\email{jose.blanchet@stanford.edu, yangkang@stat.columbia.edu,
  karthyek\_murthy@sutd.edu.sg} 

\maketitle
\begin{abstract}
  We show that several machine learning estimators, including
  square-root LASSO (Least Absolute Shrinkage and Selection) and
  regularized logistic regression can be represented as solutions to
  distributionally robust optimization (DRO) problems. The associated
  uncertainty regions are based on suitably defined Wasserstein
  distances. Hence, our representations allow us to view
  regularization as a result of introducing an artificial adversary
  that perturbs the empirical distribution to account for
  out-of-sample effects in loss estimation. In addition, we introduce
  RWPI (Robust Wasserstein Profile Inference), a novel inference
  methodology which extends the use of methods inspired by Empirical
  Likelihood to the setting of optimal transport costs (of which
  Wasserstein distances are a particular case). We use RWPI to show
  how to optimally select the size of uncertainty regions, and as a
  consequence, we are able to choose regularization parameters for
  these machine learning estimators without the use of cross
  validation. Numerical experiments are also given to validate our
  theoretical findings.
\end{abstract}



\section{Introduction\label{Sec-Introduction}}
Regularization has become crucial in machine learning practice and the
goal of this paper is to revisit the idea of regularization from an
optimal transport perspective. Specifically, we show that the role of
regularization in machine learning can often be interpreted as the
result of optimally transporting mass from the empirical measure in
order to maximize a certain loss under a budget constraint. Thus, our
results connect directly optimal transport phenomena (a classical
concept in probability reviewed in Section \ref{Sec-OT-Wass-Defn}) to
regularization (a key tool in machine learning to be discussed in the
sequel).

Moreover, this connection will show that the so-called regularization
parameter (i.e. the coefficient of the regularization term) coincides
with the size of the budget constraint by which we permit mass
transportation to occur.  
As we shall see, the budget constraint has a natural interpretation
based on a Distributionally Robust Optimization (DRO) formulation,
which in turn allows us to define a reasonable optimization criterion
for the regularization parameter. Thus, our approach uses optimal mass
transportation phenomena to explain the nature of regularization and
how to select the regularization parameter in several machine learning
estimators -- including square-root LASSO (Least Absolute Shrinkage
and Selection), Regularized Logistic Regression, among others.

The size of the budget constraint is also referred as the radius (or
size) of the uncertainty set in the literature of Distributionally
Robust Optimization (DRO). The method that we develop for optimally
choosing this budget constraint can actually be applied to a wide
range of inference and decision problems, but we have focused our
discussion on machine learning applications because of the substantial
amount of activity that the area has generated, and as well to
demonstrate the utility of the tools that are commonly used in applied
probability to this rapidly growing area.

\subsection{Regularization in Linear
  Regression\label{Sect_Intro_RWPI_LASSO}}
In order to introduce the proposed method for optimally choosing the
radius of the uncertainty set, let us walk through a simple
application in a familiar context, namely, that of linear regression.
Throughout the paper any vector is understood to be a column vector
and the transpose of $x$ is denoted by $x^{T}.$ We use the notation
$\E_{\Pr}[\cdot]$ to denote expectation with respect to a probability
distribution $\Pr$.

\begin{example}[Square-root Lasso]
\label{Ex_GLasso}
\textnormal{\sloppy{Consider a training data set $\{(X_{1},Y_{1}),\ldots,(X_{n}%
  ,Y_{n})\}$,} where the input $X_{i}\in\mathbb{R}^{d}$ is a vector of
$d$ predictor variables, and $Y_{i}\in\mathbb{R}$ is the response
variable.  It is postulated that
\[
Y_{i}=\beta_{\ast}^{T}X_{i}+e_{i},
\]
for some $\beta_{\ast}\in\mathbb{R}^{d}$ and errors $\left\{ e_{1}%
  ,...,e_{n}\right\} $. Under suitable statistical assumptions one may
be interested in estimating $\beta_{\ast}$. Underlying is a general
loss function, $l( x,y;\beta) $, which we shall take for
simplicity in this discussion to be the quadratic loss, namely,
$l(x,y;\beta)=( y-\beta^{T}x) ^{2}$. Let $\Pr_{n}$ denote the
empirical distribution: 
\[
\Pr_{n}\left(  dx,dy\right)  :=\frac{1}{n}\sum_{i=1}^{n}\mathcal{\delta
}_{\left\{  (X_{i},Y_{i})\right\}  }(dx,dy).
\]
Over the last two decades, various regularized estimators have been
introduced and studied. Many of them have gained substantial
popularity because of their good empirical performance and insightful
theoretical properties, (see, for example,
\cite{tibshirani_regression_1996} for an early reference and
\cite{hastie_elements_2005} for a discussion on regularized
estimators). One such regularized estimator, implemented, for example
in the \textquotedblleft flare\textquotedblright\ package, see
\cite{li_flare_2015}, is the so-called square-root LASSO estimator;
which is obtained by solving the following convex optimization problem
in $\beta$
\begin{equation}
\min_{\beta\in\mathbb{R}^{d}}\left\{  \sqrt{\E_{\Pr_{n}}\left[  l\left(
X,Y;\beta\right)  \right]  }+\lambda\left\Vert \beta\right\Vert _{1}\right\}
=\min_{\beta\in\mathbb{R}^{d}}\left\{  \sqrt{\frac{1}{n}\sum_{i=1}^{n}l\left(
X_{i},Y_{i};\beta\right)  }+\lambda\left\Vert \beta\right\Vert _{1}\right\}  ,
\label{RegEs}%
\end{equation}
where $\left\Vert \beta\right\Vert _{p}$ denotes the
$\ell_p-$norm. The parameter $\lambda$, commonly referred to as the
regularization parameter, is crucial for the performance of the
algorithm.  It is often chosen using cross validation, a procedure
that iterates over multitude of choices of $\lambda$ in order to
choose the best.} 
\end{example}

\subsubsection{DRO representation of square-root LASSO.
  \label{Subsec_DR_Rep_LASSO}} One of our contributions in this paper
(see Section \ref{Sec-ML-RWPI}) is a representation of (\ref{RegEs})
in terms of a Distributionally Robust Optimization formulation. We
construct a discrepancy measure, $\mathcal{D}%
_{c}\left( \Pr,\mathbb{Q}\right) $, corresponding to a
Wasserstein-type distance between two probability measures $\Pr$ and
$\mathbb{Q}$ which is defined in terms of a suitable transportation
cost function $c(\cdot)$. If $c( \cdot) $ is based on the
$\ell_{q}-$distance (for $q > 1$), we show that
\begin{equation}
  \min_{\beta\in\mathbb{R}^{d}}\left\{  \sqrt{\E_{\Pr_{n}}\left[  l\left(
          X,Y;\beta\right)  \right]  }+\lambda\left\Vert \beta\right\Vert _{p}\right\}
  ^{2}=\min_{\beta\in\mathbb{R}^{d}}\ \ \max_{\Pr:\ \mathcal{D}_{c}(\Pr,\Pr_{n}%
    )\leq\delta}\ \ \E_{\Pr}\left[  l(X,Y;\beta)\right]  ,\label{AUX_DRO_1}%
\end{equation}
where $1/p+1/q=1$ and $\lambda = \sqrt{\delta}.$ We can gain a great
deal of insight from (\ref{AUX_DRO_1}). For example, note that the
regularization parameter, $\lambda = \sqrt{\delta},$ is fully
determined by the size (or `radius') of the uncertainty, $\delta$, in
the distributionally robust optimization formulation on the right hand
side of (\ref{AUX_DRO_1}). In addition, we can interpret
(\ref{AUX_DRO_1}) as a game in which an artificial adversary is
introduced in order to explore and quantify out-of-sample effects in
our estimates of the expected loss.


\subsubsection{Optimal choice of the radius $\delta$. 
  \label{Subsec_Opt_Reg_LASSO_Int}}

The set
$\mathcal{U}_{\delta}(\Pr_{n})=\{\Pr:\mathcal{D}_{c}(\Pr,\Pr_{n})\leq\delta\}$
is called the uncertainty set in the language of distributionally
robust optimization, and it represents the class of models that are,
in some sense, plausible variations of $\Pr_{n}$. Note that
$\mathcal{U}_{\delta}(\Pr_{n})$ is precisely the feasible region over
which the maximization is taken in (\ref{AUX_DRO_1}). Then we define
the collection,
\begin{align}
  \Lambda_{n}(\delta) := \bigcup_{\Pr \, \in \,
  \mathcal{U}_\delta(\Pr_n)} \hspace{-10pt} 
  \arg \min_{\beta \in \mathbb{R}^d} \E_\Pr\left[l(X,Y;\beta)\right], 
\label{Lambda-n-Defn}
\end{align}
comprising optimal $\beta$ for every $\Pr \in \mathcal{U}_\delta(\Pr_n)$
to be the set of \textit{plausible} selections of the parameter
$\beta_\ast.$
For $\delta$ chosen sufficiently large, the set
$\Lambda_{n}( \delta) $ is a natural confidence region for
$\beta_{\ast}$. Moreover, we shall see that any $\beta$ that solves
$\inf_\beta \sup_{\Pr \in \mathcal{U}_\delta(\Pr_n)} \E\left[
  l(X,Y;\beta)\right]$ is a member of $\Lambda_n(\delta).$

Given these interpretations, it is natural to select a confidence
level, $1-\alpha$, and then choose $\delta = \delta_n^\ast$
optimally via, 
\begin{equation}
  \delta_{n}^{\ast}=\min\{\delta > 0: \Pr \left(
    \beta_{\ast}\in\Lambda_{n}\left( 
      \delta\right)  \right)  \geq 1-\alpha\}.\label{Opt_Choice_Delta}  
\end{equation}
In words, the optimization criterion can be stated as finding the
smallest $\delta$ such that $\beta_{\ast}$ is, itself, a plausible
selection with $1-\alpha$ confidence. 
Essentially, given a desired confidence level $1-\alpha,$ we seek to
choose a $\delta$ just large enough such that $\Lambda_n(\delta)$ is a
$(1-\alpha)-$confidence region for the parameter $\beta_\ast.$ As we
shall see in Section \ref{Sec-ML-RWP}, this choice ensures that any
$\beta$ that minimizes
$\inf_\beta \sup_{\Pr \in \mathcal{U}_\delta(\Pr_n)} \E\left[
  l(X,Y;\beta)\right]$ is indeed in the confidence region
$\Lambda_n(\delta).$ We next explain how to solve the optimization
problem in \eqref{Opt_Choice_Delta} asymptotically as
$n\rightarrow\infty$.

\subsubsection{The associated Wasserstein Profile Function.
\label{Subsec_Intro_WPF_LASSO}}
In order to asymptotically solve (\ref{Opt_Choice_Delta}) we introduce
a novel statistical inference methodology, which we call RWPI (Robust
Wasserstein-distance Profile-based Inference -- pronounced similar to
Rupee)
. RWPI can be understood as an extension of Empirical
Likelihood (EL) that uses optimal transport cost rather than the
likelihood. The extension is not just a formality, as we shall see,
because different phenomena and scalings arise relative to EL.

We next illustrate how $\delta_{n}^{\ast}$ in \eqref{Opt_Choice_Delta}
corresponds to the quantile of a certain object which we call the
Robust Wasserstein Profile (RWP) function evaluated at
$\beta_{\ast}$. This will motivate a systematic study of the RWP
function as the sample size, $n$, increases.

Observe by convexity of the loss function that
$\beta\in\Lambda_{\delta }( \Pr_{n}) $ if and only if there exists
$\Pr\in\mathcal{U}_{\delta }(\Pr_{n})$ such that $\beta$ satisfies the
first order optimality condition, namely,%
\begin{equation}
  D_\beta \E_\Pr\left[ l(X,Y;\beta)\right] = \E_{\Pr}\left[  \left(
      Y-\beta^{T}X\right)  X\right]  =\mathbf{0} 
  .\label{EST_EQN}%
\end{equation}
We then introduce the following object, which is the RWP function associated
with the estimating equation (\ref{EST_EQN}),
\begin{equation}
  R_{n}( \beta)  =\inf\left\{  \mathcal{D}_{c}\left(  \Pr,\Pr_{n}%
    \right)  :\E_{\Pr}\left[  \left(  Y-\beta^{T}X\right)  X\right]  =\mathbf{0}%
  \right\}  .\label{PROF_R}%
\end{equation}
It turns out that the infimum is achieved in the previous expression,
so we can write $\min$ instead; this is not crucial for our discussion
but it is sometimes helpful to keep in mind. Using this definition of
$R_{n}(\beta),$ we can see immediately that the events, 
\[
  \{R_{n}(\beta_{\ast}) \leq\delta\}=\{\beta_{\ast}\in\Lambda _{n}(\delta) \},
\]
which implies that $\delta_{n}^{\ast}$ is precisely the $1-\alpha$
quantile, $\chi_{_{1-\alpha}}$, of $R_{n}( \beta_{\ast});$ that is
\[
  \delta_{n}^{\ast}=\chi_{_{1-\alpha}}=\inf\big\{z:\Pr\left( R_{n}(
    \beta_{\ast}) \leq z\right) \geq1-\alpha\big\}.
\]
Moreover, note that $R_{n}(\beta)$ allows to provide an explicit
characterization of $\Lambda_{n}(\chi_{_{1-\alpha}}),$ namely, 
\[
  \Lambda_{n}( \chi_{_{1-\alpha}}) =\{\beta:R_{n}(\beta)
  \leq\chi_{_{1-\alpha}}\}.
\]
So, $\Lambda_{n}\left(  \chi_{_{1-\alpha}}\right)  = \{\beta:
R_n(\beta) \leq \chi_{1-\alpha}\}$ is a $(1-\alpha)$-confidence
region for $\beta^{\ast}$.

\subsubsection{Further intuition behind the RWP function.
\label{Subsec_Intuit_RWP}}
In order to further explain the role of $R_{n}(\beta_{\ast})$, let us define
$ \mathcal{P}_{opt}:=\left\{ \Pr:\E_{\Pr}\left[ \big(Y-\beta_{\ast}^{T}%
    X\big)X\right] =\mathbf{0}\right\}.$ In words, $\mathcal{P}_{opt}$
is the set of probability measures for which $\beta_{\ast}$ is an
optimal risk minimization parameter. Naturally, the distribution of
$(X,Y),$ from which the samples are generated, is an element of
$\mathcal{P}_{opt}.$ Since
$R_{n}(\beta_{\ast})=\inf\{\mathcal{D}_{c}(\Pr,\Pr_{n}):\Pr\in\mathcal{P}_{opt}\},$
the set $\{\Pr:\mathcal{D}_{c}(\Pr,\Pr_{n})\leq R_{n}(\beta_{\ast})\}$
denotes the smallest uncertainty region around $\Pr_{n}$ (in terms of
$\mathcal{D}_{c}$) for which there exists a distribution $\Pr$
satisfying the optimality condition
$\E_{\Pr}\left[ (Y-\beta_{\ast}^{T}X)X\right] =\mathbf{0}.$ See Figure
\ref{Fig-Intuition-DRRegression} for a pictorial representation of
$\mathcal{P}_{opt}$ and $R_n(\beta^\ast).$
\begin{figure}[h]
  \begin{center}
  \includegraphics[width=10cm]{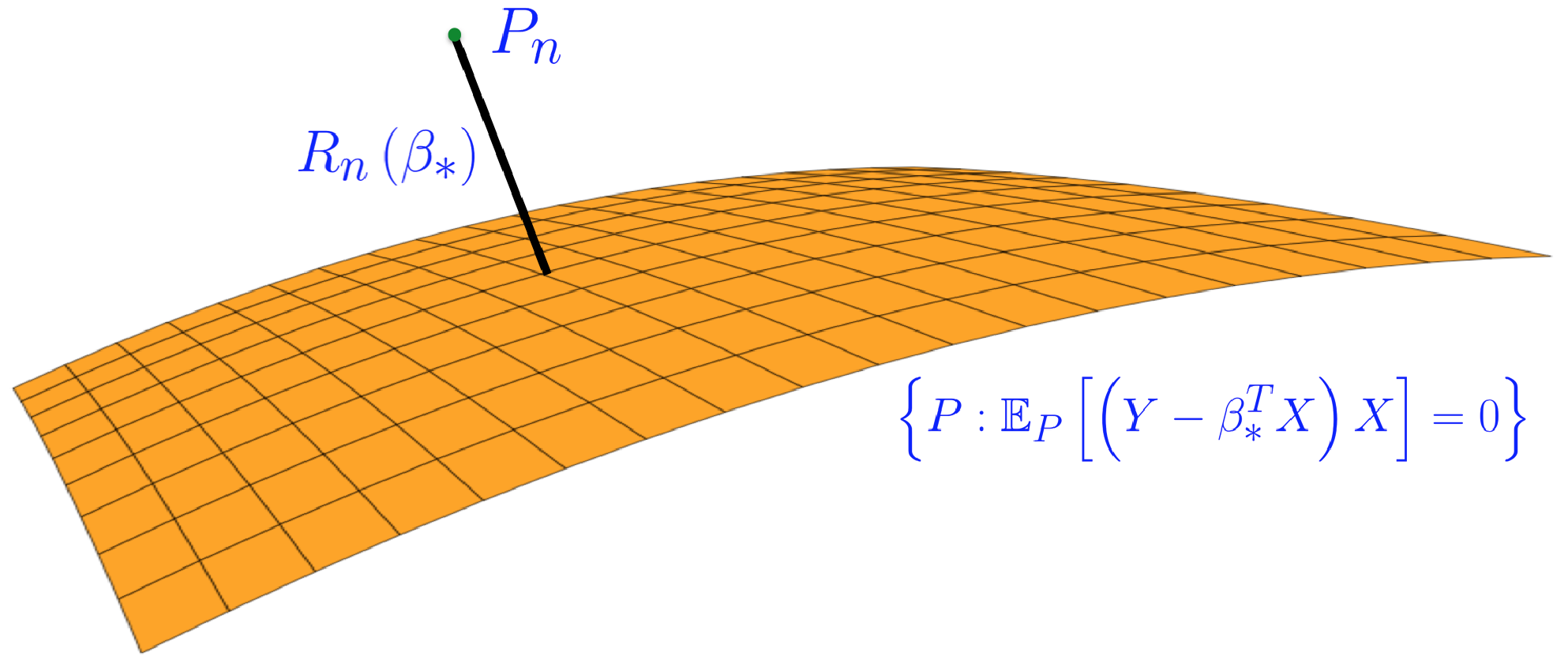}
  \caption{Illustration of RWP function evaluated at $\beta_{\ast}$}
  \end{center} 
  \label{Fig-Intuition-DRRegression} 
\end{figure}

In summary, $R_{n}(\beta_{\ast})$ denotes the smallest size of uncertainty
that makes $\beta_{\ast}$ a \textit{plausible choice}. If we were to select a
radius of uncertainty smaller than $R_{n}(\beta_{\ast})$, then no probability
measure in the neighborhood will satisfy the optimality condition
$\E_{\Pr}\left[  (Y-\beta_{\ast}^{T}X)X\right]  =\mathbf{0}$. On the other hand,
if $\delta>R_{n}(\beta_{\ast}),$ the set
\[
\left\{  \Pr:\E_{\Pr}\big[(Y-\beta_{\ast}^{T}X)X\big]=\mathbf{0},\mathcal{D}%
_{c}\big(\Pr,\Pr_{n}\big)\leq\delta\right\}
\] is nonempty. 


\subsection{A broader perspective of our contribution}

The previous discussion in the context of linear regression highlights two key
ideas: a) the RWP function as a key object of analysis, and b) the role of
distributionally robust representation of regularized estimators.

The RWP function can be applied much more broadly than in the context
of regularized estimators. We shall study the RWP function for
estimating equations\ generally and systematically but we showcase the
use of the RWP function only in the context of 
optimal regularization.

Broadly speaking, RWPI can be seen as a statistical methodology that
utilizes a suitably defined RWP function to estimate a parameter of
interest. From a philosophical standpoint, RWPI borrows heavily from
Empirical Likelihood (EL), introduced in the seminal work of
\cite{owen_empirical_1988,owen_empirical_1990}. There are important
methodological differences, however, as we shall discuss in the
sequel. In the last three decades, there have been a great deal of
successful applications of Empirical Likelihood for inference
\cite{owen_empirical_1991,zhou2015empirical,bravo_empirical_2004,hjort_extending_2009,qin_empirical_1994}. In 
principle all of those applications can be revisited using the RWP
function and its ramifications.

\smallskip

We now provide a more precise description of our contributions:

\textbf{A)} We explain how, by judiciously choosing $\mathcal{D}%
_{c}\left( \cdot\right) $, we can define a family of regularized
regression estimators (See Section \ref{Sec-ML-RWPI}). In particular,
we show how square-root LASSO (see Theorem \ref{Cor-lp-good-equiv}),
regularized logistic regression and support vector machines (see
Theorem \ref{Cor-Classification-Equivalences}) arise as particular
cases of suitable DRO formulations.

\bigskip

\textbf{B)} We derive general limit theorems for the asymptotic
distribution (as the sample size increases) of the RWP function
defined for general estimating equations. These limit theorems,
derived in Section \ref{Sec_Quad_WPF}, allow one to employ RWPI to
perform inference and choose the radius of uncertainty $\delta$ in
settings that are more general than linear/logistic regression.

\bigskip

\textbf{C) }We use our results from \textbf{B)} to obtain
prescriptions for regularization parameters in square-root LASSO\ and
regularized logistic regression settings (see Section
\ref{Sec-ML-RWP}). We also illustrate how coverage results for the
optimal risk, that demonstrate $O(n^{-1/2})$ rate of convergence, are
obtained immediately as a consequence of choosing $\delta \geq
R_n(\beta_\ast).$ 


\bigskip

\textbf{D) } We analyze our regularization selection in the
high-dimensional setting for square-root LASSO. Under standard
regularity conditions, we show (see Theorem
\ref{Thm-RWP-UB-d-Growing}) that the regularization parameter
$\lambda$ might be chosen as, 
\[
  \lambda= \frac{\pi}{\pi- 2} \frac{\Phi^{-1}\left(
      1-\alpha/2d\right) }%
  {\sqrt{n}},
\]
where $\Phi\left( \cdot\right) $ is the cumulative distribution of the
standard normal random variable and $1-\alpha$ is a user-specified
confidence level. The behavior of $\lambda$ as a function of $n$ and
$d$ is consistent with regularization selections studied in the
literature motivated by different considerations (see Section
\ref{Sec-Nonasymp-Gen-Lasso} for further details).

\bigskip

\textbf{E) } We analyze the empirical performance of RWPI based
selection of regularization parameter in the context of square-root
LASSO. In Section \ref{Sec-Num-Eg}, we compare the performance of RWPI
based optimal regularization with that of cross-validation based
approach on both simulated and real data. We conclude that RWPI based
approach yields a similar performance, without having to repeat the
algorithm over various choices of regularization parameters (as done
in cross-validation).  \smallskip

We now provide a discussion on topics which are related to RWPI.

\subsection{On related literature in Robust Optimization,
  Distributionally Robust Optimization and Optimal Transport}
Connections between robust optimization and regularization procedures
such as LASSO and Support Vector Machines have been studied in the
literature, see
\cite{xu_robust_2009,xu_robustness_2009,BERTSIMAS2018931}. The methods
proposed here differ subtly: While the papers
\cite{xu_robust_2009,xu_robustness_2009} add deterministic
perturbations of a certain size to the predictor vectors $X$ to
quantify uncertainty, the Distributionally Robust Representations that
we derive measure perturbations in terms of deviations from the
empirical distribution. While this change may appear cosmetic, it
brings a significant advantage: measuring deviations from the
empirical distribution, as we shall see, allows us to derive suitable
limit laws (or) probabilistic inequalities that can be used to give a
systematic prescription for the radius of uncertainty, $\delta$, in
the definition of the uncertainty region
$\mathcal{U}_{\delta}(\Pr_{n}) =\{\Pr:
\mathcal{D}_{c}(\Pr,\Pr_{n})\leq\delta\}$.

It is well-understood that as the number of samples $n$ increase, the
expected deviation of the empirical distribution from the true
distribution decays to zero, as a function of $n,$ at a specific
rate. To begin with, as a direct approach towards choosing the size of
uncertainty $\delta$, one can perhaps use a suitable concentration
inequality that measures such rate of convergence in terms of
Wasserstein distances (see, for example, \cite{fournier_rate_2015},
and references therein).  Such a simple specification of the size of
uncertainty, suitably as a function of $n$, does not arise naturally
in the deterministic robust optimization approaches in
\cite{xu_robust_2009,xu_robustness_2009}. 

For an application of these concentration inequalities to choose the
size of uncertainty set in the context of distributionally robust
logistic regression and data-driven DRO, refer
\cite{shafieezadeh-abadeh_distributionally_2015,esfahani_data-driven_2015-1}. The
exact representation for regularized logistic regression we derive
later in Section \ref{Sec-ML-DRR} can be can be seen as an extension,
in which the approximate representation described in \cite[Remark
1]{shafieezadeh-abadeh_distributionally_2015} is made to coincide
exactly with the regularized logistic regression estimator that has
been widely used in practice. It is important to note that, despite
imposing severe tail assumptions, the concentration inequalities used
to choose the radius of uncertainty set in
\cite{shafieezadeh-abadeh_distributionally_2015,esfahani_data-driven_2015-1}
dictate the size of uncertainty to decay at the rate $O(n^{-1/d})$;
unfortunately, this prescription scales non-gracefully as the number
of dimensions $d$ increase and the resulting coverage guarantees
suffer from a poor rate of convergence (see, for example,
\cite[Theorem
2]{shafieezadeh-abadeh_distributionally_2015},\cite[Theorem
3.5]{esfahani_data-driven_2015-1}).  Since most of the modern learning
and decision problems have huge number of covariates, application of
such concentration inequalities with poor rate of decay with
dimensions may not be most suitable for applications.

In contrast to directly using concentration inequalities, as we shall
see, the prescription obtained via RWPI typically has a rate of
convergence of order $O\left( n^{-1/2}\right) $ as
$n\rightarrow\infty$ (for fixed $d$). In particular, as we discuss in
the case of LASSO, according to our results corresponding to
contribution \textbf{E)}, RWPI based prescription of the size of
uncertainty actually can be shown (under suitable regularity
conditions) to decay at rate $O(\sqrt{\log d/n})$ (uniformly over $d$
and $n$ such that $\log^2d \ll n$), which is in agreement with the
findings of high-dimensional statistics literature (see
\cite{candes2007, negahban_unified_2012,belloni_square-root_2011} and
references therein). A profile function based approach towards
calibrating the radius of uncertainty in the context of empirical
likelihood based DRO can be found in
\cite{lam_empirical_2016,duchi2016statistics,gotoh2017calibration,lam2016recovering}.

Although we have focused our discussion on the context of regularized
estimators, our results are directly applicable to the area of
data-driven Distributionally Robust Optimization whenever the
uncertainty sets are defined in terms of a Wasserstein distance or,
more generally, an optimal transport metric. In particular, consider a
distributionally robust formulation of the form
\[
\min_{\theta:G\left(  \theta\right)  \leq0}\ \ \max_{\Pr:\ \mathcal{D}%
_{c}(\Pr,\Pr_{n})\leq\delta}\ \ \E_{\Pr}\left[  H(W,\theta)\right]  ,
\]
for a random element $W$ and a convex function $H(W,\cdot)$ defined
over a convex region $\{\theta:G\left( \theta\right) \leq0\}$
(assuming $G:\mathbb{R}^{d}\rightarrow\mathbb{R}$ convex). Here
$\Pr_{n}$ is the empirical measure of the sample
$\{W_{1},...,W_{n}\}$. One can then follow a reasoning parallel to
what we advocate throughout our LASSO\ discussion.  Argue, by applying
the corresponding KKT (Karush-Kuhn-Tucker)\ conditions, if possible,
that an optimal solution $\theta_{\ast}$ to the problem
\[
\min_{\theta:G\left(  \theta\right)  \leq0}\E_{\Pr_{true}}[H\left(
W,\theta\right)  ]
\]
satisfies a system of estimating equations of the form
$\E_{\Pr_{true}}[h\left( W,\theta_{\ast}\right) ]=0,$
for a suitable $h\left( \cdot\right) $ (where $\Pr_{true}$ is the weak
limit of the empirical measure $\Pr_{n}$ as $n\rightarrow\infty$). Then,
given a confidence level $1-\alpha$, one should choose $\delta$ as the
$(1-\alpha)$ quantile of the RWP function, 
\[
R_{n}(  \theta_{\ast})  =\inf\{\mathcal{D}_{c}(\Pr,\Pr_{n}%
):\E_{\Pr}[h\left(  W,\theta_{\ast}\right)  ]= \mathbf{0}\}.
\]
The results in Section 2 can then be used directly to approximate the
$(1-\alpha)$-quantile of $R_{n}\left( \theta_{\ast}\right) $. Just as
we explain in our discussion of the square-root LASSO example, the
selection of $\delta$ is the smallest possible choice for which
$\theta_{\ast}$ is plausible with $(1-\alpha)$ confidence.


\subsection{Connections to related inference literature}
We next discuss the connections between RWPI and EL. In EL one builds
a Profile Likelihood for an estimating equation. For instance, in the
context of EL applied to estimating $\beta$ satisfying
(\ref{EST_EQN}), one would build a Profile Likelihood Function in
which the optimization object is defined as the likelihood (or the
log-likelihood) between a given distribution $\Pr$ with respect to
$\Pr_{n}$. Therefore, the analogue of the uncertainty set
$\{\Pr:\mathcal{D}_{c}(\Pr,\Pr_{n})\leq\delta\}$, in the context of EL, will
typically contain distributions whose support coincides with that of
$\Pr_{n}$.  In contrast, the definition of the RWP function does not
require the likelihood between an alternative plausible model $\Pr$, and
the empirical distribution, $\Pr_{n}$, to exist. Owing to this
flexibility, for example, we are able to establish the connection
between regularization estimators and a suitable profile function.

There are other potential benefits of using a profile function which does not
restrict the support of alternative plausible models. For example, it has been
observed in the literature that in some settings EL\ might exhibit low
coverage \cite{owen_empirical_2001,chen_smoothed_1993,wu_weighted_2004}. It is
not the goal of this paper to examine the coverage properties of RWPI
systematically, but it is conceivable that relaxing the support of alternative
plausible models, as RWPI does, can translate into desirable coverage properties.

From a technical standpoint, the definition of the Profile Function in
EL gives rise to a finite dimensional optimization problem. Moreover,
there is a substantial amount of smoothness in the optimization
problems defining the EL Profile Function. This smoothness can be
leveraged in order to obtain the asymptotic distribution of the
Profile Function as the sample size increases. In contrast, the
optimization problem underlying the definition of RWP function\ in
RWPI is an infinite dimensional linear program. Therefore, the
mathematical techniques required to analyze the associated RWP
function are different (more involved) than the ones which are
commonly used in the EL setting.

A significant advantage of EL, however, is that the limiting
distribution of the associated Profile Function is typically
chi-squared. Moreover, this distribution is self-normalized in the
sense that no parameters need to be estimated from the
data. Unfortunately, this is typically not the case in using RWPI. In
many settings, however, the parameters of the distribution can be
easily estimated from the data itself.

Another methodology, strongly related to RWPI, has been studied
recently by the name of SOS (Sample-Out-of-Sample) inference
\cite{blanchet_sample_2016}. A suitable RWP function is built in this
setting as well, but the support of alternative plausible models is
assumed to be finite (but not necessarily equal to that of
$\Pr_{n}$). Instead, the support of alternative plausible models is
assumed to be generated not only by the available data, but additional
samples from independent distributions (defined by the user). The
limit results obtained for the RWP function in the context of SOS\ are
different from those obtained in this paper. For example, in the SOS
setting, the rates of convergence are dimension-dependent, which is
not the case in the RWPI. As explained in
\cite{blanchet_sample_2016,blanchet2017semi}, SOS inference is natural
in applications such as semi-supervised learning, in which massive
amounts of unlabeled data inform the support of the covariates. 

\subsection{Organization of the paper}
The rest of the paper is organized as follows. Section
\ref{Sec-ML-RWPI} corresponds to contribution \textbf{A)}; we first
introduce Wasserstein distances and then discuss distributionally
robust representations of popular machine learning algorithms. Section
\ref{Sec_WPF} deals with contribution\textbf{\ B)}; we discuss the RWP
function as an inference tool in a way which is parallel to the
Profile Likelihood in EL, and derive the asymptotic distribution of
the RWP function for general estimating equations.  Section
\ref{Sec-ML-RWP} discusses contributions \textbf{C)}, namely the
application of the results from \textbf{B)} for optimal
regularization. Our high-dimensional analysis of the RWP function in
the case of square-root LASSO\ is also presented in Section
\ref{Sec-ML-RWP}.  Numerical experiments using both simulated and real
data sets are given in Section \ref{Sec-Num-Eg}.  Proofs of all the
results are presented in the supplementary material \cite{supp_RWPI}
made available at the end of this article.

\section{Optimal Transport Definitions and DRO Representations of Machine
Learning Estimators \label{Sec-ML-RWPI}}
We begin with definitions of optimal transport costs and Wasserstein
distances. 


\subsection{Optimal Transport Costs and Wasserstein
Distances\label{Sec-OT-Wass-Defn}}

Let $c:\mathbb{R}^{m}\times\mathbb{R}^{m}\rightarrow\lbrack0,\infty]$
be any lower semi-continuous function such that $c(u,u)=0$ for every
$u \in \mathbb{R}^m.$ Given two probability distributions
$\Pr(\cdot) $ and $\mathbb{Q}(\cdot) $ supported on $\mathbb{R}^{m},$
the optimal transport cost or discrepancy between $\Pr$ and
$\mathbb{Q},$ denoted by $\mathcal{D}_{c}(\Pr,\mathbb{Q})$, is defined
as,
\begin{equation}
  \mathcal{D}_{c}(\Pr,\mathbb{Q})  =\inf\big\{\E_{\pi}\left[  c (
    U,W)  \right]  :\pi\in\mathcal{P}\left(  \mathbb{R}^{m} \times
    \mathbb{R}^{m}\right)  , \ \pi_{_{U}}=\Pr, \ \pi_{_{W}}= \mathbb{Q} \big\}.
\label{Discrepancy_Def}%
\end{equation}
Here, $\mathcal{P}\left( \mathbb{R}^{m} \times\mathbb{R}^{m} \right) $
is the set of joint probability distributions $\pi$ of $(U,W)$
supported on $\mathbb{R}^{m} \times\mathbb{R}^{m}$ and
$\pi_{_{U}},\pi_{_{W}}$ denote the marginals of $U$ and $W,$
respectively, under the joint distribution $\pi.$ Intuitively, the
quantity $c(u, w)$ can be interpreted as the cost of transporting unit
mass from $u$ in $\mathbb{R}^{m}$ to another element $w$ in
$\mathbb{R}^{m}.$ Then the expectation $\E_{\pi}[c(U,W)]$ corresponds
to the expected transport cost associated with the joint distribution
$\pi.$ 

In addition to the stated assumptions on the cost function $c(\cdot),$
if $c^{1/\rho}$ satisfies the properties of a metric for any
$\rho > 1,$ then $\mathcal{D}_{c}^{1/\rho}(\Pr,\mathbb{Q}) $ defines a
metric between probability distributions (see
\cite{villani_optimal_2008} for a proof and other properties of
$D_{c}$). For example, if
$c( u,w) =\left\Vert u-w\right\Vert _{2}^{2}$, then $\rho=2$ yields
that $c( u,w) ^{1/2}=\left\Vert u-w\right\Vert _{2}$ is symmetric,
non-negative, lower semi-continuous and it satisfies the triangle
inequality. In that case,
\[
  \mathcal{D}_{c}^{1/2}(\Pr,\mathbb{Q}) =\inf\left\{
    \sqrt{\E_{\pi}\left[ \left\Vert U-W\right\Vert _{2}^{2}\right] }:\
    \pi\in\mathcal{P}\left( \mathbb{R}^{m} \times\mathbb{R}^{m}\right)
    ,\text{ }\pi_{U}=\Pr,\text{ }\pi _{W}=\mathbb{Q} \right\}
\]
coincides with the Wasserstein distance of order two. More generally,
if we choose $c^{1/\rho}\left( u,w\right) =\Vert u - w \Vert_q$ for
some $\rho, q\geq1,$ then $\mathcal{D}_{c}^{1/\rho}( \cdot) $ is the
known as the Wasserstein distance of order $\rho.$

Wasserstein distances metrize weak convergence of probability measures
under suitable moment assumptions and have received immense attention
in probability theory (see \cite{rachev1998mass, rachev1998massII,
  villani_optimal_2008} for a collection of classical
applications). In addition, earth-mover's distance, a particular
example of Wasserstein distances, has been of interest in image
processing (see \cite{Rubner2000, Solomon2014}). More recently,
optimal transport metrics and Wasserstein distances are being actively
investigated for its use in various machine learning applications (see
\cite{NIPS2015_5680, peyre2016gromov, NIPS2015_5679,
  srivastava2015wasp} and references therein for a growing list of new
applications).

Throughout this paper, we consider optimal transport costs
$\mathcal{D}_{c}(\cdot)$ for a judiciously chosen cost function
$c(\cdot) $ to result in formulations such as (\ref{AUX_DRO_1}). As we
shall see in Section \ref{Sec-ML-DRR}, it is useful to allow
$c(\cdot) $ to be lower semi-continuous and potentially be infinite in
some region. Thus our setting requires discrepancy choices which are
slightly more general than standard Wasserstein distances.

\subsection{DRO formulation using optimal transport costs}
A common theme in machine learning problems is to find the best fitting
parameter in a family of parameterized models that relate a vector of
predictor variables $X\in\mathbb{R}^{d}$ to a response $Y\in\mathbb{R}.$ In
this section, we shall focus on a useful class of such models, namely, linear
and logistic regression models. Associated with these models, we have a loss
function $l(X_{i},Y_{i};\beta)$ which evaluates the fit of regression
coefficient $\beta$ for the given data points $\{(X_{i},Y_{i}):i=1,\ldots
,n.\}$
Then, just as we explained in the case of square-root LASSO in the
Introduction, our first step will be to show that regularized linear
and logistic regression estimators admit a Distributionally Robust
Optimization (DRO) formulation of the form,
\begin{equation}
\inf_{\beta\in\mathbb{R}^{d}}\sup_{\Pr:\mathcal{D}_{c}\big(\Pr,\Pr_{n}%
\big)\leq\delta}\E_{\Pr}\left[  l\big(X,Y;\beta\big)\right]  .\label{RWPI-DRO}%
\end{equation}


In contrast to the empirical risk minimization that performs well only
on the training data, the DRO problem (\ref{RWPI-DRO}) aims to find an
optimizer $\beta$ that performs uniformly well over all probability
measures in the neighborhood that can be perceived as perturbations to
the empirical training data distribution. Hence the solution to
(\ref{RWPI-DRO}) is said to be \textquotedblleft distributionally
robust\textquotedblright, and can be expected to generalize
better. See \cite{xu_robust_2009,xu_robustness_2009} and
\cite{shafieezadeh-abadeh_distributionally_2015} for earlier works
that relate robustness and generalization.

Recasting regularized regression as a DRO problem of form
\eqref{RWPI-DRO} lets us view these regularized estimators under the
lens of distributional robustness. The regularized estimators that we
consider in this paper include the regularized logistic regression
estimators in Example \ref{Ex_RLR} below, the support vector machines
(see \cite{hastie_elements_2005}) and the family of $\ell_p-$norm
penalized linear regression estimators of the form,
\begin{align}
  \min_{\beta\in\mathbb{R}^{d}}\bigg\{\sqrt{\E_{\Pr_{n}}\left[  l\left(
  X,Y;\beta\right)  \right]  }+\lambda\left\Vert \beta\right\Vert _{p}\bigg\},\label{Lp-Pen-LR}\end{align}
for any $p\in\lbrack1,\infty).$ This collection  includes the square-root Lasso estimator
described in Example \ref{Ex_GLasso} as a special case where $p=1.$


\begin{example}
[Regularized Logistic Regression]\label{Ex_RLR} \textnormal{ Consider the
  context of binary classification in which case the training data is
  of the form $\{(X_{1},Y_{1}), \ldots,(X_{n},Y_{n})\}$, with
  $X_{i}\in\mathbb{R}^{d},$ response $Y_{i}\in\{-1,1\}$ and the model
  postulates that, 
\[
\log\left(  \frac{\Pr\left(  Y_{i}=1|X_{i}=x\right)  }{1-\Pr\left(  Y_{i}=1|X_{i}=x\right)  }\right)  =\beta_{\ast}^{T}x
\]
for some $\beta_{\ast}\in\mathbb{R}^{d}$. In this case, the log-exponential
loss function (or negative log-likelihood for binomial distribution) is
\[
l\left(  x,y;\beta\right)  =\log\left( 1+\exp(-y\cdot\beta^{T}x) \right) ,
\]
and one is interested in estimating $\beta_{\ast}$ by solving
\begin{equation}
  \min_{\beta\in\mathbb{R}^{d}} \bigg\{\E_{\Pr_{n}}\left[  l\left(  X,Y;\beta
    \right)  \right]  +\lambda\left\Vert \beta\right\Vert _{p}\bigg\},\label{Lp-Pen-LogR}\end{equation}
for $p\in\lbrack1,\infty).$ Refer \cite{hastie_elements_2005} for a
more detailed discussion on regularized logistic regression.}
\end{example}


\subsection{Dual form of the DRO formulation (\ref{RWPI-DRO})}
\label{Sec-Dual-Form-DRO} Though the DRO formulation (\ref{RWPI-DRO})
involves optimizing over uncountably many probability measures, recent
strong duality results for Wasserstein DRO (see, for example, Theorem
1 in \cite{blanchet_quantifying_2016}) ensures that the inner supremum
in (\ref{RWPI-DRO}) 
admits a reformulation which is a simple, univariate optimization
problem. Before stating the result, we recall that the definition of
discrepancy measure $\mathcal{D}_{c}$ (see (\ref{Discrepancy_Def}))
requires the specification of cost function
$c( (x,y),(x^{\prime},y^{\prime})) $ between any two
predictor-response pairs
$(x,y),(x^{\prime},y^{\prime})\in\mathbb{R}^{d+1}.$

\begin{proposition}
  \label{Prop-Duality} Let
  $c: \mathbb{R}^{d+1} \times \mathbb{R}^{d+1} \rightarrow [0,\infty]$
  be a lower semi-continuous cost function satisfying
  $c((x,y),(x^\prime,y^\prime)) = 0$ whenever
  $(x,y) = (x^\prime,y^\prime).$ For $\gamma\geq0$ and loss functions
  $l(x,y;\beta)$ that are upper semi-continuous in $(x,y)$ for each
  $\beta$, define
\begin{equation}
\label{Adjust-Infty}\phi_{\gamma}(X_{i},Y_{i};\beta):=\sup_{u\in\mathbb{R}%
^{d},\ v\in\mathbb{R} }\bigg\{l(u,v;\beta)-\gamma c\big((u,v), (X_{i}%
,Y_{i})\big)\bigg\}.
\end{equation}
Then
\[
\sup_{\Pr:\ \mathcal{D}_{c}(\Pr,\Pr_{n})\leq\delta}\E_{\Pr}\big[l(X,Y;\beta
)\big] =\min_{\gamma\geq0}\left\{  \gamma\delta+\frac{1}{n}\sum_{i=1}^{n}%
\phi_{\gamma}(X_{i},Y_{i};\beta)\right\}  .
\]
\end{proposition}

\noindent Consequently, the DR regression problem (\ref{RWPI-DRO}) reduces to
\begin{equation}
\inf_{\beta\in\mathbb{R}^{d}}\sup_{\Pr:\ \mathcal{D}_{c}(\Pr,\Pr_{n})\leq\delta
}\E_{\Pr}\big[l(X,Y;\beta)\big]=\inf_{\beta\in\mathbb{R}^{d}}\ \min_{\gamma\geq
0}\left\{  \gamma\delta+\frac{1}{n}\sum_{i=1}^{n}\phi_{\gamma}(X_{i}%
,Y_{i};\beta)\right\}  . \label{DR-Reg-Equiv}%
\end{equation}
Proposition \ref{Prop-Duality} follows as a straightforward
application of \cite[Theorem 1]{blanchet_quantifying_2016}. As we
shall see in Section \ref{Sec-DRR}, the function
$\phi_{\gamma}(\cdot)$ is explicitly computable for various examples
of interest. Of the reformulations in the literature for Wasserstein
distance based DRO (see \cite{esfahani_data-driven_2015,
  blanchet_quantifying_2016,gao_distributionally_2016}%
), 
the general cost structure assumed in \cite[Theorem
1]{blanchet_quantifying_2016} is essential for the exact recovery of
machine learning estimators that are presented below in Section
\ref{Sec-ML-DRR}.


\subsection{Distributionally Robust Representations\label{Sec-ML-DRR}}
\label{Sec-DRR}

\subsubsection{Example \ref{Ex_GLasso} (continued): Recovering regularized
estimators for linear regression. \label{Sec-Equiv-Regressions}}

We examine the right-hand side of (\ref{DR-Reg-Equiv}) for the square loss
function for the linear regression model $Y=\beta^{T}X+e$, and obtain the
following result without any further distributional assumptions on $X,Y$ and
the error $e.$ For brevity, let $\bar{\beta}=(-\beta,1),$ and recall the
definition of the discrepancy measure $\mathcal{D}_{c}$ in
(\ref{Discrepancy_Def}).

\begin{proposition}
[DR linear regression with square loss]\label{Thm-lp-Equiv}Fix $q\in
(1,\infty].$ Consider the square loss
function and second order discrepancy measure
$\mathcal{D}_{c}$ defined using
$\ell_{q}$-norm. In other words, take
$l(x,y;\beta)=(y-\beta^{T}x)^{2}$ {\ and }
$c\big((x,y),(u,v)\big)=\Vert (x,y)-(u,v)\Vert_{q}^{2}$. Then,
\begin{equation}
\inf_{\beta\in\mathbb{R}^{d}}\sup_{\Pr:\ \mathcal{D}_{c}(\Pr,\Pr_{n})\leq\delta
}\E_{\Pr} \big[l(X,Y;\beta)\big]=\inf_{\beta\in\mathbb{R}^{d}}\left(  \sqrt{
MSE_{n}(\beta)}+\sqrt{\delta}\ \Vert\bar{\beta}\Vert_{p}\right)  ^{2},
\label{lp-bad-Equiv}%
\end{equation}
where $MSE_{n}(\beta)=\E_{\Pr_{n}}[(Y-\beta^{T}X)^{2}]=\frac{1}{n} \sum_{i=1}%
^{n}(Y_{i}-\beta^{T}X_{i})^{2}$ is the mean square error for the coefficient
choice $\beta$ and $p$ is such that $1/p+1/q=1.$
\end{proposition}

As an important special case, we consider $q=\infty$ and identify the
following equivalence for DR regression applying discrepancy measure based on
neighborhoods defined using $\ell_{\infty}$ norm:
\[
\text{arg}\,\text{min}_{\beta\in\mathbb{R}^{d}}\sup_{\Pr:\ \mathcal{D}%
_{c}(\Pr,\Pr_{n})\leq\delta}\E_{\Pr} \big[l(X,Y;\beta)\big]=\text{arg}\,\text{min}_{\beta\in\mathbb{R}^{d}}\left\{  \sqrt{ MSE_{n}(\beta)}+\sqrt{\delta
}\ \Vert\bar{\beta}\Vert_{1}\right\}  .
\]

The right hand side of (\ref{lp-bad-Equiv}) resembles $\ell_{p}$-norm
regularized regression (except for the fact that we have
$\Vert\bar{\beta }\Vert_{p}$ instead of $\Vert\beta\Vert_{p}$). In
order to obtain the exact equivalence, we introduce a slight
modification to the norm $\Vert \cdot\Vert_{q}$ to be used as the cost
function, $c(\cdot)$, in defining $\mathcal{D}_{c}$. We define
\begin{align}
N_{q}\big((x,y),(u,v)\big)=
\begin{cases}
\Vert x-u\Vert_{q},\quad\quad & \text{ if }y=v,\\
\infty, & \text{ otherwise,}%
\end{cases} 
\label{Modified-Cost-Fn}%
\end{align}
in order to use $c(\cdot) = N_{q}(\cdot)$ as the transportation cost
instead of the standard $\ell_{q} $ norm
$\Vert(x,y)-(u,v)\Vert_{q}$. Subsequently, one can consider modified
cost functions of form
$c((x,y),(u,v))=(N_{q}((x,y),(u,v)))^{\rho}$. As this modified cost
function assigns infinite cost when $y\neq v$, the infimum in
(\ref{PROF_R}) is effectively over joint distributions that do not
alter the marginal distribution of $Y$. As a consequence, the
resulting neighborhood set $\{\Pr:\mathcal{D}_{c} (\Pr,\Pr_{n})\leq\delta\}$
admits distributional ambiguities only with respect to the predictor
variables $X$.

The following result is essentially the same as Proposition
\ref{Thm-lp-Equiv} except for the use of the modified cost $N_{q}$ and
the resulting norm regularization of form $\Vert\beta\Vert_{p}$
(instead of $\Vert\bar{\beta }\Vert_{p}$ as in Proposition
\ref{Thm-lp-Equiv}), thus exactly recovering the regularized
regression estimators in \eqref{Lp-Pen-LR}.

\begin{theorem}
  \label{Cor-lp-good-equiv}Consider the square loss,
  $l(x,y;\beta) = (y-\beta^Tx)^2,$ and discrepancy measure
  $\mathcal{D}_{c}(\Pr,\Pr_{n})$ defined as in (\ref{Discrepancy_Def})
  using the cost function $c((x,y),(u,v))=(N_{q}((x,y),(u,v)))^{\rho}$
  with $\rho = 2.$ Then,
\[
  \inf_{\beta\in\mathbb{R}^{d}}\sup_{\Pr:\
    \mathcal{D}_{c}(\Pr,\Pr_{n})\leq\delta
  }\E_{\Pr}\big[l(X,Y;\beta)\big]=\inf_{\beta\in\mathbb{R}^{d}}\left(
    \sqrt{MSE_{n}(\beta)}+\sqrt{\delta}\ \Vert\beta\Vert_{p}\right)
  ^{2},
\]
where $MSE_{n}(\beta)=\E_{\Pr_{n}}[(Y-\beta^{T}X)^{2}]=n^{-1}\sum_{i=1}^{n}%
(Y_{i}-\beta^{T}X_{i})^{2}$ is the mean square error for the coefficient
choice $\beta$ and $p$ is such that $1/p+1/q=1.$
\end{theorem}

\subsubsection{Example \ref{Ex_RLR} (continued): Recovering regularized
estimators for classification.}

Apart from exactly recovering norm regularized estimators for linear
regression,
the discrepancy measure $\mathcal{D}_{c}$ based on the modified norm
$N_{q}$ in (\ref{Modified-Cost-Fn}) is natural when our interest is in
learning problems where the responses $Y_{i}$ take values in a finite
set -- as in the binary classification problem where the response
variable $Y$ takes values in $\{-1,+1\}$. The following result allows
us to recover the DRO formulation behind the regularized logistic
regression estimators discussed in Example \ref{Ex_RLR} and as well
for the widely used support vector machines (see
\cite{hastie_elements_2005}).

\begin{theorem}
[Regularized regression for Classification]%
\label{Cor-Classification-Equivalences} Consider the discrepancy
measure $\mathcal{D}_{c}(\cdot)$ defined using the cost function
$c((x,y),(u,v))=N_{q} ((x,y),(u,v))^\rho$ with $\rho = 1.$ Then for
logistic regression with log-exponential loss function and Support
Vector Machine (SVM) with the hinge loss function, we have
\[
\inf_{\beta\in\mathbb{R}^{d}}\sup_{\Pr:\ \mathcal{D}_{c}(\Pr,\Pr_{n})\leq\delta
}\E_{\Pr} \big[\log(1+e^{-Y\beta^{T}X})\big]=\inf_{\beta\in\mathbb{R}^{d}}%
\frac{1}{ n}\sum_{i=1}^{n}\log\left(  1+e^{-Y_{i}\beta^{T}X_{i}}\right)
+\delta\left\Vert \beta\right\Vert _{p},
\]
and
\[
\inf_{\beta\in\mathbb{R}^{d}} \sup_{\Pr : \ D_{c}(\Pr, \Pr_{n}) \leq\delta} \E_{\Pr}
\big[ (1-Y\beta^{T}X)^{+}\big] = \frac{1}{n}\sum_{i=1}^{n}(1-Y_{i}\beta
^{T}X_{i})^{+}+ \ \delta\left\Vert \beta\right\Vert _{p},
\]
where $p$ is such that $1/p+1/q=1$.
\end{theorem}

\noindent The proofs of all of the results in this subsection are
provided in Appendix A.1 in the supplementary
material \cite{supp_RWPI}. The example of logistic regression with
Wasserstein distance based uncertainty sets has been considered in
\cite{shafieezadeh-abadeh_distributionally_2015}. The representation
for regularized logistic regression in Theorem
\ref{Cor-Classification-Equivalences} can be seen as an extension in
which the approximate representation described in \cite[Remark
1]{shafieezadeh-abadeh_distributionally_2015} is made to coincide
exactly with the regularized logistic regression estimator that has
been widely used in practice.  The approximate representation for
regularized logistic regression in
\cite{shafieezadeh-abadeh_distributionally_2015} is based on the
semi-infinite linear programming duality results due to
\cite{shapiro_duality_2001}. On the other hand, due to the presence of
infinite transportation costs in our DRO formulation that results in
the desired exact representation (see Theorem 2), we utilize a
different strong duality result, \cite[Theorem
1]{blanchet_quantifying_2016}, that is specifically derived for
Wasserstein DRO with general cost structures. In addition, other
equivalences described for square-root LASSO and Support Vector
Machines in terms of Wasserstein DRO, as far as we know, have been
reported for the first time in this paper. See
\cite{abadehregularization} for additional examples.


\section{The Robust Wasserstein Profile Function\label{Sec_WPF}}
Given an estimating equation $\E_{\Pr_{n}}[h(W,\theta)]=\mathbf{0},$ the
objective of this section is to study the asymptotic behavior of the
associated RWP function $R_{n}(\theta)$. As discussed in the
Introduction, this analysis is key in our approach towards
constructing the confidence region $\Lambda_n(\theta)$ and choosing
the radius of the uncertainty set optimally.

\subsection{The RWP Function for Estimating Equations and Its Use in
  Constructing Confidence Regions\label{Sec-RWP-Inference-Tool}}

The Robust Wasserstein Profile function's definition is inspired by the notion
of the Profile Likelihood function, introduced in the pioneering work of Art
Owen in the context of EL (see \cite{owen_empirical_2001}). We provide the
definition of the RWP function for estimating $\theta_{\ast}\in\mathbb{R}^{l}%
$, which we assume satisfies
\begin{equation}
\E\left[  h\left(  W,\theta_{\ast}\right)  \right]  =\mathbf{0},
\label{EST_EQ_A}%
\end{equation}
for a given random variable $W$ taking values in $\mathbb{R}^{m}$ and
an integrable function
$h:\mathbb{R}^{m}\times\mathbb{R}^{l}\rightarrow \mathbb{R}^{r}$. The
parameter $\theta_{\ast}$ is required to be unique to ensure
consistency, but uniqueness is not necessary for the limit theorems
that we shall state, unless we explicitly indicate so.

Given a set of samples $\{W_{1},...,W_{n}\}$, which are assumed to be
i.i.d.  copies of $W$, we define the Wasserstein Profile function for
the estimating equation (\ref{EST_EQ_A}) as,
\begin{equation}
  R_{n}(\theta)  :=\inf\big\{D_{c}(\Pr,\Pr_{n}):\ \E_{\Pr}\left[
    h(W,\theta)\right]  =\mathbf{0}\big\}.\label{RWP-Defn}%
\end{equation}
Here, recall that $\Pr_{n}$ denotes the empirical distribution associated with
the training samples $\{W_{1},\ldots,W_{n}\}$ and $c(\cdot)$ is a chosen cost
function. In this section, we are primarily concerned with cost functions of
the form,
\begin{equation}
  c( u,w)  =\left\Vert w-u\right\Vert _{q}^{\rho},\label{c_quad}%
\end{equation}
where $\rho \in [1,\infty)$ and $q \in (1,\infty].$ We remark,
however, that the methods presented here can be easily adapted to more
general cost functions. For simplicity, we assume that the samples
$\{W_{1},\ldots,W_{n}\}$ are distinct.

Since, as we shall see, the asymptotic behavior of the RWP function
$R_{n}(\theta)$ is dependent on the exponent $\rho$ in \eqref{c_quad},
we sometimes write $R_{n}\left( \theta;\rho\right) $ to make this
dependence explicit; but whenever the context is clear, we drop $\rho$
to avoid notational burden. Also, observe that the profile function
defined in (\ref{PROF_R}) for the linear regression example is
obtained as a particular case by selecting $W=\left( X,Y\right) $,
$\beta=\theta$ and defining $h( x,y,\theta) =(y-\theta^{T}x)x$.

Our goal in this section is to develop an asymptotic analysis of the RWP
function which parallels that of the theory of EL. In particular, we shall
establish,
\begin{equation}
n^{\rho/2}R_{n}\left(  \theta_{\ast};\rho\right)  \Rightarrow\bar{R}\left(
\rho\right), \label{R_2}%
\end{equation}
for a suitably defined random variable $\bar{R}( \rho).$ Throughout
this paper, the symbol ``$\Rightarrow$'' is used to denote convergence
in distribution.

As the empirical distribution weakly converges to the underlying
probability distribution from which the samples are obtained, it
follows from the definition of RWP function in \eqref{R_2} that
$R_{n}(\theta;\rho )\rightarrow0,$ as $n\rightarrow\infty,$ if and
only if $\theta$ satisfies $\E[h(W,\theta)]=\mathbf{0}$; for every
other $\theta,$ we have that
$n^{\rho/2}R_{n}(\theta;\rho)\rightarrow\infty.$ Therefore, the result
in (\ref{R_2}) can be used to provide confidence regions around
$\theta_{\ast}$ as follows: Given a confidence level $1-\alpha$ in
(0,1), if we denote $\eta_{\alpha}$ as the $(1-\alpha)-$quantile of
$\bar {R}(\rho) $, that is,
$\Pr( \bar{R}( \rho) \leq\eta_{\alpha}) =1-\alpha,$ then the set,
\[
  \bar{\Lambda}_{n}\left( n^{-\rho/2}\eta_{\alpha}\right) =\left\{
    \theta:R_{n}(\theta;\rho) \leq n^{-\rho/2}\eta_{\alpha}\right\}
\]
is an approximate $(1-\alpha)-$confidence region for
$\theta_{\ast}$. This is because, by definition of
$\bar{\Lambda}_{n}\left( \cdot \right) $,
\begin{align}
  \Pr\left( \theta_{\ast}\in\bar{\Lambda}_{n}\left(
      n^{-\rho/2}\eta_{\alpha}\right) \right) =\Pr\left(
    n^{\rho/2}R_{n}\left( \theta_{\ast};\rho\right) \leq
    \eta_{\alpha}\right) \approx \Pr\left( \bar{R}\left( \rho\right)
    \leq \eta_{\alpha}\right) =1-\alpha.
\label{Conf-Region}
\end{align}

Throughout the development in this section, the dimension $m$ of the
random vector $W$ is kept fixed and the sample size $n$ is sent to
infinity; the function $h\left( \cdot\right) $ can be quite
general. 

\subsection{The dual formulation of RWP function\label{Sec-Dual-RWP}}

The first step in the analysis of the RWP function $R_{n}(\theta)$ is to use
the definition of the discrepancy measure $D_{c}$ to rewrite $R_{n}(\theta)$
as,
\begin{align*}
  R_{n}(\theta)  &  = \inf\big\{\E_{\pi}\left[  c(U,W)\right]  : \pi
                   \in\mathcal{P}\left(  \mathbb{R}^{m} \times\mathbb{R}^{m} \right)  ,\text{\ }
                   \E_{\pi}\left[  h(U,\theta)  \right]  = \mathbf{0}, \text{ }%
                   \pi_{_{W}}=\Pr_{n} \big\},
\end{align*}
which is a \textit{problem of moments} of the form,
\begin{align}
  R_{n}(\theta)  = \vspace{-25pt} \inf_{\pi\in\mathcal{P}(\mathbb{R}^{m} \times
  \mathbb{R}^{m})} \left\{  \E_{\pi} \left[  c(U,W)\right]: \ \E_{\pi}\left[
  h\left(  U,\theta\right)  \right]  = \mathbf{0}, \ \E_{\pi}\left[  \mathbb{I}(W
  = W_{i})\right]  = \frac{1}{n},\  i  \leq n\right\}. 
\label{RWP-LP}
\end{align}

The problem of moments is a classical linear programming problem for
which the respective dual formulation and strong duality have been
well-studied (see, for example,
\cite{Isii1962,doi:10.1287/opre.43.5.807}). The linear program problem
over the variable $\pi$ in \eqref{RWP-LP} admits a simple dual
semi-infinite linear program of form,
\begin{align*}
&  \sup_{a_{i} \in\mathbb{R}, \lambda\in\mathbb{R}^{r}} \left\{  a_{0} +
\frac{1}{n}\sum_{i=1}^{n} a_{i} : \ a_{0} + \sum_{i=1}^{n} a_{i}
\mathbf{1}_{\{w = W_{i}\}}(u,w) + \lambda^{T}h(u,\theta) \leq c(u,w),\ 
  \forall u,w \in\mathbb{R}^{m}\right\} \\ 
&  \quad\quad\quad\quad=\sup_{\lambda\in\mathbb{R}^{r}} \left\{  \frac{1}{n}
\sum_{i=1}^{n} \inf_{u \in\mathbb{R}^{m}} \left\{  c(u,W_{i}) - \lambda^{T}
h(u,\theta) \right\}  \right\} \\
&  \quad\quad\quad\quad=\sup_{\lambda\in\mathbb{R}^{r}} \left\{  - \frac{1}{n}
\sum_{i=1}^{n} \sup_{u \in\mathbb{R}^{m}} \left\{  \lambda^{T} h(u,\theta) -
c(u,W_{i}) \right\}  \right\}  .
\end{align*}
Proposition \ref{Prop-RWP-Duality} below states that strong duality holds
under mild assumptions, and the dual formulation above indeed equals
$R_{n}(\theta).$

\begin{proposition}
\label{Prop-RWP-Duality} Let $h(\cdot,\theta)$ be Borel measurable, and
$\Omega= \{ (u,w) \in\mathbb{R}^{m} \times\mathbb{R}^{m}: c(u,w) < \infty\}$
be Borel measurable and non-empty. Further, suppose that $\mathbf{0}$ lies in
the interior of the convex hull of $\{ h(u, \theta): u \in\mathbb{R}^{m}\}. $
Then,
\begin{align*}
R_{n}(\theta) = \sup_{\lambda\in\mathbb{R}^{r}} \left\{  - \frac{1}{n}
\sum_{i=1}^{n} \sup_{u \in\mathbb{R}^{m}} \left\{  \lambda^{T} h(u,\theta) -
c(u,W_{i}) \right\}  \right\}  .
\end{align*}

\end{proposition}

A proof of Proposition \ref{Prop-RWP-Duality}, along with an
introduction to the problem of moments, is provided in Appendix B in
the supplementary material \cite{supp_RWPI}.

\subsection{Asymptotic Distribution of the RWP Function\label{Sec_Quad_WPF}}

In order to gain intuition behind (\ref{R_2}), let us first consider the
simple example of estimating the expectation $\theta_{\ast}= \E[W]$ of a
real-valued random variable $W$, using $h\left(  w,\theta\right)  =w-\theta.$

\begin{example}
  \textnormal{Let $h\left( w,\theta\right) =w-\theta$ with $m=1=l=r.$
    First, suppose that the choice of cost function is
    $c( u,w) =\left| u-w\right| ^{\rho}$ for some $\rho>1$. As long as
    $\theta$ lies in the interior of convex hull of support of $W,$
    Proposition \eqref{Prop-RWP-Duality} implies,
\begin{align*}
R_{n}(\theta;\rho)  & = \sup_{\lambda\in\mathbb{R}} \left\{  - \frac{1}{n}
\sum_{i=1}^{n} \sup_{u \in\mathbb{R}} \big\{ \lambda(u-\theta)- \left|  W_{i}
- u\right| ^{\rho} \big\}\right\} \\
& = \sup_{\lambda\in\mathbb{R}}\left\{  -\frac{\lambda}{n} \sum_{i=1}^{n}(W_{i}-\theta)-\frac{1}{n}\sum_{i=1}^{n}\sup_{u \in\mathbb{R}}
\big\{\lambda\left(  u - W_{i}\right)  -\left|  W_{i}-u\right|  ^{\rho}
\big\}\right\} .
\end{align*}
As $\max_{\Delta}\{\lambda\Delta- |\Delta|^{\rho}\} = (\rho-1)|\lambda
/\rho|^{\rho/(\rho-1)},$ we obtain
\begin{align}
R_{n}\left(  \theta;\rho\right)   & =\sup_{\lambda}\left\{  -\frac{\lambda}{n}\sum_{i=1}^{n}(W_{i}-\theta)- (\rho-1)\left|  \frac{\lambda}{\rho}\right|
^{\frac{\rho}{\rho-1}} \right\} \nonumber\\
& = \left\vert \frac{1}{n}\sum_{i=1}^{n}\left(  W_{i}-\theta\right)
\right\vert ^{\rho}.\nonumber
\end{align}
Then, under the hypothesis that $\E\left[  W \right]  = \theta_{\ast},$ and
assuming $\text{Var}[W] = \sigma_{_{W}}^{2} <\infty$, we obtain,
\[
n^{\rho/2}R_{n}\left(  \theta_{\ast};\rho\right)  \Rightarrow\bar{R} \left(
\rho\right)  \sim\sigma_{_{W}}^{\rho}\left\vert N\left(  0,1\right)
\right\vert ^{\rho},
\]
where $N\left(  0,1\right)  $ denotes a standard Gaussian random variable. The
limiting distribution for the case $\rho= 1$ can be formally obtained by
setting $\rho= 1$ in the above expression for $\bar{R}(\rho)$, but the
analysis is slightly different. When $\rho= 1,$
\begin{align*}
R_{n}\left(  \theta\right)   & =\sup_{\lambda\in\mathbb{R}}\left\{
-\frac{\lambda}{n} \sum_{i=1}^{n}(W_{i}-\theta)-\frac{1}{n}\sum_{i=1}^{n}\sup_{u \in\mathbb{R}} \big\{\lambda\left(  u - W_{i}\right)  -\left\vert u -
W_{i}\right\vert \big\}\right\} \\
& =\sup_{\lambda}\left\{  -\frac{\lambda}{n}\sum_{i=1}^{n}(W_{i}-\theta
)-\sup_{\Delta\in\mathbb{R}}\big\{\lambda\Delta-\left\vert \Delta\right\vert
\big\}\right\} .
\end{align*}
Following the notion that $\infty\times0 = 0,$
\begin{align*}
R_{n}(\theta)  & =\sup_{\lambda}\left\{  \frac{\lambda}{n}\sum_{i=1}^{n}(W_{i}-\theta)-\infty I\left(  \left\vert \lambda\right\vert >1\right)
\right\} \\
& =\max_{\left\vert \lambda\right\vert \leq1}\frac{\lambda}{n}\sum_{i=1}^{n}(W_{i}-\theta)=\left\vert \frac{1}{n}\sum_{i=1}^{n}(W_{i}-\theta
)\right\vert .
\end{align*}
So, indeed if $\E[W] =\theta_{\ast}$ and $Var\left[  W\right]  = \sigma_{_{W}}^{2} <\infty$, we obtain
\[
n^{1/2}R_{n}\left(  \theta_{\ast}\right)  \Rightarrow\sigma_{_{W}}\left\vert
N\left(  0,1\right)  \right\vert .
\]
} \label{Eg-Mean}
\end{example}
We now discuss far reaching extensions to the developments in Example
\ref{Eg-Mean} by considering estimating equations that are more general.
First, we state a general asymptotic stochastic upper bound, which we believe
is the most important result from an applied standpoint as it captures the
speed of convergence of $R_{n}(\theta_{\ast})$ to zero. Following this, we
obtain an asymptotic stochastic lower bound that matches with the upper bound
(and therefore the weak limit) under mild, additional regularity conditions.
We discuss the nature of these additional regularity conditions, and also why
the lower bound in the case $\rho=1$ can be obtained basically without
additional regularity.


For the asymptotic upper bound we shall impose the following assumptions.

\textbf{Assumptions:}

\textbf{A1)} Assume that $c( u,w) =\Vert u-w\Vert_{q}^{\rho}$ for
$q\geq1$ and $\rho\geq1$. For a chosen $q \geq1,$ let $p$ be such that
$1/p + 1/q = 1.$


\textbf{A2) }Suppose that $\theta_{\ast}\in\mathbb{R}^{l}$ satisfies
$\mathbb{E}\left[  h(W,\theta_{\ast})\right]  =\mathbf{0}$ and $\mathbb{E}
\left\Vert h(W,\theta_{\ast})\right\Vert _{2}^{2} <\infty$. (While we do not
assume that $\theta_{\ast}$ is unique, the results are stated for a fixed
$\theta_{\ast}$ satisfying $\E[h(W,\theta_{\ast})] = \mathbf{0}$.)


\textbf{A3) }Suppose that the function $h(\cdot,\theta_{\ast})$ is
continuously differentiable with derivative $D_{w}h(\cdot,\theta_{\ast})$.


\textbf{A4)} Suppose that for each $\zeta\neq0$,
\begin{equation}
\Pr\left(  \left\Vert \zeta^{T}{D}_{w}h\left(  W,\theta_{\ast}\right)
\right\Vert _{p}>0\right)  >0. \label{U_T1}%
\end{equation}

Assumptions A1) - A3) make precise the setting considered. Assumption
A4) is the only assumption which is technical in nature and it can be
equivalently stated as
\[\E[D_wh(W,\theta_\ast)D_wh(W,\theta_\ast)^T] \succ 0, \] where
$A \succ 0$ is used to denote that the matrix $A$ is positive
definite. Verification of this positive definitenes condition for the
linear and logistic regression problems is executed, respectively, in
Sections \ref{Sec-RWPF-Lin-Reg} and \ref{Sec-RWPF-Log-Reg}. In order
to state the theorem, let us introduce the notation for asymptotic
stochastic upper bound,
\[
n^{\rho/2}R_{n}(\theta_{\ast};\rho)\lesssim_{D}\bar{R}\left(  \rho\right)  ,
\]
which expresses that, for every continuous and bounded non-decreasing
function $f\left( \cdot\right),$ we have,
\[
\overline{\lim}_{n\rightarrow\infty}\E\left[  f\left(  n^{\rho/2}R_{n}%
(\theta_{\ast};\rho)\right)  \right]  \leq \E \left[  f\left(  \bar{R}\left(
\rho\right)  \right)  \right]  .
\]
Similarly, we write $\gtrsim_{D}$ for an asymptotic stochastic lower bound,
namely
\[
\underline{\lim}_{n\rightarrow\infty}\E \left[  f\left(  n^{\rho/2}R_{n}%
(\theta_{\ast};\rho)\right)  \right]  \geq \E \left[  f\left(  \bar{R}\left(
\rho\right)  \right)  \right]  .
\]
Therefore, if both stochastic upper and lower bounds hold, then $n^{\rho
/2}R_{n}(\theta_{\ast};\rho)\Rightarrow\bar{R}\left(  \rho\right)  $ as $n
\rightarrow\infty.$ (see, for example, \cite{billingsley2013convergence}). Now
we are ready to state our asymptotic upper bound.

\begin{theorem}
\label{Thm-WPF_RWPI} Under Assumptions A1) to A4) we have, as $n\rightarrow
\infty$,
\[
n^{\rho/2}R_{n}(\theta_{\ast};\rho)\lesssim_{D}\bar{R}\left(  \rho\right)  ,
\]
where, for $\rho>1$,
\[
\bar{R}\left(  \rho\right)  :=\max_{\zeta\in\mathbb{R}^{r}}\left\{  \rho
\zeta^{T}H - (\rho-1) \E\left\Vert \zeta^{T}D_{w}h\left(  W, \theta_{\ast
}\right)  \right\Vert _{p}^{\rho/(\rho-1)} \right\}  ,
\]
and if $\rho=1$,
\[
\bar{R}\left(  1\right)  :=\max_{\zeta:\Pr\left(  \left\Vert \zeta^{T}{D}
_{w}h\left(  W,\theta_{\ast}\right)  \right\Vert _{p}>1\right)  =0}
\{\zeta^{T}H\}.
\]
In both cases $H\sim\mathcal{N}(\mathbf{0}, \text{Cov}[h(W,\theta_{\ast})])$,
and $\text{Cov}[h(W,\theta_{\ast})]=\E\left[  h(W,\theta_{\ast})h(W,\theta
_{\ast})^{T}\right]  $.
\end{theorem}

We remark that as $\rho\rightarrow1$, one can verify that
$\bar{R}( \rho) \Rightarrow\bar{R}\left( 1\right) $, so formally one
can simply keep in mind the expression $\bar{R}\left( \rho\right) $
with $\rho>1$. It is interesting to note that $\bar{R}(\rho)$
resembles Fenchel transform when viewed as a function of $H$. Indeed,
in the case where $p= q = \rho =2$ and $\E[D_wh(W,\theta_\ast)]$ is
invertible, the expression for $\bar{R}(\rho)$ simplifies as follows:
\begin{align}
  \bar{R}(\rho) = \max_{\zeta \in \mathbb{R}^r} \left\{ 2\zeta^TH -
  \zeta^T \E\left[ D_wh(W, \theta_\ast) \right] \zeta \right\} =
  H^T \left(\E\left[D_wh(W,\theta_\ast)\right]\right) ^{-1} H. 
\label{simplified-rho-p-eq-2}
\end{align}

We now study some sufficient conditions which
guarantee that $\bar{R }\left( \rho\right) $ is also an asymptotic
lower bound for $n^{\rho/2}R_{n}(\theta_{\ast};\rho)$. We consider the
case $\rho=1$ first, which will be used in applications to logistic
regression discussed later in the paper.

\begin{proposition}
\label{Prop_SLB_rho_1}In addition to assuming A1) to A4), suppose that $W$ has
a positive density (almost everywhere) with respect to the Lebesgue measure.
Then,
\[
  n^{1/2}R_{n}(\theta_{\ast};1) \Rightarrow \bar{R}(1) .
\]

\end{proposition}

The following set of assumptions can be used to obtain tight
asymptotic stochastic lower bounds when $\rho>1;$ the corresponding
result will be applied to the context of square-root LASSO.


\textbf{A5) }(\textit{Growth condition}) Assume that there exists $\kappa
\in\left(  0,\infty\right)  $ such that for $\left\Vert w\right\Vert _{q}%
\geq1$,%
\begin{equation}
\left\Vert D_{w}h(w,\theta_{\ast})\right\Vert _{p}\leq\kappa\left\Vert
w\right\Vert _{q}^{\rho-1}, \label{U_T0}%
\end{equation}
and that $\E\left\Vert W_{i}\right\Vert ^{\rho}<\infty$.


\textbf{A6)} (\textit{Locally Lipschitz continuity}) Assume that there exists
exists $\bar{\kappa}: \mathbb{R}^{m}\rightarrow[0,\infty)$ such that,
\[
\left\Vert D_{w}h(w+\Delta,\theta_{\ast})-D_{w}h(w,\theta_{\ast})\right\Vert
_{p}\leq\bar{\kappa}\left(  W_{i}\right)  \left\Vert \Delta\right\Vert _{q},
\]
for $\left\Vert \Delta\right\Vert _{q}\leq1$, and
$ \E \left[ \bar{\kappa}\left( w\right) ^{c} \right] <\infty, $ for
$c \leq\max\{2, \frac{\rho}{\rho-1}\}.$ \bigskip

We now summarize our last weak convergence result of this section.

\begin{proposition}
\label{Prop_SLB_rho_L_1}If Assumptions A1) to A6) hold and $\rho>1$,
then
\[
n^{\rho/2}R_{n}(\theta_{\ast};\rho)\Rightarrow\bar{R}\left(  \rho\right)  .
\]

\end{proposition}

\bigskip

Before we move on with the applications of the previous results, it is worth
discussing the nature of the additional assumptions introduced to ensure that
an asymptotic lower bound can be obtained which matches the upper bound in
Theorem \ref{Thm-WPF_RWPI}.

As we shall see in the technical development in Appendix A.3 (see
supplementary material \cite{supp_RWPI}) where the proofs of the above
results are furnished, the dual formulation of RWP function in
Proposition \ref{Prop-RWP-Duality} can be re-expressed, assuming only
A1) to A4), as,
\begin{equation}
n^{\rho/2}R_{n}\left(  \theta_{\ast};\rho\right)  =\sup_{\zeta}\left\{
\zeta^{T}H_{n}-\frac{1}{n}\sum_{k=1}^{n}\sup_{\Delta}\left\{  \int_{0}%
^{1}\zeta^{T}Dh\left(  W_{i}+ \Delta u /n^{1/2},\theta_{\ast}\right)  \Delta
du-\left\Vert \Delta\right\Vert _{q}^{\rho} \right\}  \right\}  .
\label{RDisc}%
\end{equation}

In order to make sure that the lower bound asymptotically matches the upper
bound obtained in Theorem \ref{Thm-WPF_RWPI} we need to make sure that we rule
out cases in which the inner supremum is infinite in (\ref{RDisc}) with
positive probability in the prelimit.

In Proposition \ref{Prop_SLB_rho_1} we assume that $W$ has a positive density
with respect to the Lebesgue measure because in that case the condition
\[
\Pr\left(  \left\Vert \zeta^{T}Dh\left(  W,\theta_{\ast}\right)  \right\Vert
_{p}\leq1\right)  =1,
\]
(which appears in the upper bound obtained in Theorem \ref{Thm-WPF_RWPI})
implies that $\left\Vert \zeta^{T}Dh\left(  w,\theta_{\ast}\right)
\right\Vert _{p}\leq1$ almost everywhere with respect to the Lebesgue measure.
Due to the appearance of the integral in the inner supremum in (\ref{RDisc}),
an upper bound can be obtained for the inner supremum, which translates into a
tight lower bound for $n^{\rho/2}R_{n}\left(  \theta_{\ast}\right)  $.

Moving to the case $\rho>1$ studied in Proposition \ref{Prop_SLB_rho_L_1},
condition (\ref{U_T0}) in A5) guarantees that\ (for fixed $W_{i}$ and $n$)
\[
\left\Vert Dh\left(  W_{i}+\Delta u/n^{1/2},\theta_{\ast}\right)
\Delta\right\Vert =O\left(  \left\Vert \Delta\right\Vert _{q}^{\rho
}/n^{\left(  \rho-1\right)  /2}\right)  ,
\]
as $\left\Vert \Delta\right\Vert _{q}\rightarrow\infty$. Therefore,
the cost term
$\left( -\left\Vert \Delta\right\Vert _{q}^{\rho}\right) $ in
(\ref{RDisc}) will ensure a finite optimum in the prelimit for large
$n$. The condition that $\E \left\Vert W\right\Vert _{q}^{\rho}<\infty$
is natural because we are using a optimal transport cost
$c( u,w) =\left\Vert u-w\right\Vert _{q}^{\rho}$. If this condition is
not satisfied, then the underlying nominal distribution is at infinite
transport distance from the empirical distribution.

The local Lipschitz assumption A6) is just imposed to simplify the analysis
and can be relaxed; we have opted to keep A6) because we consider it mild in
view of the applications that we will study in the sequel.

\section{Using RWPI for optimal regularization\label{Sec-ML-RWP}}
In this section, we aim to utilize the limit theorems for the RWP
function derived in Section \ref{Sec_Quad_WPF} to select the radius of
uncertainty, $\delta,$ in the DRO formulation \eqref{RWPI-DRO}.
Then owing to the DRO representations derived in Section
\ref{Sec-ML-DRR}, this would imply an automatic choice of
regularization parameter $\lambda=\sqrt{\delta}$ in the square-root
LASSO example (following Theorem \ref{Cor-lp-good-equiv}), or
$\lambda=\delta$ in the regularized logistic regression (following
Theorem \ref{Cor-Classification-Equivalences}). In the development
below, we follow the logic described in the Introduction for the
square-root LASSO setting.

\subsection{Selection of $\delta$ and coverage properties}
\label{sec-sel-delta-cov-prop}
Throughout this section, let $\beta_{\ast}$ denote the underlying
linear or logistic regression model parameter from which the training
samples $\{(X_{i}%
,Y_{i}):i=1,\ldots,n\}$ are obtained.  Lemma
\ref{Lem-inf-sup-exchange-apps} below establishes that the infimum and
the supremum in the DRO formulation \eqref{RWPI-DRO} can be
exchanged. See Appendix C for a proof of Lemma
\ref{Lem-inf-sup-exchange-apps}.
\begin{lemma}
In the settings of Theorems \ref{Cor-lp-good-equiv} and
\ref{Cor-Classification-Equivalences}, if $\E\Vert X \Vert_{2}^{2} < \infty,$
we have that
\begin{equation}
\inf_{\beta\in\mathbb{R}^{d}}\sup_{\Pr \in\ \mathcal{U}_{\delta}(\Pr_{n})}%
\E_{\Pr}\left[  l\big( X,Y;\beta\big)\right]  =\sup_{\Pr \in\ \mathcal{U}_{\delta
}(\Pr_{n})}\inf_{\beta\in\mathbb{R }^{d}}\E_{\Pr}\left[  l\big(X,Y;\beta
\big)\right]  . \label{M_M}%
\end{equation}
\label{Lem-inf-sup-exchange-apps}
\end{lemma}
Recall the definition of $\Lambda_n(\delta)$ in \eqref{Lambda-n-Defn}.
As a consequence of Lemma \ref{Lem-inf-sup-exchange-apps}, the set
$\Lambda_n(\delta)$ contains the optimal solution obtained by solving
the problem in the left hand side of \eqref{M_M}. Indeed, if this was
not the case, the left hand side in \eqref{M_M} would be strictly
smaller than the right hand side of \eqref{M_M}. Recall from Section
\ref{Subsec_Opt_Reg_LASSO_Int} that our primary criterion for choosing
$\delta$ is to choose $\delta$ large enough so that
$\beta_\ast \in \Lambda_n(\delta)$ with desired confidence.
The property that estimator obtained by solving the DRO formulation
\eqref{RWPI-DRO} lies in $\Lambda_n(\delta)$, we believe, makes our
selection of $\delta$ logically consistent with the ultimate goal of
the overall estimation procedure, namely, estimating $\beta_{\ast}$.


Due to the optimality of $\beta_\ast,$ the convexity of the loss
$\ell(x,y;\,\cdot\,)$ in Examples \ref{Ex_GLasso} - \ref{Ex_RLR} and
finiteness of $\E\Vert X \Vert_2^2,$ we have that
$\E[D_\beta l(X,Y;\beta_\ast)] = \mathbf{0}.$ Consider the RWP function
with estimating equation $D_\beta l(x,y;\beta) = \mathbf{0}$ given by,
\[
  R_{n}(\beta)=\inf\big\{\mathcal{D}_{c}(\Pr,\Pr_{n}): D_\beta
  \E_\Pr\left[l(X,Y;\beta) \right] = \mathbf{0} \big\}.\] Then, as
explained in Section \ref{Subsec_Intro_WPF_LASSO}, the events
$\{R_n(\beta_\ast) \leq \delta\}$ and
$\{ \beta_\ast \in \Lambda_n(\delta)\}$ coincide. If $\delta$ is
selected so that $\delta \geq R_n(\beta_\ast),$ then the worst-case
loss estimated by the DRO formulation \eqref{RWPI-DRO} can be shown to
form an upper bound to the empirical risk evaluated at $\beta_\ast,$
thus controlling the bias portion of the generalization error. This is
the content of Proposition \ref{Prop-Gen-Bias-Control} below.
\begin{proposition}
\label{Prop-Gen-Bias-Control}
In the settings of Theorems \ref{Cor-lp-good-equiv} and
\ref{Cor-Classification-Equivalences}, if
$\delta \geq R_n(\beta_\ast),$ 
we have,
\begin{align*}
\left\vert  \E_{\Pr_n}\left[ l(X,Y;\beta_
\ast)\right] - \inf_\beta
  \sup_{\Pr\,\in\,\mathcal{U}_\delta(\Pr_n)} \E_\Pr\left[ l(X,Y;\beta)\right]
  \right\vert \leq C_1\delta + C_2(n)\mathbf{1}_{\{\rho = 2\}}\sqrt{\delta},
\end{align*}
where $C_1 := (2\rho-1) \Vert \beta_\ast \Vert^\rho$ and
$C_2(n) := 2\Vert \beta_\ast \Vert_p \sqrt{\E_{\Pr_n}[l(X,Y;\beta_\ast)]}.$ 
\end{proposition}

Now, in order to guarantee that $\delta \geq R_n(\beta_\ast)$ (or
equivalently, $\beta^\ast \in \Lambda_n(\delta)$) with a desired
confidence $1-\alpha,$ it is sufficient to proceed as in Section
\ref{Sec-RWP-Inference-Tool}: Let $\eta_\alpha$ be the
$(1-\alpha)$-quantile of the weak limit, $\bar{R},$ resulting from
$n^{\rho/2}R_n(\beta_\ast) \Rightarrow \bar{R},$ derived in Section
\ref{Sec_Quad_WPF}. In light of Theorems \ref{Cor-lp-good-equiv} and
\ref{Cor-Classification-Equivalences}, we have $\rho = 2$ for Example
\ref{Ex_GLasso} and $\rho = 1$ for Example \ref{Ex_RLR}.  If we take
$\eta \geq \eta_\alpha,$
\begin{align}
  \label{delta-sel}
  \delta = n^{-\rho/2}\eta \quad\text{ and } \quad \Lambda_n(\delta) = \{\beta: R_n(\beta) \leq
  n^{-\rho/2}\eta\},
\end{align}
then
$\lim_{n \rightarrow \infty}\Pr(R_n(\beta_\ast) > n^{-\rho/2}\eta)
\leq \alpha.$ Then, as demonstrated in \eqref{Conf-Region}, we have
$\lim_{n \rightarrow \infty}\Pr(\beta_\ast \in \Lambda_n(\delta)) \geq
1-\alpha.$ In Sections \ref{Sec-RWPF-Lin-Reg} - \ref{Sec-RWPF-GLM}
below, we illustrate the application of this prescription by deriving
upper bounds for $\bar{R}$ that are not dependent on the knowledge of
$\beta_\ast.$
\begin{theorem}
\label{thm-Gen-err-control}
In the settings of Theorems \ref{Cor-lp-good-equiv} and
\ref{Cor-Classification-Equivalences}, suppose that the samples
$\{(X_i,Y_i):i \leq n\}$ are obtained from the distribution $\Pr_\ast$
and $\E_{\Pr_\ast}\Vert X \Vert_{2}^{2} < \infty.$ For any
$1-\alpha \in (1/2,1),$ if $\delta$ is chosen to be $n^{-\rho/2}\eta$
for some $\eta \geq \eta_\alpha,$ then 
we have that, 
\begin{align*}
  \lim_{n \rightarrow \infty}  \Pr\left( \left\vert \inf_{\beta \in \mathbb{R}^d} \E_{\Pr_\ast}\left[l(X,Y;\beta)\right]  -
  \inf_{\beta \in \mathbb{R}^d}\sup_{\Pr\,\in\,\mathcal{U}_\delta(\Pr_n)}
  \E_\Pr\left[ l(X,Y;\beta)\right]\right\vert < \frac{C}{\sqrt{n}} \right) \geq 1- 2\alpha, 
\end{align*}
for some positive constant $C$ depending on $\rho,$
$\E_{\Pr_\ast}\left[ \ell(X,Y;\beta_\ast)\right]$ and
$\textnormal{Var}_{\Pr_\ast}\left[ \ell(X,Y;\beta_\ast\right].$
\end{theorem}
Proofs of Propositon \ref{Prop-Gen-Bias-Control} and Theorem
\ref{thm-Gen-err-control} are furnished in Appendix A.2 in the
supplementary material. Explicit prescriptions for the selection of
$\delta$ satisfying conditions of Theorem \ref{thm-Gen-err-control}
for the case of linear and logistic regression examples are provided
in Sections \ref{Sec-RWPF-Lin-Reg} and \ref{Sec-RWPF-Log-Reg}.

In contrast to the $O(n^{-1/d})$ rate of convergence for the
prescription of $\delta$ resulting from concentration inequalities for
$D_c(\Pr_n, \Pr_\ast)$ (see, for example, \cite[Theorem
2]{shafieezadeh-abadeh_distributionally_2015}, \cite[Theorem
3.5]{esfahani_data-driven_2015-1}), Theorem \ref{thm-Gen-err-control}
asserts that the DRO formulation with RWPI based prescription for
$\delta$ enjoys the optimal $O(n^{-1/2})-$rate of convergence for the
optimal risk. Roughly speaking, this is because the objective of RWPI
is to choose the radius $\delta$ resulting in good coverage properties
for the optimal parameter $\beta_\ast,$ which has $d-$degrees of
freedom; on the other hand, the objective behind concetration
inequalities is to choose $\delta$ with good coverage properties for
the data-generating probability distribution itself, which is an
infinite dimensional object. It is well-known that the distance
between a probability distribution and an empirical version of itself
constituting $n-$independent samples is $\Omega(n^{-1/d})$ as
$n \rightarrow \infty$ (see, for example, \cite{talagrand1992}).

Coverage for the optimal risk, for the particular example of LASSO
estimator, can also be derived, for example, from the limit theorems
in \cite{knight2000}. Once $\delta$ is chosen using RWP function, as
it can be seen from the proofs of Proposition
\ref{Prop-Gen-Bias-Control} and Theorem \ref{thm-Gen-err-control}, the
deduction of the rate of convergence and coverage turns out to be
fairly intuitive and simple. This serves to illustrate the fundamental
role played by the RWP function in determining the radius of the
uncertainty set.  A unified profile function based method to deduce
coverage of optimal risk for regularized estimators is entirely
novel. We believe that the approach described here could serve as a
template for deducing similar coverage guarantees for more general DRO
formulations that are not necessarily amenable to be recast as
regularized estimators.

\subsection{Linear regression models with squared loss function}
\label{Sec-RWPF-Lin-Reg}


In this section, we derive the asymptotic limiting distribution of suitably
scaled profile function corresponding to the estimating equation, 
$\E[(Y-\beta^{T}X)X] = \mathbf{0}.$
The chosen estimating equation describes the optimality condition for
the expected loss 
$\E[(Y-\beta^{T}X)^{2}],$ and therefore, the corresponding
$R_{n}(\beta_{\ast})$ is suitable for choosing
$\delta$ as in \eqref{delta-sel}, and the regularization parameter
$\lambda= \sqrt{\delta}$ in Example \ref{Ex_GLasso}.

\subsubsection{A stochastic upper bound for the RWP limit.}
Let $H_{0}$ denote the null hypothesis that the training samples
$\{(X_{1},Y_{1}),\ldots,(X_{n},Y_{n})\}$ are obtained independently from the
linear model $Y=\beta_{\ast}^{T}X+e$, where the error term $e$ has zero mean,
variance $\sigma^{2},$ and is independent of $X.$ Let $\Sigma= \E[XX^T].$

\begin{theorem}
\label{Thm-WPF} Consider the discrepancy measure $\mathcal{D}_{c}(\cdot)$
defined as in (\ref{Discrepancy_Def}) using the cost function
$c((x,y),(u,v))=(N_{q}((x,y),(u,v)))^{2}$ (the function $N_{q}$ is defined in
(\ref{Modified-Cost-Fn})). For $\beta\in\mathbb{R}^{d},$ let
\begin{align*}
R_{n}(\beta) = \inf\big\{ D_{c}(\Pr,\Pr_{n}): \ \E_{\Pr}\big[(Y-\beta^{T}X)X\big] =
\mathbf{0}\big\}.
\end{align*}
Then, under the null hypothesis $H_{0},$
\[
nR_{n}(\beta_{\ast}) \Rightarrow L_{1}:=\ \max_{\xi\in\mathbb{R}^{d}}\left\{
2\sigma\xi^{T}Z-\E\left\Vert e\xi-(\xi^{T}X)\beta_{\ast}\right\Vert _{p}%
^{2}\right\}  ,
\]
as $n\rightarrow\infty$. In the above limiting relationship, $Z\sim
\mathcal{N}(\mathbf{0},\Sigma).$ Further,
\[
L_{1}\overset{D}{\leq}\ L_{2}:=\frac{\E[e^{2}] }{\E[e^{2}] - \left(
\E|e|\right)  ^{2}}\Vert Z\Vert_{q}^{2}.
\]
Specifically, if the additive error term $e$ follows a centered normal
distribution, then
\[
L_{1}\overset{D}{\leq}\ L_{2}:=\frac{\pi}{\pi-2}\Vert Z\Vert_{q}^{2}.
\]

\end{theorem}

\noindent
In the above theorem, the relationship $L_{1}\overset{D}{\leq}L_{2}$
denotes that $L_{1}$ is stochastically dominated by $L_{2},$ in the
sense that, $\Pr(L_1 \geq x) \leq \Pr(L_2 \geq x)$ for all
$x \in \mathbb{R}.$ Note that this notation for stochastic upper bound
is different from the notation $\lesssim_{D}$ introduced in Section
\ref{Sec_Quad_WPF} to denote asymptotic stochastic upper bound. A
proof of Theorem \ref{Thm-WPF} as an application of Theorem
\ref{Thm-WPF_RWPI} and Proposition \ref{Prop_SLB_rho_L_1} is presented
in Appendix Section A.4 (see supplementary
material \cite{supp_RWPI}).

\subsubsection{Using Theorem \ref{Thm-WPF} to obtain regularization
  parameter for (\ref{Lp-Pen-LR}).} Let $\eta_{_{1-\alpha}}$ denote
the $(1-\alpha)$ quantile of the limiting random variable $L_{1}$ in
Theorem \ref{Thm-WPF}, or its stochastic upper bound $L_{2}.$ 
Then following the prescription in \eqref{delta-sel} and the DRO
equivalence in Theorem \ref{Cor-lp-good-equiv}, regularization
parameter for the $\ell_{p}$-penalized linear regression in
(\ref{Lp-Pen-LR}) can be chosen as follows:
\begin{itemize}
\item[1)] Draw samples $Z$ from $\mathcal{N}(\mathbf{0}, \Sigma)$ to estimate
the $1-\alpha$ quantile of one of the random variables $L_{1}$ or $L_{2} $ in
Theorem \ref{Thm-WPF}. Let us use $\hat{\eta}_{_{1-\alpha}}$ to denote the
estimated quantile. While $L_{2}$ is simply the norm of $Z,$ obtaining
realizations of limit law $L_{1}$ involves solving an optimization problem for
each realization of $Z.$ If $\Sigma= \E[XX^T]$ is not known, one can use
a simple plug-in estimator for $\E[XX^T]$ in place of $\Sigma.$

\item[2)] Choose the regularization parameter $\lambda$ to be,
\begin{align*}
\lambda= \sqrt{\delta} = \sqrt{{\hat{\eta}_{_{1-\alpha}}}/{n}}.
\end{align*}
\end{itemize}

It is interesting to note that the prescription of regularization
parameter obtained by using $L_2$ does not depend on the variance of
$e,$ thus removing the need for estimating the variance of $e.$ This
property is a key advantage of the use of square-root LASSO estimator
over the traditional LASSO (see \cite{belloni_square-root_2011}).

\subsubsection{On the approximation ratio $L_2/L_1$ when $p=q=2$.}
In the case where $q$ is taken to be $q = p = 2$ in Theorem
\ref{Cor-lp-good-equiv} (corresponding to $\ell_2-$penalization as in
ridge regression), it is possible to obtain an explicit expression for
the limit law $L_1$ as follows: Under the assumptions stated in
Theorem \ref{Thm-WPF}, we have
$\E[(e\mathbb{I}_d - X\beta_\ast^T) (eI_d - X\beta_\ast^T)^T] =
\sigma^2 \mathbb{I}_d + \Vert\beta_\ast\Vert^2 \Sigma.$ Then, as in
\eqref{simplified-rho-p-eq-2}, we obtain,
$L_1 = \sigma^2 Z^T\left(\sigma^2 \mathbb{I}_d +
  \Vert\beta_\ast\Vert^2 \Sigma\right)^{-1}Z.$ Suppose that $X$ is
centered so that $\E[X] = \mathbf{0}$ and $\Sigma$ is invertible. Then,
if $\Sigma = U\Lambda U^T$ is the eigen decomposition of $\Sigma,$ we
have that $N = \Lambda^{-1/2}U^T Z$ has normal distribution with mean
$\mathbf{0}$ and covariance $\mathbb{I}_d.$ As a result,
\begin{align*}
  L_1 = \sigma^2 Z^T\left(\sigma^2 \mathbb{I}_d + 
  \Vert\beta_\ast\Vert^2 \Sigma\right)^{-1}Z = \sum_{i=1}^d
  \frac{\Lambda_{ii}}{1+ \Lambda_{ii}\Vert \beta_\ast \Vert^2/\sigma^2 }
  N_i^2, \text{ and }
\end{align*}
\begin{align*}
  \frac{\E[e^2] - (\E\vert e
  \vert)^2}{\E[e^2]} L_2 = \Vert Z \Vert_2^2 =
  \sum_{i=1}^d\Lambda_{ii}N_i^2. 
\end{align*}
If we let
$c_1 = 1 + \sigma^{-2}\Vert \beta_\ast\Vert^2
\max_{i=1,\ldots,d}\Lambda_{ii}$ and
$c_2 = \text{Var}[e]/\text{Var}\vert e \vert,$ we arrive at the
relationship that, $L_1 \leq L_2 \leq c_1c_2L_1.$

One could aim to achieve lower bias in estimation by working with the
$(1-\alpha)$-quantile of the limit law $L_1$ (see Proposition
\ref{Prop-Gen-Bias-Control}), instead of that of the stochastic upper
bound $L_2.$ In order to do so, we propose to use any consistent
estimator for $\beta_\ast$ to be plugged in the expression for $L_1$
to result in asymptotically optimal prescription for $\delta.$ The
argument goes as follows: Let us write the limit law $L_1$ as
$L_1(\beta_\ast)$ in order to make the dependence of the limit law
$L_1$ on $\beta_\ast$ explicit. As $L_1(\cdot)$ is a continuous
function, if $\beta _{n}\rightarrow \beta _{\ast }$ in probability, we
have
\begin{equation*}
nR_{n}\left( \beta _{\ast }\right) - L_1\left( \beta _{n}\right) = \left(nR_{n}\left(
\beta _{\ast }\right) -L_1\left( \beta _{\ast }\right) \right) + \left(L_1\left( \beta _{\ast
}\right) - L_1\left( \beta _{n}\right)\right) \Rightarrow 0.
\end{equation*}
One could use, for example, sample average approximations (without
regularization) to compute $\beta _{n}$. 
We seek to verify in future research that the estimator
obtained via this plug-in approach indeed enjoys better generalization
guarantees.

\subsection{Logistic Regression with log-exponential loss
function\label{Sec-RWPF-GLM}}
\label{Sec-RWPF-Log-Reg}
In this section, we apply results in Section \ref{Sec_Quad_WPF} to prescribe
regularization parameter for $\ell_{p}$-penalized logistic regression in
Example \ref{Ex_RLR}.

\subsubsection{A stochastic upper bound for the RWP function.}
Let $H_{0}$ denote the null hypothesis that the training samples $(X_{1}%
,Y_{1}), \ldots, (X_{n},Y_{n})$ are obtained independently from a logistic
regression model satisfying
\begin{align*}
\log\left(  \frac{\Pr(Y = 1 | X = x)}{1-\Pr(Y = 1|X = x)} \right)  = \beta_{\ast
}^{T}x,
\end{align*}
for predictors $X \in\mathbb{R}^{d}$ and corresponding responses $Y
\in\{-1,1\};$ further, under null hypothesis $H_{0},$ the predictor $X$ has
positive density almost everywhere with respect to the Lebesgue measure on
$\mathbb{R}^{d}.$ The log-exponential loss (or negative log-likelihood) that
evaluates the fit of a logistic regression model with coefficient $\beta$ is
given by
\begin{align*}
l(x,y;\beta) = -\log p(y|x;\beta) = \log\big(1 + \exp(-y\beta^{T}x ) \big).
\end{align*}
If we let
\begin{align}
h(x,y;\beta) = D_{\beta}l(x,y;\beta) = \frac{-yx}{1+\exp(y\beta^{T}x)},
\label{Est-Eq-RLogR}%
\end{align}
then the optimal $\beta^{\ast}$ satisfies the first order condition
that 
$\E \left[  h(x,y;\beta_{\ast})\right]  = \mathbf{0}.$

\begin{theorem}
\label{Thm-GLM-WPF} Consider the discrepancy measure $\mathcal{D}_{c} (\cdot)
$ defined as in $($\ref{Discrepancy_Def}$)$ using the cost function
$c((x,y),(u,v))=N_{q}((x,y),(u,v))$ $($the function $N_{q}$ is defined in
(\ref{Modified-Cost-Fn})). For $\beta\in\mathbb{R}^{d},$ let
\begin{align*}
R_{n}(\beta) = \inf\left\{  D_{c}(\Pr,\Pr_{n}): \E_{\Pr} \big[ h(x,y;\beta) \big] =
\mathbf{0} \right\}  ,
\end{align*}
where $h(\cdot)$ is defined in (\ref{Est-Eq-RLogR}). Then, under the null
hypothesis $H_{0},$
\[
\sqrt{n}R_{n}(\beta_{\ast}) \Rightarrow\ L_{3}:=\ \sup_{\xi\in A} \xi^{T}Z
\]
as $n\rightarrow\infty.$ In the above limiting relationship,
\begin{align*}
  Z &\sim\mathcal{N}\left(  \mathbf{0},\E\left[  \frac{XX^{T}}{(1+\exp
      (Y\beta_{\ast}^{T}X))^{2}} \right]  \right)  \text{ and }\\
  A &= \left\{  \xi
  \in\mathbb{R}^{d}: \mathnormal{\ ess}\ \mathnormal{sup}_{x,y} \big\| \xi^{T}%
  D_{x}h(x,y;\beta_\ast)\big\|_{p} \ \leq1 \right\}  .
\end{align*}
Moreover, the limit law $L_{3}$ admits the following simpler stochastic
bound:
\[
L_{3}\overset{D}{\leq}\ L_{4}:= \|\tilde{Z}\|_{q},
\]
where $\tilde{Z} \sim\mathcal{N}(\mathbf{0},\E[XX^{T}]).$
\end{theorem}

\noindent A proof of Theorem \ref{Thm-WPF} as an application of Theorem
\ref{Thm-WPF_RWPI} and Proposition \ref{Prop_SLB_rho_1} is presented in
Appendix A.4 (see supplementary material
\cite{supp_RWPI}). 

\subsubsection{Using Theorem \ref{Thm-GLM-WPF} to obtain regularization
parameter for (\ref{Lp-Pen-LogR}) .} Similar to linear regression, the
regularization parameter for Regularized Logistic Regression discussed in
Example \ref{Ex_RLR} can be chosen by the following procedure:

\begin{itemize}
\item[1)] Estimate the $(1-\alpha)$-quantile of $L_{4} := \|\tilde{Z}\|_{q},$
where $\tilde{Z} \sim\mathcal{N}(\mathbf{0},\E[XX^{T}]).$ Let us use $\hat
{\eta}_{_{1-\alpha}}$ to denote the estimate of the quantile.

\item[2)] Choose the regularization parameter $\lambda$ in the norm
regularized logistic regression estimator (\ref{Lp-Pen-LogR}) in Example
\ref{Ex_RLR} to be, 
\begin{align*}
\lambda= \delta= {\hat{\eta}_{_{1-\alpha}}}/\sqrt{n}.
\end{align*}

\end{itemize}

\subsection{Optimal regularization in high-dimensional square-root LASSO
  \label{Sec-Nonasymp-Gen-Lasso}}

In this section, let us restrict our attention to the square-loss
function $l(x,y;\beta) = (y - \beta^{T}x)^{2}$ for the linear
regression model and the discrepancy measure $D_{c}$ defined using the
cost function $c = N_{q}$ with $q = \infty$ in
(\ref{Modified-Cost-Fn}). Then, due to Theorem
\ref{Cor-lp-good-equiv}, this corresponds to the interesting case of
square-root LASSO or $\ell_{2}$-LASSO that was rather a particular
example in the class of $\ell_{p}$ norm penalized linear regression
estimators considered in Section \ref{Sec-RWPF-Lin-Reg}.



As an interesting byproduct of the RWP function analysis, the
following theorem presents a prescription for regularization parameter
even in high dimensional settings where the ambient dimension $d$ is
larger than the number of samples $n.$ Given observations
$\{(X_i,Y_i): i = 1,\ldots,n\}$ from the linear model
$Y = \beta_\ast^TX + e,$ let
$\tilde{e}_i := (Y_i - \beta_\ast^TX_i)/\sigma,$ for $i =
1,\ldots,n.$ 
We have that the variance of the normalized error terms $\tilde{e}_i$
do not depend on $\sigma.$


\begin{theorem}
  Suppose that the assumptions imposed in Theorem \ref{Thm-WPF}
  hold. Then, 
\[
nR_{n}(\beta_{\ast}) \overset{D}{\leq} \frac{\Vert Z_{n}\Vert_{\infty}^{2}
}{\text{Var}_n|\tilde{e}|}, 
\]
where, $Z_{n}:=\frac{1}{\sqrt{n}} \sum_{i=1}%
^{n}\tilde{e}_{i}X_{i}$ and
$\text{Var}_n\vert \tilde{e} \vert := \sum_{i=1}^n(\vert \tilde{e}_i
\vert - n^{-1}\sum_{k=1}^n \vert \tilde{e}_i \vert)^2.$ 
\label{Thm-RWP-UB-d-Growing}
\end{theorem}

\begin{remark}
  \textnormal{Suppose that the additive error $e$ is normally
    distributed and the observations $X_i = (X_{i1},\ldots, X_{id})$
    are normalized so that $n^{-1}\sum_{i=1}^nX_{ij}^2 = 1$ for
    $j = 1,\ldots,d.$ Then, for any $\alpha < 1/8,$
    $C > 0, \varepsilon > 0,$ due to Lemma 1(iii) of
    \cite{belloni_square-root_2011}, the stochastic bound in Theorem
    \ref{Thm-RWP-UB-d-Growing} simplifies as follows: Conditional on
    the observations $\{X_i:i=1,\ldots,n\},$ we have, 
\[
  \sqrt{R_{n}(\beta_{\ast})} \leq \frac{\pi}{\pi-2}\frac{\Phi
    ^{-1}(1-\alpha/2d)}{\sqrt{n}},
\]
with probability asymptotically larger than $1-\alpha,$ as
$n \rightarrow \infty,$ uniformly in $d$ such that
$\log d \leq C n^{1/2-\varepsilon}.$ Here, $\Phi ^{-1}(1-\alpha)$
denotes the quantile $x$ satisfying $\Phi(x)=1-\alpha$ and
$\Phi(\cdot)$ is the cumulative distribution function of the standard
normal distribution defined on $\mathbb{R}.$ Moreover, if the additive
error $e$ is not normally distributed, then under additional
assumption that
$\sup_{n \geq 1}\sup_{1 \leq j \leq d}\E_{\Pr_n}\vert X_j \vert^a < \infty$ for
some $a > 2,$ we obtain from Lemma 2(iii) of
\cite{belloni_square-root_2011} that, 
\[
\sqrt{R_{n}(\beta_{\ast})} \leq \frac{\E[e^2]}{\E[e^2] - (\E\vert e
  \vert)^2}\frac{\Phi^{-1}(1-\alpha/2d)}{\sqrt{n}},
\]
with probability asymptotically larger than $1-\alpha,$ as
$n \rightarrow \infty,$ uniformly in $d$ such that
$ d \leq 0.5\alpha n^{(a-2-\varepsilon)/2}.$ }
\end{remark}

A proof of Theorem \ref{Thm-RWP-UB-d-Growing} is presented in Appendix
A.4 (see supplementary material \cite{supp_RWPI}). A commonly adopted
approach in the high dimensional regression literature (see, for
example,
\cite{bickel2009,negahban_unified_2012,belloni_square-root_2011,banerjee_estimation_2014}
and references therein) is to start with any choice
$\lambda > \Vert \tilde{S}\Vert_q,$ where $\tilde{S}$ is the score
function $D_\beta \E_n\left[l(X,Y;\beta_\ast)\right].$ This choice, in
the context of square-root Lasso, results in the regularization
parameter to be chosen larger than the $(1-\alpha)$-quantile of
$n^{-1/2}\Vert Z\Vert_\infty/\sqrt{\text{Var}_n[\tilde{e}]}$ (see (10)
in \cite{belloni_square-root_2011}). As observed in Theorem
\ref{Thm-RWP-UB-d-Growing}, working with an upper bound of the RWP
function results in choosing the $(1-\alpha)$-quantile of
$n^{-1/2}\Vert Z\Vert_\infty/\sqrt{\text{Var}_n\vert\tilde{e}\vert}.$
Indeed, this agreement of the regularization parameter with the high
dimensional linear regression literature strengthens the RWPI based
approach for selecting the radius of uncertainty. 
Since the RWPI based approach results in a prescription of
regularization parameter that is larger (by a factor
$\text{Var}_n[\tilde{e}]/\text{Var}_n\vert \tilde{e} \vert),$ the
generalization error bounds derived in the literature for high
dimensional regularized regression (see, for example, \cite[Corollary
1]{belloni_square-root_2011}) hold.

The approach in Theorem \ref{Thm-RWP-UB-d-Growing} is to identify an
upper bound that does not depend on $\beta_\ast.$ Instead, one could
choose $\delta \geq R_n(\hat{\beta}_n),$ by plugging in any consistent
estimator $\hat{\beta}_n.$ We identify investigating the possibility
of obtaining tighter error bounds via this plug-in approach as a
subject of future research.



\section{Numerical Examples\label{Sec-Num-Eg}}

In this section, we consider two examples that compare the numerical
performances of the square-root LASSO algorithm (see Example
\ref{Ex_GLasso}) when the regularization parameter $\lambda$ is
selected in the following two ways: 1) as described in Section
\ref{Sec-RWPF-Lin-Reg} using a suitable quantile of the RWPI limiting
distribution, and 2) using cross-validation. For comparison purposes,
we also list the performance of the respective ordinary least squares
estimator. As such, in both the examples, cross-validation based
approach iterates over multitude of choices of $\lambda,$ whereas the
optimal regularization via RWPI utilizes the respective square-root
LASSO algorithm only once for the prescribed value of $\lambda.$ This
naturally suggests potentially huge savings in computation that could
be valuable in large scale settings.

\begin{example}
  \textnormal{Consider the linear model
    $Y = 3X_{1} + 2X_{2} + 1.5X_{4} + e$ where the vector of predictor
    variables $X = (X_{1}, \ldots, X_{d})$ is distributed according to
    the multivariate normal distribution
    $\mathcal{N}(\mathbf{0}, \Sigma)$ with
    $\Sigma_{k,j} = 0.5^{|k-j|}$ and additive error $e$ is normally
    distributed with mean $0$ and standard deviation $\sigma=10.$
    Letting $n$ denote the number of training samples, we illustrate
    the effectiveness of the RWPI based square-root LASSO procedure
    for various values of $d$ and $n$ by computing the mean square
    loss / error (MSE) over a simulated test data set of size
    $N = 10000.$ Specifically, we take the number of predictors to be
    $d = 300$ and $600,$ the number of standardized i.i.d. training
    samples to range from $n = 350, 700, 3500, 10000,$ and the desired
    confidence level to be 95\%, that is, $1-\alpha= 0.95.$ In each
    instance, we run the square-root LASSO algorithm using the `flare'
    package proposed in \cite{li_flare_2015} (available as a library
    in R) with regularization parameter $\lambda$ chosen as prescribed
    in Section \ref{Sec-RWPF-Lin-Reg}. }

  \textnormal{Repeating each experiment 100 times, we report the
    average training and test MSE in Tables
    \ref{Tab-sparseRegressionResult} and
    \ref{Tab-sparseRegressionResult-II}, along with the respective
    results for ordinary least squares regression (OLS) and
    square-root LASSO algorithm with regularization parameter chosen
    as prescribed by cross-validation (denoted as SQ-LASSO CV in the
    tables.) We also report the average $\ell_{1}$ and $\ell _{2}$
    error of the regression coefficients in Tables
    \ref{Tab-sparseRegressionResult} and
    \ref{Tab-sparseRegressionResult-II}. In addition, we report the
    empirical coverage probability that the optimal error
    $\E[(Y - \beta_{\ast}^{T} X)^{2}] = \sigma^{2} = 100$ is smaller
    than the worst case expected loss computed by the DRO formulation
    \eqref{RWPI-DRO}. As this empirical coverage probability reported
    in Table \ref{Tab-Emp-Cov-Prob} is closer to the desired
    confidence $1-\alpha= 0.95,$ the worst case expected loss computed
    by \eqref{RWPI-DRO} can be seen as a tight upper bound of the
    optimal loss $\E[l(X,Y;\beta_{\ast})]$ (thus controlling
    generalization) with probability at least $1-\alpha=
    0.95$.} \label{Num-Eg-GLasso-Sim-Data}
\end{example}

\begin{example}
  \textnormal{Consider the diabetes data set from the `lars' package
    in R (see \cite{efron_least_2004}), where there are 64 predictors
    (including 10 baseline variables and other 54 possible
    interactions) and 1 response. After standardizing the variables,
    we split the entire data set of 442 observations into $n = 142$
    training samples (chosen uniformly at random) and the remaining
    $N = 300$ samples as test data for each experiment, in order to
    compute training and test mean square errors using the square-root
    LASSO algorithm with regularization parameter picked as in Section
    \ref{Sec-RWPF-Lin-Reg}.  After repeating the experiment 100 times,
    we report the average training and test errors in Table
    \ref{Tab-Diab-Data}, and compare the performance of RWPI based
    regularization parameter selection with other standard procedures
    such as OLS and square-root LASSO algorithm with regularization
    parameter chosen according to
    cross-validation. } \label{Num-Eg-GLasso-Diabetes-Data}
\end{example}

\begin{table}[h]
{\small \centering
\begin{tabular}
[c]{c|c|cccc}%
Training data & Method & Training Error & Test Error & $\ell_{1}$ loss &
$\ell_{2}$ loss\\
size, $n$ &  &  &  & $\Vert\beta- \beta_{\ast}\Vert_{1}$ & $\Vert\beta-
\beta_{\ast}\Vert_{2}$\\\hline
\multirow{4}{*}{$350$} & RWPI & $101.16 (\pm8.11) $ & $122.59 (\pm6.64) $ &
$4.08 (\pm0.69)$ & $5.23 (\pm0.76)$\\
& SQ-LASSO CV & $92.23 (\pm7.91)$ & $117.25 (\pm6.07)$ & $3.91 (\pm0.42)$ &
$5.02(\pm1.28)$\\
& OLS & $13.95 (\pm2.63)$ & $702.73 (\pm188.05)$ & $31.59 (\pm3.64)$ & $436.19
(\pm50.55)$\\\hline
\multirow{4}{*}{$700$} & RWPI & $101.81 (\pm3.01) $ & $117.96 (\pm4.80)$ &
$3.31 (\pm0.40)$ & $4.38 (\pm0.48)$\\
& SQ-LASSO CV & $99.66 (\pm4.64)$ & $115.46 (\pm4.36)$ & $2.96 (\pm0.37)$ &
$3.98(\pm0.66)$\\
& OLS & $56.82 (\pm3.94)$ & $178.44 (\pm21.74)$ & $10.99 (\pm0.57)$ & $152.04
(\pm8.25)$\\\hline
\multirow{4}{*}{$3500$} & RWPI & $102.55 (\pm2.39) $ & $108.44 (\pm2.54)$ &
$2.18 (\pm0.16)$ & $3.28 (\pm1.66)$\\
& SQ-LASSO CV & $100.74 (\pm2.35)$ & $113.83 (\pm2.33)$ & $2.66 (\pm0.14)$ &
$3.91 (\pm2.18)$\\
& OLS & $90.37 (\pm2.17)$ & $114.78 (\pm5.50)$ & $3.96 (\pm0.20)$ & $54.67
(\pm3.09)$\\\hline
\multirow{4}{*}{$10000$} & RWPI & $102.12 (\pm8.11) $ & $105.97 (\pm0.88)$ &
$1.13 (\pm0.08)$ & $1.63 (\pm0.11)$\\
& SQ-LASSO CV & $100.69 (\pm7.91)$ & $112.82 (\pm0.71)$ & $1.15 (\pm0.07)$ &
$1.94 (\pm0.12)$\\
& OLS & $95.91 (\pm1.11)$ & $107.74 (\pm2.96)$ & $2.23 (\pm0.10)$ & $30.91
(\pm1.43)$\\
&  &  &  &  &
\end{tabular}
}\caption{ Sparse linear regression for $d = 300$ predictor variables in
  Example \ref{Num-Eg-GLasso-Sim-Data}. The training and test mean square errors
  of RWPI based square-root LASSO regularization parameter selection is compared
  with ordinary least squares estimator (written as OLS) and cross-validation
  based square-root LASSO estimator (written as SQ-LASSO CV)}%
\label{Tab-sparseRegressionResult}%
\end{table}

\begin{table}[h]
\centering
{\small
\begin{tabular}
[c]{c|c|cccc}%
Training data & Method & Training Error & Test Error & $\ell_{1}$ loss &
$\ell_{2}$ loss\\
size, $n$ &  &  &  & $\Vert\beta- \beta_{\ast}\Vert_{1}$ & $\Vert\beta-
\beta_{\ast}\Vert_{2}$\\\hline
\multirow{4}{*}{$350$} & RWPI & $108.05 (\pm8.38) $ & $109.46 (\pm4.68) $ &
$4.02 (\pm0.71)$ & $4.08 (\pm0.70)$\\
& SQ-LASSO CV & $93.17 (\pm10.83)$ & $104.51 (\pm4.76)$ & $2.23 (\pm0.38)$ &
$6.89 (\pm2.35)$\\
& OLS & $-$ & $-$ & $-$ & $-$\\\hline
\multirow{4}{*}{$700$} & RWPI & $104.33 (\pm5.03) $ & $103.18 (\pm2.14) $ &
$2.91 (\pm0.42)$ & $2.99 (\pm0.43)$\\
& SQ-LASSO CV & $100.50 (\pm4.70)$ & $99.92 (\pm2.18)$ & $1.45 (\pm0.28)$ &
$2.82 (\pm0.64)$\\
& OLS & $14.27 (\pm2.02)$ & $699.06 (\pm137.45)$ & $31.66 (\pm2.21)$ & $518.02
(\pm44.87)$\\\hline
\multirow{4}{*}{$3500$} & RWPI & $101.52 (\pm2.52) $ & $96.38 (\pm0.80) $ &
$1.23 (\pm0.24)$ & $1.32 (\pm0.24)$\\
& SQ-LASSO CV & $102.58 (\pm2.49)$ & $98.55 (\pm0.94)$ & $1.18 (\pm0.15)$ &
$1.94 (\pm0.24)$\\
& OLS & $82.22 (\pm2.31)$ & $102.01 (\pm6.14)$ & $6.76 (\pm0.23)$ & $114.05
(\pm5.73)$\\\hline
\multirow{4}{*}{$10000$} & RWPI & $101.36 (\pm1.11) $ & $94.86 (\pm0.36)$ &
$0.75 (\pm0.13)$ & $0.81 (\pm0.14)$\\
& SQ-LASSO CV & $103.00 (\pm1.11)$ & $98.55 (\pm0.49)$ & $1.16 (\pm0.08)$ &
$1.94 (\pm0.13)$\\
& OLS & $95.11 (\pm1.10)$ & $99.53 (\pm4.83)$ & $3.26 (\pm0.11)$ & $63.67
(\pm2.16)$\\
&  &  &  &  &
\end{tabular}
}\caption{ Sparse linear regression for $d = 600$ predictor variables in
  Example \ref{Num-Eg-GLasso-Sim-Data}. The training and test mean square errors
  of RWPI based square-root LASSO regularization parameter selection is compared
  with ordinary least squares estimator (written as OLS) and cross-validation
  based square-root LASSO estimator (written as SQ-LASSO CV). As $n < d$ when $n
  = 350,$ OLS estimation is not applicable in that case (denoted by a blank)}%
\label{Tab-sparseRegressionResult-II}%
\end{table}

\begin{table}[h]
\centering
\par%
\begin{tabular}
[c]{c|p{1.2cm}p{1.0cm}p{1.1cm}p{1.1cm}}%
No. of predictors & Training & sample & size & \\
$d$ & 350 & 700 & 3500 & 10000\\\hline
300 & 0.974 & 0.977 & 0.975 & 0.969\\
600 & 0.963 & 0.966 & 0.970 & 0.968\\
&  &  &  &
\end{tabular}
\caption{Coverage Probability of empirical worst case expected loss in Example
\ref{Num-Eg-GLasso-Sim-Data}}%
\label{Tab-Emp-Cov-Prob}%
\end{table}

\begin{table}[h]
\centering
\begin{tabular}
[c]{c|cc}
& Training Error & Testing Error\\\hline
RWPI & $0.58 (\pm0.05) $ & $0.60 (\pm0.04)$\\
SQ-LASSO CV & $0.44 (\pm0.06)$ & $0.57 (\pm0.03)$\\
OLS & $0.26 (\pm0.05)$ & $1.38 (\pm0.68)$\\
&  &
\end{tabular}
\caption{Linear Regression for Diabetes data in Example
  \ref{Num-Eg-GLasso-Diabetes-Data} with 142 training samples and 300 test
  samples. The training and test mean square errors of RWPI based square-root
  LASSO regularization parameter selection is compared with ordinary least
  squares estimator (written as OLS) and cross-validation based square-root
  LASSO estimator (written as SQ-LASSO CV).}
\label{Tab-Diab-Data} 
\end{table}

\section{Conclusions}
We showed that popular machine learning estimators such as square-root
LASSO, regularized logistic regression, support vector machines,
etc. can be recast as particular examples of optimal transport based
DRO formulation in \eqref{RWPI-DRO}. We introduced Robust Wasserstein
Profile function and utilized its behaviour at the optimal parameter
$\beta_\ast$ to present a criterion for choosing the radius, $\delta,$
in the DRO formulation \eqref{RWPI-DRO}. We illustrated how this
translates to choosing regularization parameters and coverage
guarantees for optimal risk in the settings of $\ell_p-$norm
regularized linear and logistic regression. We observe that the
proposed prescriptions of the radius $\delta$ for the DRO formulation
\eqref{RWPI-DRO} result in similar prescriptions that arise from
independent considerations in Statistics literature. This indeed
strengthens the Wasserstein Profile function based approach towards
choosing the radius, $\delta,$ for the DRO formulation
\eqref{RWPI-DRO}.

Following the results presented in this paper, we investigate the
behaviour of the profile function $R_n(\theta)$ in the vicinity of the
optimal parameter $\theta_\ast$ in \cite{Confidence_prep} and
establish a limiting relationship of the form,
$n^{\rho/2}R_n(\theta_\ast + \Delta/\sqrt{n}) \Rightarrow L(\Delta),$
for a continuous $L(\cdot).$ Such a relationship can be used to
accomplish the following tasks: 1) construct confidence intervals for
the optimal parameter $\theta_\ast,$ 2) establish error bounds for the
solution to the DRO formulation \eqref{RWPI-DRO}, and 3)
systematically establish the validity of plugging-in any consistent
estimator for $\theta_\ast$ in order to obtain an asymptotically
optimal prescription of the radius $\delta.$ Such a plug-in approach
would obviate the need to derive stochastic upper bounds, on a
case-by-case basis, as is presently required in Section
\ref{Sec-ML-RWP}.

\section*{Supplementary material}
Proofs of the all the results in this article are furnished in the
supplementary material \cite{supp_RWPI} available after the References
section below.

\section*{Acknowledgements}
Support from NSF grant 1436700, NSF grant 1720451, NSF grant 1820942,
DARPA grant N660011824028 and Norges Bank are gratefully acknowledged
by J. Blanchet. 

\bibliographystyle{plain}
\bibliography{DR4}

\clearpage
\begin{center}
  \large 
  Supplementary material to the paper\\
  \textbf{Robust Wasserstein Profile Inference and Applications to
    Machine Learning}
\end{center}
\ \\ 
\appendix
This supplementary material to the paper ``Robust Wasserstein Profile
Inference and Applications to Machine Learning'' is organized as
follows: Proofs of all the main results in the paper are furnished in
Section \ref{Sec-Proofs}. As some of the main results in our paper
utilize strong duality for problems of moments, a quick introduction
to problem of moments along with a well-known strong duality result
that is useful in our context is provided in Section
\ref{AppSec-Str-Duality}. A technical result on exchange of sup and
inf in the DRO formulation \eqref{RWPI-DRO} is presented in Section
\ref{Sec-App-Sup-Inf}.  Relevant bibliography utilized in this
supplementary material is available at the end of this supplementary
material.

\section{Proofs of main results\label{Sec-Proofs}}
This section, comprising the proofs of the main results, is organized
as follows: Subsection \ref{Sec-ML-DRR-Proofs} is devoted to derive
the results on distributionally robust representations presented in
Section \ref{Sec-ML-DRR}. The proofs of results on coverage properties
are presented in Section
\ref{App-sec-proofs-coverage-properties}. Subsection
\ref{Sec-Quad-WPF-Proofs} contains the proofs of stochastic upper and
lower bounds (and hence weak limits) presented in Section
\ref{Sec_Quad_WPF}. Subsection \ref{Sec-RWP-Eg-Proofs} contains the
proofs of Theorems \ref{Thm-WPF} and \ref{Thm-GLM-WPF} as applications
of the stochastic upper and lower bounds presented in Section
\ref{Sec_Quad_WPF}. Some of the useful technical results that are not
central to the argument are presented in Sections
\ref{AppSec-Str-Duality} and \ref{Sec-App-Sup-Inf}.

\subsection{Proofs of the distributionally robust representations in Section 
\protect\ref{Sec-ML-DRR}\label{Sec-ML-DRR-Proofs}}

Here we provide proofs for results in Sections
\ref{Sec-Dual-Form-DRO}, \ref{Sec-ML-DRR} that recover various norm
regularized regressions as a special cases of distributionally robust
regression (Proposition 
\ref{Thm-lp-Equiv}, Theorems
\ref{Cor-lp-good-equiv} and \ref{Cor-Classification-Equivalences}).


\begin{proof}[\textbf{Proof of Proposition \protect\ref{Thm-lp-Equiv}}]
We utilize the duality result in Proposition \ref{Prop-Duality} to prove
Proposition \ref{Thm-lp-Equiv}. For brevity, let $\bar{X}_{i}=(X_{i},Y_{i})$ and 
$\bar{\beta}=(-\beta ,1).$ Then the loss function becomes $%
l(X_{i},Y_{i};\beta )=(\bar{\beta}^{T}\bar{X}_{i})^{2}.$ We first decipher
the function $\phi _{\gamma }(X_{i},Y_{i};\beta )$ defined in Proposition %
\ref{Prop-Duality}: 
\begin{equation*}
\phi _{\gamma }(X_{i},Y_{i};\beta )=\sup_{\bar{u}\in \mathbb{R}
^{d+1}}\left\{ (\bar{\beta}^{T}\bar{u})^{2}-\gamma \Vert \bar{X}_{i}-\bar{u}
\Vert _{q}^{2}.\right\}
\end{equation*}
To proceed further, we change the variable to $\Delta =\bar{u}-\bar{X}_{i},$
and apply H\"{o}lder's inequality to see that $|\bar{\beta}^{T}\Delta |\leq
\Vert \bar{\beta}\Vert _{p}\Vert \Delta \Vert _{q},$ where the equality
holds for some $\Delta \in \mathbb{R}^{d+1.}$ Therefore, 
\begin{align*}
\phi _{\gamma }(\bar{X}_{i};\beta )& =\sup_{\Delta \in \mathbb{R}^{d+1}} %
\big\{\left( \bar{\beta}^{T}\bar{X}_{i}+\bar{\beta}^{T}\Delta \right)
^{2}-\gamma \left\Vert \Delta \right\Vert _{q}^{2}\big\} \\
& =\sup_{\Delta \in \mathbb{R}^{d+1}}\left\{ \left( \bar{\beta}^{T}\bar{X}
_{i}+\text{sign}\left( \bar{\beta}^{T}\bar{X}_{i}\right) \left\vert \bar{
\beta}^{T}\Delta \right\vert \right) ^{2}-\gamma \left\Vert \Delta
\right\Vert _{q}^{2}\right\} \\
& =\sup_{\Delta \in \mathbb{R}^{d+1}}\left\{ \left( \bar{\beta}^{T}\bar{X}
_{i}+\text{sign}\left( \bar{\beta}^{T}\bar{X}_{i}\right) \left\Vert \Delta
\right\Vert _{q}\left\Vert \bar{\beta}\right\Vert _{p}\right) ^{2}-\gamma
\left\Vert \Delta \right\Vert _{q}^{2}\right\} .
\end{align*}
On expanding the squares, the above expression simplifies as below: 
\begin{align}
\phi _{\gamma }(\bar{X}_{i};\beta )& =\left( \bar{\beta}^{T}\bar{X}
_{i}\right) ^{2}+\sup_{\Delta \in \mathbb{R}^{d+1}}\left\{ -\left( \gamma
-\left\Vert \bar{\beta}\right\Vert _{p}^{2}\right) \left\Vert \Delta
\right\Vert _{q}^{2}+2\left\vert \bar{\beta}^{T}\bar{X}_{i}\right\vert
\left\Vert \bar{\beta}\right\Vert _{p}\left\Vert \Delta \right\Vert
_{q}\right\}  \notag \\
& =\left\{ 
\begin{array}{rcl}
\left( \bar{\beta}^{T}\bar{X}_{i}\right) ^{2}{\gamma }/{(\gamma -\left\Vert 
\bar{\beta}\right\Vert _{p}^{2})} & \text{ if }\gamma >\left\Vert \bar{\beta}
\right\Vert _{p}^{2}, &  \\ 
+\infty \text{ } & \text{ if }\gamma \leq \left\Vert \bar{\beta}\right\Vert
_{p}^{2}. & 
\end{array}
\right.  \label{Inter-Duality-0}
\end{align}
With this expression for $\phi _{\gamma }(X_{i},Y_{i};\beta ),$ we next
investigate the right hand side of the duality relation in Proposition \ref%
{Prop-Duality}. As $\phi _{\gamma }(x,y;\beta )=\infty $ when $\gamma \leq
\Vert \beta \Vert _{p}^{2},$ we obtain from the dual formulation in
Proposition \ref{Prop-Duality} that 
\begin{align}
\sup_{\Pr:\mathcal{D}_{c}(\Pr,\Pr_{n})\leq \delta }\E_{\Pr}\left[ l(X,Y;\beta )\right] &
=\inf_{\gamma \geq 0}\left\{ \gamma \delta +\frac{1}{n}\sum_{i=1}^{n}\phi
_{\gamma }(X_{i},Y_{i};\beta )\right\}  \notag \\
& =\inf_{\gamma > \Vert \beta \Vert _{p}^{2}}\left\{ \gamma \delta +\frac{
\gamma }{\gamma -\Vert \bar{\beta}\Vert _{p}^{2}}\frac{1}{n}\sum_{i=1}^{n}( 
\bar{\beta}^{T}\bar{X}_{i})^{2}\right\} .  \label{Inter-Duality}
\end{align}
Now, see that $\sum_{i=1}^{n}(\bar{\beta}^{T}\bar{X}_{i})^{2}/n$ is nothing
but the mean square error $MSE_{n}(\beta ).$ Next, as the right hand side of
(\ref{Inter-Duality}) is a convex function growing to $\infty $ (when $%
\gamma \rightarrow \infty $ or $\gamma \rightarrow \Vert \bar{\beta}\Vert
_{p}^{2}$ ), its global minimizer can be characterized uniquely via first
order optimality condition. This, in turn, renders the right hand side of (%
\ref{Inter-Duality}) as 
\begin{equation*}
\sup_{\Pr:\mathcal{D}_{c}(\Pr,\Pr_{n})\leq \delta }\E_{\Pr}\left[ l(X,Y;\beta )\right] =\left( 
\sqrt{MSE_{n}(\beta)}+\sqrt{\delta }\Vert \bar{\beta}\Vert _{p}\right) ^{2}.
\end{equation*}
This completes the proof of Proposition \ref{Thm-lp-Equiv}.
\end{proof}

\bigskip

\noindent \textbf{\textit{Outline of a proof of Theorem \ref{Cor-lp-good-equiv}.}}
The proof of Theorem \ref{Cor-lp-good-equiv} is essentially the same as the
proof of Proposition \ref{Thm-lp-Equiv}, except for adjusting for $\infty $ in
the definition of cost function $N_{q}((x,y),(u,v))$ when $y\neq v$ (as in
the derivation leading to $\phi _{\gamma }(X_{i},Y_{i};\beta )$ defined in (%
\ref{Adjust-Infty})). First, see that 
\begin{align*}
\phi_\gamma(X_i, Y_i;\beta) = \sup_{x^{\prime }\in \mathbb{R}^d,y^{\prime
}\in \mathbb{R}} \big\{ (y^{\prime T}x^{\prime 2 }- \gamma N_q\big(%
(x^{\prime },y^{\prime }), (X_i,Y_i)\big)\big\}.
\end{align*}
As $N_q((x^{\prime },y^{\prime }), (X_i,Y_i)) = \infty$ when $y^{\prime
}\neq Y_i,$ the supremum in the above expression is effectively over only $%
(x^{\prime },y^{\prime })$ such that $y^{\prime }= Y_i.$ As a result, we
obtain, 
\begin{align*}
\phi_\gamma(X_i, Y_i;\beta) &= \sup_{x^{\prime }\in \mathbb{R}^d} \big\{ %
(Y_i - \beta^Tx^{\prime 2 }- \gamma N_q\big((x^{\prime },Y_i), (X_i,Y_i)\big)%
\big\}. \\
&= \sup_{x^{\prime }\in \mathbb{R}^d} \big\{ (Y_i - \beta^Tx^{\prime 2 }-
\gamma \Vert x^{\prime }- X_i\Vert_q^2\big)\big\}.
\end{align*}
Now, following same lines of reasoning as in the proof of Theorem \ref%
{Thm-lp-Equiv} and the derivation leading to (\ref{Inter-Duality-0}), we
obtain 
\begin{equation*}
\phi _{\gamma }(x,y;\beta )= 
\begin{cases}
\frac{\gamma }{\gamma -\Vert \beta \Vert _{p}^{2}} & (Y_i - \beta^TX_i)^2
\quad\quad \text{ when }\lambda > \Vert {\beta}\Vert _{p} ^{2}, \\ 
& \quad\quad \quad\quad +\infty \quad\quad \text{ otherwise.}%
\end{cases}%
\end{equation*}
The rest of the proof is same as in the proof of Proposition \ref{Thm-lp-Equiv}.

\begin{proof}[\textbf{Proof of Theorem \protect\ref%
{Cor-Classification-Equivalences}}]
As in the proof of Proposition \ref{Thm-lp-Equiv}, we apply the duality
formulation in Proposition \ref{Prop-Duality} to write the worst case
expected log-exponential loss function as: 
\begin{equation*}
\sup_{\Pr:\ \mathcal{D}_{c}(\Pr,\Pr_{n})\leq \delta }\E_{\Pr}\big[l(X,Y;\beta )\big] %
=\inf_{\lambda \geq 0}\left\{ \delta \lambda +\frac{1}{n}\sum_{i=1}^{n}
\sup_{x}\left\{ \log \left( 1+\exp (-Y_{i}\beta ^{T}x)\right) -\lambda
\left\Vert x-X_{i}\right\Vert _{p}\right\} \right\} .
\end{equation*}
For each $(X_{i},Y_{i})$, following Lemma 1 in
\cite{shafieezadeh-abadeh_distributionally_2015}, we obtain  
\begin{equation*}
\sup_{x}\left\{ \log \left( 1+\exp (-Y_{i}\beta ^{T}x)\right) -\lambda
\left\Vert x-X_{i}\right\Vert _{p}\right\} =\left\{ 
\begin{array}{rcl}
\log \left( 1+\exp (-Y_{i}\beta ^{T}X_{i})\right) & \text{ if }\left\Vert
\beta \right\Vert _{q}\leq \lambda, &  \\ 
&  &  \\ 
+\infty & \text{ if }\left\Vert \beta \right\Vert _{q}>\lambda. & 
\end{array}
\right.
\end{equation*}
Then we can write the worst case expected loss function as,
\begin{align*}
\inf_{\lambda \geq 0}&\left\{ \delta \lambda +\frac{1}{n}\sum_{i=1}^{n}
\sup_{x}\left\{ \log \left( 1+\exp (-Y_{i}\beta ^{T}x)\right) -\lambda
\left\Vert x-X_{i}\right\Vert _{p}\right\} \right\} \\
&=\inf_{\lambda \geq 0}\left\{ \delta \lambda + \frac{1}{n}
\sum_{i=1}^{n}\left(\log \left(1 +\exp (-Y_{i}\beta ^{T}X_{i})\right) {1}_{
\{\lambda > \left\Vert \beta \right\Vert _{q}\} }+\infty {1}_{ \{\lambda
\leq \left\Vert \beta \right\Vert _{q}\} }\right) \right\} \\
&= \inf_{\lambda > \left\Vert \beta \right\Vert _{q}}\left\{ \delta \lambda
+ \frac{1}{n}\sum_{i=1}^{n}\log \left(1 +\exp (-Y_{i}\beta ^{T}X_{i})\right)
\right\} \\
&=\frac{1}{n}\sum_{i=1}^{n}\log \left( 1+\exp (-Y_{i}\beta ^{T}X_{i})\right)
+\delta \left\Vert \beta \right\Vert _{q},
\end{align*}
which is equivalent to regularized logistic regression in the theorem
statement.

For SVM with hinge loss function, let us apply the duality formulation
in Proposition \ref{Prop-Duality} to write the worst case expected
Hinge loss function as:
\[
\sup_{\Pr:\ D_{c}(\Pr,\Pr_{n})\leq\delta}\E_{\Pr}\big[\left(  1-Y\beta^{T}X\right)
^{+}\big]=\inf_{\lambda\geq0}\left\{  \delta\lambda+\frac{1}{n}\sum_{i=1}%
^{n}\sup_{x}\left\{  \left(  1-Y_{i}\beta^{T}x\right)  ^{+}-\lambda\left\Vert
x-X_{i}\right\Vert _{p}\right\}  \right\}  .
\]
For each $i$, let us consider the maximization problem and for
simplicity we denote $\Delta_{i}=x-X_{i}$
\begin{align*}
  &  \sup_{\Delta u_{i}}\left\{  \left(  1-Y_{i}\beta^{T}\left(  X_{i}+\Delta_{i}\right)  \right)  ^{+}-\lambda\left\Vert \Delta_{i}\right\Vert
    _{p}\right\}  \\
  &  =\sup_{\Delta_{i}}\sup_{0\leq\alpha_{i}\leq1}\left\{  \alpha_{i}\left(
    1-Y_{i}\beta^{T}\left(  X_{i}+\Delta_{i}\right)  \right)  -\lambda\left\Vert
    \Delta_{i}\right\Vert _{p}\right\}  \\
  &  =\sup_{0\leq\alpha_{i}\leq1}\sup_{\Delta_{i}}\left\{  \alpha_{i}
    Y_{i}\beta^{T}\Delta_{i}-\lambda\left\Vert \Delta_{i}\right\Vert
    _{p}+\alpha_{i}\left(  1-Y_{i}\beta^{T}X_{i}\right)  \right\}  \\
  &  =\sup_{0\leq\alpha_{i}\leq1}\sup_{\Delta_{i}}\left\{  \alpha
    _{i}\left\Vert \beta\right\Vert _{q}\left\Vert \Delta_{i}\right\Vert
    _{p}-\lambda\left\Vert \Delta_{i}\right\Vert _{p}+\alpha_{i}\left(
    1-Y_{i}\beta^{T}X_{i}\right)  \right\}  \\
  &  =\left\{
\begin{array}
[c]{rcl}%
\left(  1-Y_{i}\beta^{T}X_{i}\right)  ^{+} & \text{ if }\left\Vert
\beta\right\Vert _{q}\leq\lambda & +\infty\\
&  & \\
+\infty & \text{ if }\left\Vert \beta\right\Vert _{q}>\lambda &
\end{array}
\right.
\end{align*}
The first equality follows from the observation that
$x^{+}=\sup_{0\leq\alpha\leq1}x$; second equality is because the
function is concave in $\Delta_{i},$ linear in $\alpha;$ as $\alpha$
is in a compact set, we can apply minimax theorem to switch the order
of maxima; third equality is due to applying H\"{o}lder inequality to
the first term, and since the second term only depends on the norm of
$\Delta_{i,}$ the equality holds for this maximization problem. For
the outer minimization, it is sufficient to restrict to
$\lambda \geq\left\Vert \beta\right\Vert _{q}$. As a result, we obtain
\[
\inf_{\lambda\geq\left\Vert \beta\right\Vert _{q}}\left\{  \delta\lambda
+\frac{1}{n}\sum_{i=1}^{n}\left(  1-Y_{i}\beta^{T}X_{i}\right)  ^{+}\right\}
=\frac{1}{n}\sum_{i=1}^{n}\left(  1-Y_{i}\beta^{T}X_{i}\right)  ^{+}%
+\delta\left\Vert \beta\right\Vert _{q}.
\]
This completes the proof.
\end{proof}

\subsection{Proofs of results on coverage properties}
\label{App-sec-proofs-coverage-properties}
\begin{proof}[Proof of Proposition \ref{Prop-Gen-Bias-Control}]
  Let $\hat{\Pr}$ be a probability measure from the set, 
  \[\{ \Pr : D_c(\Pr, \Pr_n) \leq \delta, \ \E_\Pr\big[ D_\beta l(X,Y;\beta_\ast)]=
    {\bf 0}\},\] which is non-empty, because
  $\delta > R_n(\beta_\ast).$ Then,
\begin{align*}
  \inf_{\beta \in \mathbb{R}^d} 
  \sup_{\Pr : D_c(\Pr,\Pr_n) \leq \delta} \E_\Pr \left[ l(X,Y;\beta)\right] 
    \geq \inf_{\beta \in \mathbb{R}^d} \E_{\hat{\Pr}} \left[ l(X,Y;\beta)
    \right] = \E_{\hat{\Pr}}\left[ l(X,Y;\beta_\ast)\right]. 
\end{align*}
Moreover, since $D_c(\cdot)$ is symmetric in its arguments, we have
$D_c(\hat{\Pr}, \Pr_n) \leq \delta.$ As a result,
\begin{align}
  \E_{\Pr_n}\left[ l(X,Y;\beta_\ast)\right] -
  \inf_\beta\sup_{\Pr\,\in\,\mathcal{U}_\delta(\Pr_n)} \E_\Pr\left[
  l(X,Y;\beta)\right] \leq \sup_{\Pr:D_c(\hat{\Pr},\Pr) \leq \delta}
  \E_\Pr\left[l(X,Y;\beta_\ast)\right] -
  \E_{\hat{\Pr}}\left[l(X,Y;\beta_\ast)\right]. 
\label{inter-bias-control}
\end{align}
On the other hand, 
\begin{align*}
  \inf_\beta\sup_{\Pr\,\in\,\mathcal{U}_\delta(\Pr_n)} \E_\Pr\left[
  l(X,Y;\beta)\right] -   \E_{\Pr_n}\left[ l(X,Y;\beta_\ast)\right]
  \leq \sup_{\Pr:D_c(\Pr_n,\Pr) \leq \delta}
  \E_\Pr\left[l(X,Y;\beta_\ast)\right] -
  \E_{\Pr_n}\left[l(X,Y;\beta_\ast)\right],   
\end{align*}
which can be bounded from above to result in the desired bound,
$C_1\delta + C_2(n) \mathbf{1}_{\rho=2}\sqrt{\delta},$ by substituting
the regularized regression estimators derived in Theorem
\ref{Cor-lp-good-equiv} (when $\rho=2$) and Theorem
\ref{Cor-Classification-Equivalences} (when $\rho=1$). Likewise,
repeating the proofs of Theorems \ref{Cor-lp-good-equiv} and
\ref{Cor-Classification-Equivalences} for the case where the baseline
distribution is set to be $\hat{\Pr}$ (instead of $\Pr_n$), we obtain for
any $\beta \in \mathbb{R}^d$ that
\begin{align*}
  \sup_{\Pr:D_c(\hat{\Pr},\Pr) \leq \delta}
  \E_\Pr\left[l(X,Y;\beta)\right] -
  \E_{\hat{\Pr}}\left[l(X,Y;\beta)\right] = \delta \Vert \beta
  \Vert_p, 
\end{align*}
for the logistic regression example in Theorem
\ref{Cor-Classification-Equivalences}; and 
\begin{align*}
  \sup_{\Pr:D_c(\hat{\Pr},\Pr) \leq \delta}  \E_\Pr\left[l(X,Y;\beta)\right] &-  \E_{\hat{\Pr}}\left[l(X,Y;\beta)\right] 
  = 2\sqrt{\delta} \Vert \beta
  \Vert_p \sqrt{\E_{\hat{\Pr}}\left[(Y-\beta^TX)^2\right]} + \delta \Vert  \beta \Vert_p^2\\
  & \leq 2\sqrt{\delta} \Vert \beta
  \Vert_p   \sqrt{\sup_{\Pr\,\in\,\mathcal{U}_\delta(\Pr_n)}\E_{\Pr}\left[(Y-\beta^TX)^2\right]}
    + \delta \Vert   \beta \Vert_p^2,\\
  &= 2\sqrt{\delta}\Vert \beta \Vert_p
    \sqrt{\E_{\Pr_n}\left[(Y-\beta^TX)^2\right]} + 3\delta \Vert
    \beta \Vert_p^2,
\end{align*}
for the linear regression example in Theorem \ref{Cor-lp-good-equiv}. 
This verifies the upper bound for
\eqref{inter-bias-control}. 
\end{proof}

\begin{proof}[Proof of Theorem \ref{thm-Gen-err-control}] 
  Since $\delta = n^{-\rho/2}\eta$ for some $\eta \geq \eta_\alpha,$
  we have from the definition of $\eta_\alpha$ that,
  \[\lim_{n \rightarrow \infty}\Pr(R_n(\beta_\ast) > \delta) = \lim_{n
      \rightarrow \infty}\Pr(n^{\rho/2}R_n(\beta_\ast) > \eta) 
    \leq \alpha,\] as $n \rightarrow \infty.$ Then it follows from
  Proposition \ref{Prop-Gen-Bias-Control} that,
\begin{align*}
  \left\vert \E_{\Pr_n}[l(X,Y;\beta_\ast)] -
  \inf_{\beta \in \mathbb{R}^d}\sup_{\Pr
  \,\in\,\mathcal{U}_\delta(\Pr_n)} \E_\Pr\left[
  l(X,Y;\beta)\right]\right\vert \leq C_1\eta n^{-\rho/2} + C_2(n)
  \sqrt{\eta} \mathbf{1}_{\{\rho = 2\}}n^{-\rho/4},
\end{align*}
with probability greater than or equal to $1-\alpha,$ as
$n \rightarrow \infty.$ Moreover, due to Chebyshev's inequality, we
obtain, 
\[ \left\vert \E_{\Pr_n}[l(X,Y;\beta_\ast)] -
    \E_{\Pr_\ast}[l(X,Y;\beta_\ast)] \right\vert \leq
  \sqrt{\frac{\text{Var}_{\Pr_\ast}[l(X,Y;\beta_\ast)]}{\alpha n}},\]
and subsequently,
$C_2(n)/(2\Vert \beta_\ast\Vert_p) \leq
\sqrt{\E_{\Pr_\ast}[l(X,Y;\beta_\ast)]} +
(\alpha^{-1}n^{-1}\textnormal{Var}_{\Pr_\ast}[l(X,Y;\beta_\ast)])^{1/4},$
with probability 
exceeding $1-\alpha.$ Since
$\E_{\Pr_\ast}[l(X,Y;\beta_\ast)] = \inf_\beta \E_{\Pr_\ast}[l(X,Y;\beta)],$
the desired convergence in the statement of Theorem
\ref{thm-Gen-err-control} follows from triangle inequality and an
application of union bound to the above two inequalities.
\end{proof}

\subsection{Proofs of asymptotic stochastic upper and lower bounds of RWP
function in Section \protect\ref{Sec_Quad_WPF}\label{Sec-Quad-WPF-Proofs}}

We first use Proposition \ref{Prop-RWP-Duality} to derive a dual
formulation for $n^{\rho/2}R_n(\theta_\ast)$ which will be the
starting point of our analysis. Due to Assumption A2),
$\E[h(W,\theta_\ast)] = \mathbf{0}.$ Combining this observation with
the positive definiteness in Assumption A4), we have that $\mathbf{0}$
lies in the interior of convex hull of
$\{ h(u,\theta_\ast): u \in \mathbb{R}^m\}$ by using a supporting
hyperplane argument as in the proof of \cite[Proposition
8]{Confidence_prep}. Then, due to Proposition
\ref{Prop-RWP-Duality},
\begin{align*}
R_n(\theta_\ast) &= \sup_{\lambda \in \mathbb{R}^r} \left\{ -\frac{1}{n}
\sum_{i=1}^n \sup_{u \in \mathbb{R}^m} \big\{ \lambda^T h(u,\theta_\ast) -
\Vert u - W_i \Vert_q^\rho \big\}\right\}.
\end{align*}
In order to simplify the notation, throughout the rest of the proof we will
write $h\left( W_{i}\right) $ instead of $h\left( W_{i},\theta _{\ast
}\right) $ and $Dh\left( W_{i}\right) $ for $D_{w}h\left( W_{i},\theta
_{\ast }\right)$.

Letting $H_n = n^{-1/2}\sum_{i=1}^nh(W_i)$ and changing variables to $\Delta
= u - W_i,$ we obtain 
\begin{align*}
R_n(\theta_\ast) = \sup_{\lambda} \left\{ -\lambda^T \frac{H_n}{n^{1/2}} - 
\frac{1}{n} \sum_{i=1}^n \sup_{\Delta}\left\{ \lambda^T\big( h(W_i + \Delta)
- h(W_i) \big) - \Vert \Delta \Vert_q^\rho\right\}\right\}.
\end{align*}
Due to the fundamental theorem of calculus (using Assumption A3)), we
have that
\begin{equation*}
h\left( W_{i}+\Delta \right) -h\left( W_{i}\right) =\int_{0}^{1}Dh\left(
W_{i}+u\Delta \right) \Delta du.
\end{equation*}
Now, redefining $\zeta =\lambda n^{\left( \rho -1\right) /2}$ and $\Delta
=\Delta/n^{1/2}$ we arrive at following representation 
\begin{equation}
n^{\rho /2}R_{n}(\theta _{\ast })=\sup_{\zeta }\left\{ -\zeta
^{T}H_{n}-M_{n}\left( \zeta \right) \right\},  \label{Scaled-RWP-Dual-Rep}
\end{equation}
where 
\begin{equation}
M_{n}\left( \zeta \right) =\frac{1}{n}\sum_{i=1}^{n}\sup_{\Delta }\left\{ {
\zeta ^{T}}\int_{0}^{1}{D}h\left( W_{i}+n^{-1/2}\Delta u\right) \Delta
du-\left\Vert \Delta \right\Vert _{q}^{\rho }\right\}.  \label{Mn-Defn}
\end{equation}
The reformulation in \eqref{Scaled-RWP-Dual-Rep} is our starting point of
the analysis.

To proceed further, we first state a result which will allow us to apply a
localization argument in the representation of $n^{\rho/2}R_{n}\left( \theta
_{\ast }\right)$ in \eqref{Scaled-RWP-Dual-Rep}. Recall the definition of $%
M_n$ above in \eqref{Mn-Defn} and that $H_n = n^{-1/2}\sum_{i=1}^nh(W_i).$

\begin{lemma}
\label{Lem_Claim_1}Suppose that the Assumptions A2) to A4) are in force.
Then, for every $\varepsilon >0$, there exists $n_{0}>0$ and $b\in \left(
0,\infty \right) $ such that 
\begin{equation*}
\Pr\left( \sup_{\left\Vert \zeta \right\Vert _{p}\geq b}\left\{ -\zeta
^{T}H_{n}-M_{n}\left( \zeta \right) \right\} >0\right) \leq \varepsilon ,
\end{equation*}
for all $n\geq n_{0}$.
\end{lemma}

\begin{proof}[Proof of Lemma \protect\ref{Lem_Claim_1}]
  Recall that $q > 1$ and $p = q/(q-1).$ For $\zeta \neq 0$, we write
  $\bar{\zeta}=\zeta /\left\Vert \zeta \right\Vert _{p}$. Let us
  define the vector
  $V_{i}\left( \bar{\zeta}\right) ={D}h\left( W_{i}\right) ^{T}\bar{%
    \zeta}$, and put
\begin{equation}
\Delta _{i}^{\prime }=\Delta _{i}^{\prime }\left( \bar{\zeta}\right)
=\left\vert V_{i}\left( \bar{\zeta}\right) \right\vert ^{p/q}sgn\left(
V_{i}\left( \bar{\zeta}\right) \right) .  \label{Def_Del_p}
\end{equation}%
Define the set
$C_{0}=\{w \in \mathbb{R}^m: \left\Vert w\right\Vert _{p}\leq
c_{0}\}$, where $c_{0}$ will be chosen large enough momentarily. Then,
for any $c>0$, plugging in $\Delta =c\Delta _{i}^{\prime }$, we have
$\zeta^T Dh(W_i) \Delta = c\Vert \zeta^T Dh(W_i)\Vert_p\Vert
\Delta_i^{\prime }\Vert_q,$ and therefore,
\begin{align}
  &\sup_{\Delta }\left\{ {\zeta ^{T}}\int_{0}^{1}{D}h(
    W_{i}+n^{-1/2}\Delta u) \Delta du -\left\Vert \Delta \right\Vert
    _{q}^{\rho }\right\}  \notag\\
  &\quad=\sup_{\Delta} \left\{ \zeta^TDh(W_i)\Delta - \Vert \Delta
    \Vert_q^\rho + {\zeta ^{T}}\int_{0}^{1} \left[{D}h(
    W_{i}+n^{-1/2}\Delta u) - Dh(W_i) \right] \Delta du \right\}
    \notag\\ 
  &\quad\geq \max \bigg\{c\left\Vert {\zeta ^{T}D}h( W_{i})
    \right\Vert _{p}\left\Vert \Delta _{i}^{\prime }\right\Vert _{q}-c^{\rho
    }\left\Vert \Delta _{i}^{\prime }\right\Vert _{q}^{\rho }  \notag \\
  &\quad\quad \quad\quad\quad\quad +c{\zeta^{T}}\int_{0}^{1} \left[{D}h(
    W_{i}+cn^{-1/2}\Delta _{i}^{\prime }u) -{D}h\left( W_{i}\right) \right]\Delta
    _{i}^{\prime }du,0 \bigg\}I\left( W_{i}\in C_{0}\right) .  \label{e1aa}
\end{align}
Due to H\"{o}lder's inequality, 
\begin{align*}
  &I\left( W_{i}\in C_{0}\right) \left\vert {\zeta ^{T}}\int_{0}^{1} \left[{D}
    h( W_{i}+cn^{-1/2}\Delta _{i}^{\prime }u) -{D}h( W_{i})
    \right]\Delta _{i}^{\prime }du\right\vert \\
  &\quad\quad \leq I\left(W_{i}\in C_{0}\right) \left\Vert \zeta
    \right\Vert _{p} \int_{0}^{1} \left\Vert \big[{D}h( W_{i}+cn^{-1/2}\Delta
    _{i}^{\prime }u) -{D}h( W_{i}) \big]\Delta _{i}^{\prime
    }\right\Vert _{q}du.
\end{align*}
Because of continuity ${D}h\left( \cdot \right) $ and the fact that $%
W_{i}\in C_{0}$
(so the integrand is bounded), we have that the previous expression
converges to zero as $n\rightarrow \infty.$ Therefore, for given
positive constants $\varepsilon^{\prime },c$ (note than convergence is
uniform on $W_{i}\in C_{0}$), there exists $n_{0}$ such that for all
$n\geq n_{0}$
\begin{equation}
  cI\left( W_{i}\in C_{0}\right) \left\vert {\zeta ^{T}}\int_{0}^{1} \left[{D}
      h( W_{i}+cn^{-1/2}\Delta _{i}^{\prime }u) -{D}h( W_{i})
    \right]\Delta _{i}^{\prime }du\right\vert \leq
  c\varepsilon^{\prime }\left\Vert  \zeta \right\Vert
  _{p}.  \label{e1_b} 
\end{equation}
Next, as
$\Vert \bar{\zeta}^T Dh(W_i)\Vert_p^{p/q} = \Vert
\Delta_i^{\prime}\Vert_q$ and $1 + p/q = p,$
\begin{eqnarray*}
c\left\Vert {\zeta ^{T}D}h\left( W_{i}\right) \right\Vert _{p}\left\Vert
\Delta _{i}^{\prime }\right\Vert _{q}-c^{\rho }\left\Vert \Delta
_{i}^{\prime }\right\Vert _{q}^{\rho } =c\left\Vert {\zeta }%
\right\Vert_{p}\Vert \bar{\zeta}^T Dh(W_i)\Vert_p^p - c^{\rho
}\Vert \bar{\zeta}^T Dh(W_i)\Vert_p^{\rho \frac{p}{q}}.
\end{eqnarray*}
Consequently, it follows from \eqref{e1aa} and \eqref{e1_b} that 
\begin{align}
  M_n(\zeta) \geq \frac{1}{n} \sum_{i=1}^n \bigg\{ c\left\Vert {\zeta }%
\right\Vert_{p}\Vert \bar{\zeta}^T Dh(W_i)\Vert_p^p - c^{\rho
}\Vert \bar{\zeta}^T Dh(W_i)\Vert_p^{\rho \frac{p}{q}} - c\varepsilon' \Vert \zeta \Vert_p \bigg\} I\left( W_i
  \in C_0\right). 
\label{Inter-e1ab}
\end{align}
Now, since the map
$\bar{\zeta} \hookrightarrow \left\Vert \bar{\zeta}^T Dh(W_i)
\right\Vert_{p}^{p}$
is Lipschitz continuous on $\left\Vert \bar{\zeta}\right\Vert _{p}=1,$
we conclude that,
\begin{align}
\frac{1}{n} \sum_{i=1}^n\left\Vert \bar{\zeta}^T Dh(W_i) \right\Vert
_{p}^{p} I\left( W_{i}\in C_{0}\right) \rightarrow &\E\left[
\left \Vert \bar{\zeta}^T {D}h\left( W\right) \right\Vert
_{p}^{p}I\left( W\in C_{0}\right) \right],
\label{Lem2-Inter}
\end{align}
with probability one as $n\rightarrow \infty $. Moreover, due to Fatou's
lemma we have that the map $\bar{\zeta}\hookrightarrow \Pr\left( \left\Vert 
\bar{\zeta}^{T}{D}h\left( W\right) \right\Vert _{p}>0\right) $ is lower
semi-continuous. Therefore, by A4), we have that there exists $\delta >0$
such that 
\begin{equation}
  \inf_{\bar{\zeta}}\E \left\Vert \bar{\zeta}^{T}{D}h\left( W\right)
  \right\Vert _{p}^p >\delta .  \label{LB_T1}
\end{equation}%
Consecutively, by selecting $c_{0}>0$ large enough, we conclude from
\eqref{Lem2-Inter} that for $n\geq N^{\prime }\left( \delta \right) $,
\begin{equation}
\frac{1}{n}\sum_{i=1}^{n} \left\Vert \bar{\zeta}^T Dh(W_i) \right\Vert
_{p}^{p} I\left( W_{i}\in C_{0}\right) >\frac{\delta }{2}.
\label{LB_T2}
\end{equation}
Further, 
if we let
$c_1 := \sup_{w \in C_0} \Vert \bar{\zeta}^TDh(w) \Vert_p^{p/q} <
\infty,$ then
\begin{equation*}
  \frac{1}{n}\sum_{i=1}^{n}\left\Vert \bar{\zeta}^TDh(W_i)
  \right\Vert_{p}^{\rho \frac{p}{q}}I\left( W_{i}\in C_{0}\right) < c_1^{\rho},
\end{equation*}
for all $n > N^\prime(\delta).$ As a consequence, if
$n\geq N^{\prime }\left( \delta \right) $, it follows from
\eqref{Inter-e1ab} and \eqref{LB_T2} that
\begin{align*}
  \sup_{\left\Vert \zeta \right\Vert _{p}>b}\left\{ -\zeta
  ^{T}H_{n}-M_{n}\left( \zeta \right) \right\} 
  &\leq \sup_{\left\Vert \zeta \right\Vert _{p}>b}\left\{ -\zeta
    ^{T}H_{n}- \left( \frac{c \delta\Vert \zeta \Vert_p}{2} -
    (cc_1)^\rho - c \varepsilon'\Vert \zeta \Vert_p\right) \right\} \\
  &\leq \sup_{\left\Vert \zeta \right\Vert _{p}>b}\left\{ -\zeta
    ^{T}H_{n}-\left\Vert \zeta \right\Vert _{p}\left\{ c
    \left(\frac{\delta }{2} - \varepsilon' \right)-\frac{
    (cc_1)^\rho}{b}\right\} \right\}.
\end{align*}
Consequently, on the set
$\left\Vert H_{n}\right\Vert _{q}\leq b^{\prime }$, we obtain
\begin{equation*}
  \sup_{\left\Vert \zeta \right\Vert _{p}>b}\left\{ -\zeta
    ^{T}H_{n}-M_{n}\left( \zeta \right) \right\} \leq \sup_{\left\Vert \zeta
    \right\Vert _{p}>b}\left\Vert \zeta \right\Vert _{p}\left[ b^{\prime
    }- \left\{ c
      \left(\frac{\delta }{2} - \varepsilon' \right)-\frac{
        (cc_1)^\rho}{b}\right\}\right] .
\end{equation*}
Now, if we take $c > 4(b^\prime + 1)/\delta,$
$\varepsilon^\prime = \delta/4$ and $b$ to be large enough such that
$b > (c c_1)^\rho$ then
\begin{align*}
  b^{\prime }- \left\{ c\left(\frac{\delta }{2} - \varepsilon'
  \right)-\frac{(cc_1)^\rho}{b}\right\} < 0.
\end{align*}
Therefore, if $n\geq n_{0}$ (see \eqref{e1_b}), then
\begin{eqnarray*}
\Pr\left( \max_{\left\Vert \zeta \right\Vert _{p}>b}\left\{ -\zeta
^{T}H_{n}-M_{n}\left( \zeta \right) \right\} >0\right) \leq \Pr\left(
\left\Vert H_{n}\right\Vert _{q}>b^{\prime }\right) +\Pr\left( N^{\prime
}\left( \delta \right) >n\right) .
\end{eqnarray*}%
The result now follows immediately from the previous inequality by
choosing $%
b^{\prime }$
large enough so that
$\Pr( \Vert H_{n}\Vert _{q}>b^{\prime }) \leq \varepsilon /2$ and later
$n_{0}$ so that $\Pr( N^{\prime }( \delta ) >n_{0}) \leq \varepsilon
/2$. The selection 
of $b^{\prime }$ is feasible due to A2). This proves the statement of
Lemma \ref{Lem_Claim_1}. 
\end{proof}

\bigskip

\begin{lemma}
\label{Lem_Claim_2}For any $b>0$ and $c_{0}\in \left(0,\infty \right),$
\begin{eqnarray*}
\frac{1}{n}\sum_{i=1}^{n}\left\Vert {\zeta }^{T}{D}h\left( W_{i}\right)
\right\Vert _{p}^{\rho /(\rho -1)}I\big( \left\Vert W_{i}
\right\Vert _{p}\leq c_{0}\big)
\rightarrow \E\left[ \left\Vert {\zeta }^{T}{D}h\left( W\right)
\right\Vert _{p}^{\rho /(\rho -1)}I( \left\Vert  W
\right\Vert _{p}\leq c_{0}) \right],
\end{eqnarray*}
uniformly over $\left\Vert \zeta \right\Vert _{p}\leq b$ in
probability as $n\rightarrow \infty.$
\end{lemma}

\begin{proof}[Proof of Lemma \protect\ref{Lem_Claim_2}]
We first argue a suitable Lipschitz property for the map
  ${\zeta \hookrightarrow }\left\Vert {\zeta }^{T}{D}h\left(
      W_{i}\right) \right\Vert _{p}^{\rho /(\rho -1)}$.
  It is elementary that for any $0\leq a_{0}<a_{1}$ and $\gamma >1$
\begin{equation*}
a_{1}^{\gamma }-a_{0}^{\gamma }=\gamma \int_{a_{0}}^{a_{1}}t^{\gamma
-1}dt\leq \gamma a_{1}^{\gamma -1}\left( a_{1}-a_{0}\right) .
\end{equation*}%
Applying this observation with 
\begin{eqnarray*}
a_{1} &=&\max \left( \left\Vert {\zeta }_{1}^{T}{D}h\left( W_{i}\right)
\right\Vert_{p},\left\Vert {\zeta }_{0}^{T}{D}h\left( W_{i}\right)
\right\Vert_{p}\right) , \\
a_{0} &=&\min \left( \left\Vert {\zeta }_{1}^{T}{D}h\left( W_{i}\right)
\right\Vert_{p},\left\Vert {\zeta }_{0}^{T}{D}h\left( W_{i}\right)
\right\Vert_{p}\right) , \\
\gamma &=&\rho /(\rho-1),
\end{eqnarray*}%
and using that 
$\left\Vert {\zeta }^{T}{D}h\left( W_{i}\right) \right\Vert _{p}\leq
b\left\Vert {D}h\left( W_{i}\right)\right\Vert_{p}$
for $\left\Vert \zeta \right\Vert _{p}\leq b,$ we obtain
\begin{eqnarray*}
\left\vert \left\Vert {\zeta }_{0}^{T}{D}h\left( W_{i}\right) \right\Vert
_{p}^{\rho /(\rho -1)}-\left\Vert {\zeta }_{1}^{T}{D}h\left( W_{i}\right)
\right\Vert_{p}^{\rho /(\rho -1)}\right\vert
\leq \frac{\rho}{\rho-1} b^{1 /(\rho-1)}\left\Vert {D}h\left( W_{i}\right) \right\Vert
_{p}^{\rho /(\rho -1)}\left\Vert {\zeta }_{0}-{\zeta }_{1}\right\Vert _{p}.
\end{eqnarray*}%
Consequently, we have that 
\begin{eqnarray*}
  \left\vert \frac{1}{n}\sum_{i=1}^{n}\left\Vert {\zeta }_{0}^{T}{D}h\left(
  W_{i}\right) \right\Vert _{p}^{\frac{\rho}{\rho -1}}-\frac{1}{n}%
  \sum_{i=1}^{n}\left\Vert {\zeta }_{1}^{T}{D}h\left( W_{i}\right)
  \right \Vert_{p}^{\frac{\rho}{\rho -1}}\right\vert 
  \leq \frac{\rho}{\rho-1}\left\Vert {\zeta }_{0}-{\zeta
  }_{1}\right\Vert_{p}\frac{b^{\frac{1}{\rho-1}}}{n}\sum_{i=1}^{n}\left\Vert
  {D}h\left( W_{i}\right) \right\Vert_{p}^{\frac{\rho}{\rho -1}}.
\end{eqnarray*}
Since $Dh(\cdot)$ is continuous,
$\E\left[ \left\Vert {D}h\left( W\right) \right\Vert _{p}^{\rho /(\rho
    -1)}I( \left\Vert W \right\Vert _{p}\leq c_{0}) \right]
<\infty,$
thus yielding the tightness of
\begin{align*}
  \frac{1}{n} \sum_{i=1}^n \Vert \zeta^T Dh(W_i)
  \Vert_p^{\rho/(\rho-1)}I \left(\Vert W_{i}
  \right\Vert _{p}\leq c_{0}), 
\end{align*}
under the uniform topology on compact sets. The Strong Law of Large
Numbers guarantees that finite dimensional distributions converge (for
any choice of $\zeta_1, \ldots, \zeta_k, k \geq 1$), and, since the
limit is deterministic, we obtain the desired convergence in
probability.
\end{proof}
\bigskip
\begin{proof}[\textbf{Proof of Theorem \protect\ref{Thm-WPF_RWPI}}] 
  Let us first observe that $R_n(\theta_\ast) \geq 0$ (choosing
  $\zeta =0$).  Then, as a consequence of Lemma \ref{Lem_Claim_1},
  there exists $b > 0$ such that the event 
\begin{equation}
\mathcal{A}_{n}= \left\{n^{\rho /2}R_{n}(\theta _{\ast })=\max_{\left\Vert \zeta
\right\Vert _{p}\leq b}\left\{ -\zeta ^{T}H_{n}-M_{n}\left( \zeta \right)
\right\}  \right\}, \label{Def_A_n_E}
\end{equation}
where the outer supremum is attained at some
$\Vert \zeta_\ast \Vert_p \leq b,$ occurs with probability at least
$1-\varepsilon,$ as long as $n \geq n_0.$ In other words,
$\Pr(\mathcal{A}_n) \geq 1-\varepsilon$ when $n \geq n_0.$

\textit{We first consider the case }$\rho >1$\textit{.} For
$\zeta \neq 0$, write
$\bar{\zeta}=\zeta /\left\Vert \zeta \right\Vert _{p}.$ Next, define
the vector $V_i(\bar{\zeta})$ via
$V_i\left( \bar{\zeta}\right) = {D}h\left(
  W_{i}\right)^{T}\bar{\zeta}$
(that is, the $j$-th entry of $V_i\left( \bar{\zeta}\right) $ is the
$j$-th entry of the vector ${D}h\left( W_{i}\right) ^{T}\bar{\zeta}$),
and put
\begin{equation}
\Delta _{i}^{\prime }=\Delta _{i}^{\prime }\left( \bar{\zeta}\right)
=\left\vert V_{i}\left( \bar{\zeta}\right) \right\vert ^{p/q}sgn\left(
V_{i}\left( \bar{\zeta}\right) \right) .  \label{Def_Del_p_INT}
\end{equation}
Next, let $\bar{\Delta}_{i}=c_i\Delta _{i}^{\prime }$ with
$c_i$ chosen so that
\begin{equation*}
\left\Vert \bar{\Delta}_{i}\right\Vert _{q}= \left(\frac{1}{\rho}\left\Vert \zeta^{T}{D}h\left( W_{i}\right) \right\Vert _{p}\right)^{1/(\rho -1)}.
\end{equation*}
In such case we have that
\begin{align}
\max_{\Delta }\left\{ {\zeta }^{T}{D}h\left( W_{i}\right) \Delta
-\left\Vert \Delta \right\Vert _{q}^{\rho }\right\}  
&=\max_{\left\Vert \Delta \right\Vert _{q}\geq 0}\left\{ \left\Vert {\zeta }
^{T}{D}h\left( W_{i}\right) \right\Vert _{p}\left\Vert \Delta \right\Vert
_{q}-\left\Vert \Delta \right\Vert _{q}^{\rho }\right\}  \notag \\
&={\zeta }^{T}{D}h\left( W_{i}\right) \bar{\Delta}_{i}-\left\Vert
  \bar{\Delta}_{i}\right\Vert _{q}^{\rho }  \notag \\ 
&=\left\Vert {\zeta }^{T}{D}h\left( W_{i}\right) \right\Vert _{p}^{\rho
/(\rho -1)}\left( \frac{1}{\rho }\right) ^{1/\left( \rho -1\right)
    }\left( 1-\frac{1}{\rho }\right) .  \label{E00} 
\end{align}
Pick $c_{0}\in \left( 0,\infty \right) $ and define $C_{0}=\{\left\Vert
W_{i}\right\Vert _{p}\leq c_{0}\}$. Note that 
\begin{equation*}
M_{n}\left( \zeta \right) \geq M_{n}^{\prime }\left( \zeta ,c_{0}\right) ,
\end{equation*}%
where%
\begin{equation*}
M_{n}^{\prime }\left( \zeta ,c_{0}\right) =\frac{1}{n}\sum_{i=1}^{n}I\left(
W_{i}\in C_{0}\right) \left\{ {\zeta ^{T}}\int_{0}^{1}{D}h\left(
W_{i}+n_{i}^{-1/2}\bar{\Delta}_{i}u\right) \bar{\Delta}_{i}du-\left\Vert 
\bar{\Delta}_{i}\right\Vert _{q}^{\rho }\right\} ^{+}.
\end{equation*}%
Therefore 
\begin{equation}
\max_{\left\Vert \zeta \right\Vert _{p}\leq b}\left\{ -\zeta
^{T}H_{n}-M_{n}\left( \zeta \right) \right\} \leq \max_{\left\Vert \zeta
\right\Vert _{p}\leq b}\left\{ -\zeta ^{T}H_{n}-M_{n}^{\prime }\left( \zeta
,c_{0}\right) \right\} .  \label{B_N_BC_2}
\end{equation}
Define
\begin{eqnarray*}
\widehat{M}_{n}\left( \zeta ,c_{0}\right) &=&\frac{1}{n}\sum_{i=1}^{n}I%
\left( W_{i}\in C_{0}\right) \left\{ {\zeta ^{T}D}h\left( W_{i}\right) \bar{%
\Delta}_{i}du-\left\Vert \bar{\Delta}_{i}\right\Vert _{q}^{\rho }\right\}
^{+} \\
&=&\frac{1}{n}\sum_{i=1}^{n}I\left( W_{i}\in C_{0}\right) \left\Vert {\zeta }%
^{T}{D}h\left( W_{i}\right) \right\Vert _{p}^{\rho /(\rho -1)}\left( \frac{1%
}{\rho }\right) ^{1/\left( \rho -1\right) }\left( 1-\frac{1}{\rho }\right) ,
\end{eqnarray*}%
where the equality follows from (\ref{E00}). We then claim that 
\begin{equation}
\sup_{\left\Vert \zeta \right\Vert _{q}\leq b}\left\vert \widehat{M}%
_{n}\left( \zeta ,c_{0}\right) -M_{n}^{\prime }\left( \zeta ,c_{0}\right)
\right\vert \rightarrow 0.  \label{Aux_B_T}
\end{equation}%
In order to verify (\ref{Aux_B_T}), note, using the continuity of $Dh\left(
\cdot \right),$ that for any $\varepsilon ^{\prime }>0$ there exists $n_{0}$
such that if $n\geq n_{0}$ then (uniformly over $\left\Vert \zeta
\right\Vert _{p}\leq b$), 
\begin{equation*}
\left\vert \int_{0}^{1}I\left( W_{i}\in C_{0}\right) \left\Vert
    \zeta^T \left[{D}h(W_{i}+n^{-1/2}\bar{\Delta}_{i}u) -{D}h( W_{i})
\right] \right\Vert_{p}\left\Vert \bar{\Delta}_{i} \right\Vert
_{q}du\right\vert \leq \varepsilon ^{\prime }.
\end{equation*}
Therefore, if $n\geq n_{0}$, 
\begin{equation*}
\frac{1}{n}\sum_{i=1}^{n}I\left( W_{i}\in C_{0}\right) \left\vert {\zeta ^{T}%
}\int_{0}^{1}\left[ {D}h( W_{i}+n^{-1/2}\bar{\Delta}_{i}u) -{D%
}h( W_{i}) \right] \bar{\Delta}_{i}du\right\vert \leq \varepsilon^{\prime }.
\end{equation*}
Since $\varepsilon ^{\prime }>0$ is arbitrary, (\ref{Aux_B_T}) stands
verified. Then, applying Lemma \ref{Lem_Claim_2} we obtain
\begin{equation*}
  \widehat{M}_{n}\left( \zeta ,c_{0}\right) \rightarrow \E\left( {\zeta ^{T}D}%
    h\left( W_{i}\right) \bar{\Delta}_{i}du-\left\Vert \bar{\Delta}%
      _{i}\right\Vert _{q}^{\rho }\right) ^{+} I\left( W_i \in
    C_0\right), 
\end{equation*}%
uniformly over $\left\Vert \zeta \right\Vert _{p}\leq b$ as $n\rightarrow
\infty $, in probability. Therefore, applying the continuous mapping
principle, we have that 
\begin{align}
&\max_{\left\Vert \zeta \right\Vert _{p}\leq b}\left\{ -\zeta
^{T}H_{n}-M_{n}^{\prime }\left( \zeta ,c_{0}\right) \right\}  \notag \\
&\quad\Rightarrow \max_{\left\Vert \zeta \right\Vert _{p}\leq b}\left\{ -\zeta
^{T}H-\kappa \left( \rho \right) \E\left[ \left\Vert {\zeta }^{T}{D}h\left(
W\right) \right\Vert _{p}^{\rho /(\rho -1)}I\left( \left\Vert 
W \right\Vert_{p}\leq c_{0}\right) \right] \right\} ,  \label{rb}
\end{align}
as $n\rightarrow \infty $, where 
\begin{equation*}
\kappa \left( \rho \right) =\left( \frac{1}{\rho }\right) ^{1/\left( \rho
-1\right) }\left( 1-\frac{1}{\rho }\right) ,
\end{equation*}%
and $H\sim \mathcal{N}\left( 0,Cov[h\left( W,\theta _{\ast }\right)
  ]\right) $. From (\ref{B_N_BC_2}) and the construction of
(\ref{Def_A_n_E}), we can easily 
obtain that $n^{\rho /2}R_{n}\left( \theta _{\ast }\right) $ is
stochastically bounded (asymptotically)\ by
\begin{equation*}
\max_{\zeta }\left\{ -\zeta ^{T}H-\kappa \left( \rho \right) \E\left[
\left\Vert {\zeta }^{T}{D}h\left( W\right) \right\Vert _{p}^{\rho /(\rho
-1)}\right] \right\},
\end{equation*}
which verifies the first part of the theorem when $\rho >1.$

\textit{Now, for }$\rho =1$\textit{, we will follow very similar steps.}
Again, due to Lemma \ref{Lem_Claim_1} we concentrate on the region $%
\left\Vert \zeta \right\Vert _{p}\leq b$ for some $b>0$. For the upper
bound, define $\Delta _{i}^{\prime }$ as in (\ref{Def_Del_p_INT}). Using a
localization technique similar to that described in the proof of Lemma
\ref{Lem_Claim_1} in which the set $C_{0}$ as introduced we might
assume that $\left\Vert W_i \right\Vert _{p}\leq c_{0}$ for some
$c_{0}>0 $. Then, for a given constant $c>0$, setting
${\Delta}_{i}=c\Delta_{i}^{\prime },$ we obtain that  
\begin{align*}
  &\max_{\left\Vert \zeta \right\Vert_p \leq b}\left\{ -\zeta
    ^{T}H_{n}-\frac{1}{n}\sum_{i=1}^{n}\sup_{{\Delta}_{i}}\left\{ \zeta
    ^{T}\int_{0}^{1}{D} h( W_{i}+\Delta _{i}u/n^{1/2}) \Delta_{i}du- \left\Vert
    \Delta_{i}\right\Vert _{q}\right\} \right\} \\ 
  &\quad\leq \max_{\left\Vert \zeta \right\Vert_p \leq b}\left\{ -\zeta
    ^{T}H_{n}- \frac{1}{n}\sum_{i=1}^{n}\left( c\zeta ^{T}\int_{0}^{1}{D}h(
    W_{i}+c\Delta _{i}^{\prime }u/n^{1/2}) \Delta _{i}^{\prime
    }du-c\left\Vert \Delta _{i}^{\prime }\right\Vert _{q}\right) I\left(
    W_i \in C_0\right)\right\} .
\end{align*}
As in the case $\rho >1$ we have that%
\begin{equation*}
  \frac{1}{n}\sum_{i=1}^{n} I(W_i \in C_0)\int_{0}^{1}\zeta ^{T}\left[
    {D}h(W_{i}+ c\Delta _{i}^{\prime }u/n^{1/2}) -{D}h( W_{i}) \right]
  \Delta _{i}^{\prime }du\rightarrow 0 
\end{equation*}
in probability uniformly on $\zeta $-compact sets. Similarly, in
addition, for any $c>0$ and any $b>0$%
\begin{align*}
  &\max_{\left\Vert \zeta \right\Vert_p \leq b}\left\{ -\zeta ^{T}H_{n}-
    \frac{1}{n}\sum_{i=1}^{n}\left( c\zeta ^{T}{D}h(W_{i}) \Delta _{i}^{\prime
    }du-c\left\Vert \Delta _{i}^{\prime }\right\Vert _{q}\right) I\left(
    W_i \in C_0\right)\right\}\\
  &\quad=\max_{\left\Vert \zeta \right\Vert \leq b}\left\{ -\zeta
    ^{T}H_{n}-\frac{1}{n}\sum_{i=1}^{n} c \left( \left\Vert \zeta
    ^{T}{D}h\left( W\right) \right\Vert _{p}-1\right) ^{+} \Vert
    \Delta_i^\prime \Vert_q I( \Vert W_i \Vert_p \leq c_0) \right\} \\ 
  &\quad\Rightarrow \max_{\left\Vert \zeta \right\Vert \leq b}\left\{ -\zeta
    ^{T}H-c\E \left[\left( \left\Vert \zeta ^{T}{D}h\left( W\right) \right\Vert
    _{p}-1\right) ^{+} \Vert \bar{\zeta}^T Dh(W)\Vert_p^{p/q} I(\Vert
    W \Vert_p \leq c_0)\right]\right\},
\end{align*}
because
$\Vert \Delta^\prime \Vert_q^q = \Vert \bar{\zeta}^T
Dh(W_i)\Vert_p^p.$
Next, as the constant $c$ can be arbitrarily large, we obtain a
stochastic upper bound of the form
\begin{equation*}
\max_{\left\Vert \zeta \right\Vert \leq b:\Pr\left( \left\Vert \zeta ^{T}{D}%
h\left( W\right) \right\Vert _{p}\leq 1\right) =1}\left\{ -\zeta
^{T}H\right\} \leq \max_{\zeta :\Pr\left( \left\Vert \zeta ^{T}{D}h\left(
W\right) \right\Vert _{p}\leq 1\right) =1}\left\{ -\zeta ^{T}H\right\} .
\end{equation*}
This completes the proof of Theorem \ref{Thm-WPF_RWPI}. 
\end{proof}

\bigskip

\begin{proof}[Proof of Proposition \protect\ref{Prop_SLB_rho_1}]
  We follow the notation introduced in the proof of Theorem
  \ref{Thm-WPF_RWPI}. Recall from \eqref{Scaled-RWP-Dual-Rep} and
  \eqref{Mn-Defn} that 
\begin{equation*}
n^{1/2}R_{n}\left( \theta _{\ast }\right) =\sup_{\zeta } \left\{\zeta ^{T}H_{n}-
\frac{1}{n}\sum_{k=1}^{n}\sup_{\Delta } \left\{\int_{0}^{1}\zeta ^{T}Dh\left(
W_{i}+\Delta u/n^{1/2}\right) \Delta du-\left\Vert \Delta \right\Vert
_{q} \right\} \right\}.
\end{equation*}
Let $A :=\{\zeta :$ \
ess$\sup \left\Vert \zeta ^{T}Dh\left( w\right) \right\Vert _{p}\leq
1\}$,
where the essential supremum is taken with respect to the Lebesgue
measure. Then, due to H\"{o}lder's inequality, if $\zeta \in A$,
\begin{align*}
&\sup_{\Delta } \left\{\int_{0}^{1}\zeta ^{T}Dh\left( W_{i}+\Delta
u/n^{1/2}\right) \Delta du-\left\Vert \Delta \right\Vert _{q} \right\} \\
&\quad\leq \sup_{\Delta } \left\{\int_{0}^{1}\left\Vert \zeta ^{T}Dh\left(
W_{i}+\Delta u/n^{1/2}\right) \right\Vert _{p}\left\Vert \Delta \right\Vert
_{q}du-\left\Vert \Delta \right\Vert _{q} \right\} \\
&\quad\leq \sup_{\Delta }\left\Vert \Delta \right\Vert _{q} \left\{\int_{0}^{1}\left(
\left\Vert \zeta ^{T}Dh\left( W_{i}+\Delta u/n^{1/2}\right) \right\Vert
_{p}-1\right) du \right\}\leq 0.
\end{align*}
Consequently, 
\begin{equation*}
n^{1/2}R_{n}\left( \theta _{\ast }\right) \geq \sup_{\zeta \in A}\zeta
^{T}H_{n}.
\end{equation*}%
Letting $n\rightarrow \infty $ we conclude that 
\begin{equation*}
\sup_{\zeta \in A}\zeta ^{T}H_{n}\Rightarrow \sup_{\zeta \in A}\zeta ^{T}H.
\end{equation*}%
Because $W_{i}$ is assumed to have a density with respect to the Lebesgue
measure it follows that $\Pr\left( \left\Vert \zeta ^{T}Dh\left( W_{i}\right)
\right\Vert _{p}\leq 1\right) =1$ if and only if $\zeta \in A$ and the
result follows.
\end{proof}

\bigskip

\noindent 
Finally, we provide the proof of Proposition \ref{Prop_SLB_rho_L_1}.

\begin{proof}[Proof of Proposition \protect\ref{Prop_SLB_rho_L_1}]
  Recall from \eqref{Scaled-RWP-Dual-Rep} and \eqref{Mn-Defn} that
\begin{equation}
n^{1/2}R_{n}\left( \theta _{\ast }\right) =\sup_{\zeta } \left\{\zeta ^{T}H_{n}-
\frac{1}{n}\sum_{k=1}^{n}\sup_{\Delta } \left\{\int_{0}^{1}\zeta ^{T}Dh\left(
W_{i}+\Delta u/n^{1/2}\right) \Delta du-\left\Vert \Delta \right\Vert
_{q}^\rho \right\} \right\}. 
\label{RWP-Repn}
\end{equation}
As in the proof of Theorem \ref{Thm-WPF_RWPI}, due to Lemma
\ref{Lem_Claim_1}, we might assume that
$\left\Vert \zeta \right\Vert _{p}\leq b$ for some $b>0$. 

The strategy will be to split the inner supremum in values of
$\left\Vert \Delta \right\Vert _{q}\leq \delta n^{1/2}$ and values
$\left\Vert \Delta \right\Vert _{q}>\delta n^{1/2}$ for a suitably
small positive constant $\delta$. In Step 1, we shall show that the
supremum is achieved with high probability in the former region. Then,
in Step 2, we analyze the region in which
$\left\Vert \Delta \right\Vert _{q}\leq \delta n^{1/2}$ and argue that
the integrals inside the summation in \eqref{RWP-Repn} can be replaced
by $\zeta ^{T}Dh\left( W_{i}\right) \Delta $. Once this substitution
is performed we can solve the inner maximization problem explicitly in
Step 3 and, finally, we will apply a weak convergence result on
$\zeta $%
-compact sets to conclude the result. We now proceed to execute this
strategy.

\noindent
\textit{Execution of Step 1:} Pick $\delta >0$ small, to be chosen in the
sequel, then note that A5) implies (by redefining $\kappa $ if needed,
due to the continuity of $%
Dh\left( \cdot \right) $) that
\begin{equation*}
\left\Vert Dh\left( w\right) \right\Vert _{p}\leq \kappa \left( 1+\left\Vert
w\right\Vert _{q}^{\rho -1}\right) .
\end{equation*}
Therefore, for $\zeta$ such that $\Vert \zeta \Vert_p \leq b,$
\begin{align*}
  &\sup_{\left\Vert \Delta \right\Vert _{q}\geq \delta
    n^{1/2}} \left\{\int_{0}^{1}\left\vert \zeta ^{T}Dh\left( W_{i}+\Delta
    u/n^{1/2}\right) \Delta \right\vert du-\left\Vert \Delta \right\Vert
    _{q}^{\rho } \right\} \\
  &\quad\leq \sup_{\left\Vert \Delta \right\Vert _{q}\geq \delta
    n^{1/2}} \left\{b\kappa
    \left( 1+\int_{0}^{1}\left\Vert W_{i}+\Delta u/n^{1/2}\right\Vert _{q}^{\rho
    -1}du\right) \left\Vert \Delta \right\Vert _{q}-\left\Vert \Delta
    \right\Vert _{q}^{\rho } \right\}.
\end{align*}
Note that if $\rho \in (1,2)$, then $0<\rho -1<1,$ and therefore by the
triangle inequality and concavity 
\begin{equation*}
\left\Vert W_{i}+\Delta u/n^{1/2}\right\Vert _{q}^{\rho -1}\leq \left(
\left\Vert W_{i}\right\Vert _{q}+\left\Vert \Delta /n^{1/2}\right\Vert
_{q}\right) ^{\rho -1}\leq \left\Vert W_{i}\right\Vert _{q}^{\rho
-1}+\left\Vert \Delta /n^{1/2}\right\Vert _{q}^{\rho -1}.
\end{equation*}
On the other hand, if $\rho \geq 2$, then $\rho -1\geq 1$ and the triangle
inequality combined with Jensen's inequality applied as follows: 
\begin{equation*}
  \left\Vert a+c\right\Vert ^{\rho -1}\leq 2^{\rho -1} \left( \frac{1}{2}\left\Vert
      a \right \Vert^{\rho-1}+\frac{1}{2} \left \Vert c\right\Vert^{\rho -1}
  \right) = 2^{\rho -2}\left( \left\Vert a\right\Vert
    ^{\rho -1}+\left\Vert c\right\Vert ^{\rho -1}\right) ,
\end{equation*}
yields
\begin{equation*}
\left\Vert W_{i}+\Delta u/n^{1/2}\right\Vert _{q}^{\rho -1}\leq 2^{\rho
-2}\left( \left\Vert W_{i}\right\Vert _{q}^{\rho -1}+\left\Vert \Delta
/n^{1/2}\right\Vert _{q}^{\rho -1}\right) .
\end{equation*}
So, in both cases we can write
\begin{align*}
&\sup_{\left\Vert \Delta \right\Vert _{q}\geq \delta
n^{1/2}} \left\{\int_{0}^{1}\left\vert \zeta ^{T}Dh( W_{i}+\Delta
u/n^{1/2}) \Delta \right\vert du-\left\Vert \Delta \right\Vert
_{q}^{\rho } \right\} \\
&\quad\leq \sup_{\left\Vert \Delta \right\Vert _{q}\geq \delta
  n^{1/2}} \left \{b\kappa
\left( 1+2^{\rho -1}\left( \left\Vert W_{i}\right\Vert _{q}^{\rho
-1}+\left\Vert \Delta /n^{1/2} \right\Vert _{q}^{\rho -1}\right) \right)
\left\Vert \Delta \right\Vert _{q}-\left\Vert \Delta \right\Vert _{q}^{\rho
} \right\} \\
&\quad \leq \sup_{\left\Vert \Delta \right\Vert _{q}\geq \delta
  n^{1/2}} \left \{b\kappa
\left( \left\Vert \Delta \right\Vert _{q}+2^{\rho -1}\left\Vert
W_{i}\right\Vert _{q}^{\rho -1}\left\Vert \Delta \right\Vert _{q}+2^{\rho
-1}\left\Vert \Delta \right\Vert _{q}^{\rho }/n^{(\rho-1)/2}\right) -\left\Vert
\Delta \right\Vert _{q}^{\rho } \right\}.
\end{align*}
Next, as $\E\Vert W_n \Vert^\rho < \infty,$ we have that for any
$\varepsilon ^{\prime }>0$,
\begin{equation*}
\Pr\left( \left\Vert W_{n}\right\Vert _{q}^{\rho }\geq \varepsilon ^{\prime }n%
\text{ i.o.}\right) =0,
\end{equation*}%
therefore we might assume that there exists $n_{0}$ such that for all $i\leq
n$ and $n\geq n_{0}$, $\left\Vert W_{i}\right\Vert _{q}^{\rho -1}\leq \left(
\varepsilon ^{\prime }n\right) ^{(\rho -1)/\rho }$. Therefore, if $\left(
\varepsilon ^{\prime }\right) ^{(\rho -1)/\rho }\leq \delta ^{\rho
-1}/\left( b\kappa 2^{\rho}\right) $, we conclude that if $\left\Vert
\Delta \right\Vert _{q}\geq \delta n^{1/2}$ and $n > n_0,$ 
\begin{eqnarray*}
b\kappa 2^{\rho -1}\left\Vert W_{i}\right\Vert _{q}^{\rho -1}\left\Vert
\Delta \right\Vert _{q} &\leq &b\kappa 2^{\rho -1}\left( \varepsilon
^{\prime }n\right) ^{\left( \rho -1\right) /\rho }\left\Vert \Delta
\right\Vert _{q} \\
&\leq &\frac{1}{2}\delta ^{\rho -1}n^{\left( \rho -1\right) /\rho}\left\Vert
\Delta \right\Vert _{q}\leq \frac{1}{2}\left\Vert \Delta \right\Vert
_{q}^{\rho }.
\end{eqnarray*}%
Similarly, choosing $n$ sufficiently large we can guarantee that 
\begin{equation*}
b\kappa \left( \left\Vert \Delta \right\Vert _{q}+2^{\rho -1}\left\Vert
\Delta \right\Vert _{q}^{\rho }/n^{ (\rho - 1)/\rho}\right) \leq \frac{1}{2}\left\Vert
\Delta \right\Vert _{q}^{\rho }.
\end{equation*}%
Therefore, we conclude that for any fixed $\delta >0,$
\begin{equation}
  \sup_{\left\Vert \Delta \right\Vert _{q}\geq \delta \sqrt{n}
  } \left\{\int_{0}^{1}\left\vert \zeta ^{T}Dh( W_{i}+\Delta u/n^{1/2})
      \Delta \right\vert du-\left\Vert \Delta \right\Vert _{q}^{\rho }
  \right\}\leq 0
\label{Former-Region}
\end{equation}
provided $n$ is large enough, thus
achieving the desired result over the region
$\Vert \Delta \Vert_q \geq \delta \sqrt{n}.$

\noindent
\textit{Execution of Step 2:} Next, we let
$\varepsilon ^{\prime \prime }>0$, and note that
\begin{align}
  &\sup_{\left\Vert \Delta \right\Vert _{q}\leq \delta \sqrt{n}%
    } \left\{\int_{0}^{1}\zeta ^{T}Dh( W_{i}+\Delta u/n^{1/2}) \Delta
    du-\left\Vert \Delta \right\Vert _{q}^{\rho } \right\} \label{Inter-1}\\
  &\quad\leq \sup_{\left\Vert \Delta \right\Vert _{q}\leq \delta \sqrt{n}%
    } \left\{\int_{0}^{1}\zeta ^{T}\left[ Dh( W_{i}+\Delta u/n^{1/2})
    -Dh( W_{i}) \right] \Delta du-\varepsilon ^{\prime \prime
    }\left\Vert \Delta \right\Vert _{q}^{\rho } \right\} \notag\\
  &\quad\quad\quad\quad
    +\sup_{\left\Vert \Delta \right\Vert _{q}\leq \delta \sqrt{n}} \left\{\zeta
    ^{T}Dh\left( W_{i}\right) \Delta -(1-\varepsilon ^{\prime \prime
    })\left\Vert \Delta \right\Vert _{q}^{\rho } \right\}. \notag
\end{align}
We now argue locally, using A6), a bound for the first term in the
right hand side of \eqref{Inter-1}:
\begin{align}
&\sup_{\left\Vert \Delta \right\Vert _{q}\leq \delta \sqrt{n}%
} \left\{\int_{0}^{1}\zeta ^{T}\left[ Dh( W_{i}+\Delta u/n^{1/2})
-Dh( W_{i}) \right] \Delta du-\varepsilon ^{\prime \prime
}\left\Vert \Delta \right\Vert _{q}^{\rho } \right\}  \\
&\quad\quad\quad\quad\leq \sup_{\left\Vert \Delta \right\Vert _{q}\leq
  \delta \sqrt{n}} \left\{ \Vert \zeta \Vert_p \bar{\kappa}\left( W_{i}\right) \left\Vert \Delta \right\Vert
_{q}^{2}/n^{1/2}-\varepsilon ^{\prime \prime }\left\Vert \Delta \right\Vert
_{q}^{\rho } \right\}  \notag \\
&\quad\quad\quad\quad\leq \sup_{\left\Vert \bar{\Delta}\right\Vert
  _{q}\leq 1} \left\{b\bar{\kappa}\left(
W_{i}\right) \left\Vert \bar{\Delta}\right\Vert _{q}^{2}\delta
^{2}n^{1/2}-\varepsilon ^{\prime \prime }\left\Vert \bar{\Delta}\right\Vert
_{q}^{\rho }\left( \delta n^{1/2}\right) ^{\rho } \right\}.  \notag 
\end{align}
As
$ \sup_{x \in [0,1]} \left\{ a_n x^2 - b_n x^\rho \right\} \leq
(\rho-2)^+ (a_n^\rho/b_n^2)^{1/(\rho-2)}/\rho$
when $b_n > a_n,$ we have, for all $n$ sufficiently large, that 
\begin{align*}
\sup_{\left\Vert \Delta \right\Vert _{q}\leq \delta \sqrt{n}%
} \left\{\int_{0}^{1}\zeta ^{T}\left[ Dh( W_{i}+\Delta u/n^{1/2})
-Dh( W_{i}) \right] \Delta du-\varepsilon ^{\prime \prime
}\left\Vert \Delta \right\Vert _{q}^{\rho } \right\}  \leq
  \frac{(\rho-2)^+}{\rho}
  \left(\frac{b\bar{\kappa}(W_i)}{\varepsilon^{\prime\prime}\sqrt{n}}\right)^{\rho/(\rho-2)}. 
\end{align*}
Since $\E[\bar{\kappa}(W)^2] < \infty$ (from Assumption A6)), we have
that
$\Pr(\bar{\kappa}(W_i) > \varepsilon^{\prime \prime \prime} \sqrt{i}
\text{ i.o.}) = 0$ for any $\varepsilon^{\prime \prime \prime} >
0.$ 
Consecutively, $\bar{\kappa}(W_i) < \varepsilon^{\prime \prime \prime}
\sqrt{i}$  for all $i$ large enough, and therefore, 
\begin{align*}
  &\overline{\lim }_{n\rightarrow \infty
    }\frac{1}{n}\sum_{i=1}^{n}\sup_{\left\Vert \Delta \right\Vert
    _{q}\leq \delta \sqrt{n}} \left\{\int_{0}^{1}\zeta 
    ^{T}\left[ Dh\left( W_{i}+\Delta u/n^{1/2}\right) -Dh\left( W_{i}\right) %
    \right] \Delta du-\varepsilon ^{\prime \prime }\left\Vert \Delta \right\Vert
    _{q}^{\rho } \right\} \\
  &\quad\quad\quad\quad \leq \frac{(\rho-2)^+}{\rho} \overline{\lim
    }_{n\rightarrow \infty}  \left(\frac{b}{\varepsilon^{\prime \prime}}
    \right)^{\rho/(\rho-2)} \frac{1}{n}\sum_{i=1}^n
    \left(\frac{\bar{\kappa}(W_i)}{\sqrt{n}}\right)^{\rho/(\rho-2)}
    \leq \frac{(\rho-2)^+}{\rho} \left(b\frac{\varepsilon^{\prime \prime \prime}}{\varepsilon^{\prime
    \prime}}\right)^{\rho/(\rho-2)},  
\end{align*}
which can be made arbitrarily small by choosing
$\varepsilon^{\prime \prime \prime}$ arbitrarily small. Therefore, for
any fixed $\varepsilon^{\prime \prime}, \delta > 0,$
\begin{align}
  \overline{\lim }_{n\rightarrow \infty
  }\frac{1}{n}\sum_{i=1}^{n}\sup_{\left\Vert \Delta \right\Vert
  _{q}\leq \delta \sqrt{n}} \left\{\int_{0}^{1}\zeta 
  ^{T}\left[ Dh\left( W_{i}+\Delta u/n^{1/2}\right) -Dh\left( W_{i}\right) %
  \right] \Delta du-\varepsilon ^{\prime \prime }\left\Vert \Delta \right\Vert
  _{q}^{\rho } \right\} = 0. 
  \label{Latter-Region}
\end{align}

\noindent
\textit{Execution of Step 3:} Next, it follows from
\eqref{Former-Region}, \eqref{Inter-1} and \eqref{Latter-Region} that
for any fixed $\varepsilon ^{\prime \prime },\delta >0$, there exists
$N_{0}$ such that if $n\geq N_{0}$,
\begin{align*}
  &\frac{1}{n}\sum_{i=1}^n\sup_{\Delta } \left\{\int_{0}^{1}\zeta ^{T}Dh\left( W_{i}+\Delta
    u/n^{1/2}\right) \Delta du-\left\Vert \Delta \right\Vert _{q}^{\rho } \right\}\\
  &\quad\quad\leq \frac{1}{n}\sum_{i=1}^n\sup_{\Delta \leq \delta \sqrt{n}}
    \left\{\zeta ^{T}Dh\left( W_{i}\right) \Delta 
    du-(1-\varepsilon ^{\prime \prime })\left\Vert \Delta \right\Vert _{q}^{\rho
    } \right\}+\delta \\
  &\quad\quad\leq \frac{1}{n}\sum_{i=1}^n\min \left\{ \kappa \left(
    \rho ,\varepsilon ^{\prime \prime}\right) \left\Vert {\zeta
    }^{T}{D}h\left( W_{i}\right) \right\Vert _{p}^{\rho /(\rho 
    -1)},\ c_{n}\right\} + \delta,
\end{align*}
where 
\begin{equation*}
\kappa \left( \rho ,\varepsilon ^{\prime \prime }\right) =\left( \frac{1}{
\rho (1-\varepsilon ^{\prime \prime })}\right) ^{1/\left( \rho -1\right)
}\left( 1-\frac{1}{\rho }\right),
\end{equation*}
and $c_{n}\rightarrow \infty $ as $n\rightarrow \infty $ (the exact value of 
$c_{n}$ is not important).

\noindent 
Next, note that A5) implies that
\begin{equation*}
\left\Vert Dh\left( W_{i}\right) \right\Vert _{p}^{\rho /\left( \rho
    -1 \right) }I\left( \left\Vert W_{i}\right\Vert \geq 1\right) \leq \kappa
I\left( \left\Vert W_{i}\right\Vert \geq 1\right) \left\Vert
W_{i}\right\Vert _{q}^{\rho }\leq \kappa \left\Vert W_{i}\right\Vert
_{q}^{\rho }
\end{equation*}%
and, therefore, since $Dh\left( \cdot \right) $ is continuous (therefore
locally bounded)\ and $\E\left\Vert W_{i}\right\Vert _{q}^{\rho }<\infty $
also by A5), we have 
that
\begin{equation*}
\E\left\Vert Dh\left( W\right) \right\Vert _{p}^{\rho /\left( \rho-1
\right) }<\infty .
\end{equation*}
Then, an argument similar to Lemma \ref{Lem_Claim_2} shows that
\begin{align*}
&\sup_{\left\Vert \zeta \right\Vert_p \leq b} \left\{\zeta ^{T}H_{n}-\frac{1}{n}
\sum_{i=1}^{n}\left\{\kappa \left( \rho ,\varepsilon ^{\prime \prime}\right)
\left\Vert {\zeta }^{T}{D}h\left( W_{i}\right) \right\Vert _{q}^{\rho /(\rho
-1)},c_{n}\right\} \right\} \\
&\quad\quad\quad\quad\Rightarrow \sup_{\left\Vert \zeta \right\Vert_p
  \leq b} \left \{\zeta^{T}H-\kappa \left( \rho ,\varepsilon
  ^{\prime \prime }\right) \E\left\Vert {\zeta }^{T}{D}h\left(
  W_{i}\right) \right\Vert _{q}^{\rho /(\rho -1)} \right\}, 
\end{align*}
as $n\rightarrow \infty $ (where $\Rightarrow $ denotes weak convergence).
Finally, we can send $\varepsilon ^{\prime \prime },\delta \rightarrow 0$
and $b\rightarrow \infty $ to obtain the desired asymptotic stochastic
lower bound. 
\end{proof}

\subsection{Proofs of RWP function limit theorems for linear and logistic
regression examples\label{Sec-RWP-Eg-Proofs}}

We first obtain the dual formulation of the respective RWP functions for
linear and logistic regressions using Proposition \ref{Prop-RWP-Duality}.
Let $\E[h(x,y;\beta)] = \mathbf{0}$ be the estimating equation under
consideration ($h(x,y;\beta) = (y-\beta^Tx)x$ for linear regression and $%
h(x,y;\beta)$ as in \eqref{Est-Eq-RLogR} for logistic regression). Recall
that the cost function is $c(\cdot) = N_q(\cdot).$ Due to the duality result
in Proposition \ref{Prop-RWP-Duality}, we obtain 
\begin{align*}
R_n(\beta_\ast) &= \inf\big\{ D_c(\Pr,\Pr_n): \ \E_\Pr[h(X,Y;\beta_\ast)] = \mathbf{%
0}\big\} \\
&= \sup_{\lambda} \left\{ -\frac{1}{n} \sum_{i=1}^n \sup_{(x^{\prime
},y^{\prime })} \left\{ \lambda^Th(x^{\prime },y^{\prime };\beta_\ast) - N_q%
\big( (x^{\prime },y^{\prime }), (X_i,Y_i)\big)\right\}\right\}.
\end{align*}
As $N_q((x^{\prime },y^{\prime }),(X_i,Y_i)) = \infty$ when $y^{\prime }\neq
Y_i,$ the above expression simplifies to, 
\begin{equation}
R_n(\beta_\ast) = \sup_{\lambda} \left\{ -\frac{1}{n} \sum_{i=1}^n
\sup_{x^{\prime }} \left\{ \lambda^Th(x^{\prime },Y_i;\beta_\ast) -\Vert
x^{\prime }- X_i \Vert_q^\rho\right\}\right\},  \label{RWP-Dual-Lin-Log-Reg}
\end{equation}
where $\rho = 2$ for the case of linear regression (Theorem
\ref{Thm-WPF}) and $\rho = 1$ for the case of logistic regression
(Theorem \ref{Thm-GLM-WPF}%
). As RWP function here is similar to the RWP function for general
estimating equation in Section \ref{Sec_Quad_WPF}, a similar limit
theorem holds. We state here the assumptions for proving RWP limit
theorems for the dual formulation in \eqref{RWP-Dual-Lin-Log-Reg}.


\noindent \textbf{Assumptions:}\newline
\textbf{A2')} Suppose that $\beta_\ast \in \mathbb{R}^d$ satisfies $%
\E[h(X,Y;\beta_\ast)] = \mathbf{0}$ and $\E\Vert h(X,Y;\beta_\ast)\Vert_2^2 <
\infty$ (While we do not assume that $\beta_{\ast }$ is unique, the results
are stated for a fixed $\beta_{\ast }$ satisfying $\E[h(X,Y;\beta_\ast)] = 
\mathbf{0}$.)


\noindent \textbf{A4')} Suppose that for each $\xi \neq \mathbf{0}$, the
partial derivative $D_xh(x,y;\beta_\ast)$ satisfies, 
\begin{align*}
\Pr\left( \big\Vert \xi^T D_{x}h(X, Y;\beta_\ast)\big\Vert_p > 0 \right) > 0.
\end{align*}

\noindent \textbf{A6')} Assume that there exists $\bar{\kappa}: \mathbb{R}^m
\rightarrow \infty$ such that 
\begin{align*}
\Vert D_xh(x + \Delta, y;\beta_\ast) - D_xh(x,y;\beta_\ast)\Vert_p \leq \bar{%
\kappa}(x,y) \Vert \Delta \Vert_q,
\end{align*}
for all $\Delta \in \mathbb{R}^d,$ and ${\E}[\bar{\kappa}(X,Y)^2] <
\infty.$

\begin{lemma}
If $\rho \geq 2,$ under Assumptions A2'), A4') and A6'), we have, 
\begin{align*}
nR_n(\beta_\ast;\rho) \Rightarrow \bar{R}(\rho),
\end{align*}
where 
\begin{align*}
\bar{R}(\rho) = \sup_{\xi \in \mathbb{R}^d} \bigg\{ \rho \xi^T H - (\rho-1)
\E \left\Vert \xi^TD_xh(X,Y;\beta_\ast)\right\Vert_p^{\rho/(\rho-1)} \bigg\},
\end{align*}
with $H \sim \mathcal{N}(\mathbf{0}, \text{Cov}[h(X,Y;\beta_\ast)]$ and $1/p
+ 1/q = 1.$ \label{Lem-Reg-RWP}
\end{lemma}

\begin{lemma}
If $\rho = 1,$ in addition to assuming A2'), A4'), suppose that $%
D_xh(\cdot,y;\beta_\ast)$ is continuous for every $y$ in the support of
probability distribution of $Y.$ Also suppose that $X$ has a positive
probability density (almost everywhere) with respect to the Lebesgue
measure. Then, 
\begin{align*}
nR_n(\beta_\ast;1) \Rightarrow \bar{R}(1),
\end{align*}
where 
\begin{align*}
\bar{R}(1) = \sup_{\xi: \Pr\left( \Vert \xi^T D_x h(X,Y;\beta_\ast)\Vert_p >
1\right) = 0} \big\{\xi^TH\big\},
\end{align*}
with $H \sim \mathcal{N}(\mathbf{0}, \text{Cov}[h(X,Y;\beta_\ast]).$ \label%
{Lem-Reg-RWP-Eq-1}
\end{lemma}

\noindent 
The proof of Lemma \ref{Lem-Reg-RWP} and \ref{Lem-Reg-RWP-Eq-1}
follows closely the proof of our results in Section \ref{Sec_WPF} and
therefore it is omitted. We prove Theorem \ref{Thm-WPF} and
\ref{Thm-GLM-WPF} as a quick application of these lemmas.

\begin{proof}[Proof of Theorem \protect\ref{Thm-WPF}]
To show that the RWP function dual formulation in %
\eqref{RWP-Dual-Lin-Log-Reg} converges in distribution, we verify the
assumptions of Lemma \ref{Lem-Reg-RWP} with $h(x,y;\beta) = (y-\beta^Tx)x.$
Under the null hypothesis $H_0,$ $Y - \beta_\ast^T X = e$ is independent of $%
X,$ has zero mean and finite variance $\sigma^2.$ Therefore, 
\begin{align*}
\E\left[ h(X,Y;\beta) \right] &= \E\left[eX \right] = 0, \text{ and } \\
\E \Vert h(X,Y;\beta) \Vert_2^2 &= \E\left[ e^2 X^TX \right] = \sigma^2 \E\Vert
X \Vert_2^2,
\end{align*}
which is finite, because trace of the covariance matrix $\Sigma$ is finite.
This verifies Assumption A2'). Further, 
\begin{align*}
D_xh(X,Y;\beta_\ast) = \big(y-\beta_\ast^TX \big)I_d - X\beta_\ast^T = eI_d
- X\beta_\ast^T,
\end{align*}
where $I_d$ is the $d \times d$ identity matrix. For any $\xi \neq \mathbf{0}%
, $ 
\begin{align*}
\Pr\left( \Vert \xi^T D_xh(X,Y;\beta_\ast)\Vert_p = 0\right) = \Pr \left( e\xi =
(\xi^TX) \beta \right) = 0,
\end{align*}
thus satisfying Assumption A4') trivially. In addition, 
\begin{align*}
\Vert D_xh(x+\Delta, y;\beta_\ast) - D_xh(x,y;\beta_\ast) \Vert_p &=
\left\Vert \beta_\ast^T\Delta I_d - \Delta \beta_\ast^T\right\Vert_p \leq c
\Vert \Delta \Vert_q,
\end{align*}
for some positive constant $c.$ This verifies Assumption A6'). As all the
assumptions imposed in Lemma \ref{Lem-Reg-RWP} are easily satisfied, using $%
\rho = 2,$ we obtain the following convergence in distribution as a
consequence of Lemma \ref{Lem-Reg-RWP}. 
\begin{align*}
R_n(\beta_\ast) \Rightarrow \sup_{\xi \in \mathbb{R}^d} \bigg\{ 2\xi^TH -
\E\left\Vert e\xi - (\xi^TX)\beta_\ast\right\Vert_p^2\bigg\},
\end{align*}
as $n \rightarrow \infty.$ Here, $H \sim \mathcal{N}(\mathbf{0}, \text{Cov}%
[h(X,Y;\beta_\ast)].$ As $\text{Cov}[h(X,Y;\beta_\ast)] = \E\left[ e^2 X X^T%
\right] = \sigma^2 \Sigma,$ if we let $Z = H/\sigma,$ we obtain the limit
law, 
\begin{align*}
L_1 = \sup_{\xi \in \mathbb{R}^d}\bigg\{ 2\sigma \xi^TZ - \E\left\Vert e\xi -
(\xi^TX)\beta_\ast\right\Vert_p^2\bigg\},
\end{align*}
where $Z = \mathcal{N}(\mathbf{0}, \Sigma),$ as in the statement of the
theorem.

\noindent \textit{Proof of the stochastic upper bound in Theorem \ref%
{Thm-WPF}:} For the stochastic upper bound, let us consider the asymptotic
distribution $L_{1}$ and rewrite the maximization problem as, 
\begin{align*}
L_1 &= \sup_{\left\Vert \xi\right\Vert_{p}=1 }\sup_{\alpha\geq 0}\left\{
2\sigma \alpha \xi ^{T}Z-\alpha^{2}\E\left\Vert e\xi -(\xi ^{T}X)\beta_{\ast
}\right\Vert _{p}^{2}\right\} \\
&\leq \sup_{\left\Vert\xi\right\Vert_{p}=1 }\sup_{\alpha\geq 0}\left\{
2\sigma\alpha \left\Vert Z\right\Vert_{q}-\alpha^{2}\E\left\Vert e\xi -(\xi
^{T}X)\beta_{\ast }\right\Vert _{p}^{2}\right\},
\end{align*}
because of H\"{o}lder's inequality. By solving the inner optimization
problem in $\alpha,$ we obtain 
\begin{align}
L_1 &\leq \sup_{\left\Vert\xi\right\Vert_{p}=1 } \frac{\sigma^2\left\Vert
Z\right\Vert_{q}^{2}}{\E\left\Vert e\xi -(\xi^{T}X)\beta_{\ast }\right\Vert
_{p}^{2}} = \frac{ \sigma^2\left\Vert Z\right\Vert_{q}^{2}}{%
\inf_{\left\Vert\xi\right\Vert_{p}=1}\E\left\Vert e\xi -(\xi
^{T}X)\beta_{\ast }\right\Vert _{p}^{2}}.  \label{Linear_Upper}
\end{align}
Next, consider the minimization problem in the denominator: Due to triangle
inequality, 
\begin{align*}
\inf_{\left\Vert\xi\right\Vert_{p}=1}\E\left\Vert e\xi -(\xi
^{T}X)\beta_{\ast }\right\Vert _{p}^{2} &\geq
\inf_{\left\Vert\xi\right\Vert_{p}=1} \E\left(\left\vert e\right\vert
\left\Vert
\xi\right\Vert_{p}-\left\vert\xi^{T}X\right\vert\left\Vert\beta_{\ast}\right%
\Vert_{p} \right)^{2} \\
&= \E\left\vert e\right\vert^{2}
+\inf_{\left\Vert\xi\right\Vert_{p}=1}\left\{ \left\Vert
\beta_{\ast}\right\Vert_{p}^{2}\E\left\vert \xi^{T}X\right\vert^{2}
-2\left\Vert \beta_{\ast}\right\Vert_{p} \E\left\vert e\right\vert
\E\left\vert \xi^{T}X\right\vert \right\} \\
&\geq \E\left\vert e\right\vert^{2}
+\inf_{\left\Vert\xi\right\Vert_{p}=1}\left\{ \left\Vert
\beta_{\ast}\right\Vert_{p}^{2}\left(\E\left\vert
\xi^{T}X\right\vert\right)^{2} -2 \left\Vert \beta_{\ast}\right\Vert_{p}
\E\left\vert e\right\vert \E\left\vert \xi^{T}X\right\vert\right\} \\
&= \E\left\vert e\right\vert^{2} - \left(\E\left\vert e\right\vert\right)^{2}
+ \inf_{\left\Vert\xi\right\Vert_{p}=1}
\left(\left\Vert\beta_{\ast}\right\Vert_{p}\E\left\vert \xi^{T}X\right\vert -
\E\left\vert e\right\vert\right)^{2} \\
&\geq \E\left\vert e\right\vert^{2} - \left(\E\left\vert
e\right\vert\right)^{2} = \text{Var}\left[\left\vert e\right\vert\right].
\end{align*}
Combining the above inequality with (\ref{Linear_Upper}), we obtain, 
\begin{eqnarray*}
\sup_{\xi \in \mathbb{R}^{d}}\left\{ \sigma^2\xi ^{T}Z-\E\left\Vert e\xi
-(\xi ^{T}X)\beta_{\ast }\right\Vert _{p}^{2}\right\} \leq \frac{%
\sigma^2\left\Vert Z\right\Vert_{q}^{2}}{\text{Var}\left\vert
  e\right\vert}. 
\end{eqnarray*}
Consequently, 
\begin{align*}
nR_{n}(\beta _{\ast })\overset{D}{\longrightarrow }\ L_{1}:=\ \max_{\xi \in 
\mathbb{R}^{d}}\left\{ \sigma\xi ^{T}Z-\E\left\Vert e\xi -(\xi ^{T}X)\beta
_{\ast }\right\Vert _{p}^{2}\right\} \overset{D}{\leq } \frac{\E[e^{2}]}{%
\E[e^{2}] - (\E\left\vert e\right\vert)^{2} } \Vert {Z}\Vert_{q}^2.
\end{align*}
If random error $e$ is normally distributed, then 
\begin{equation*}
nR_{n}(\beta _{\ast }) \lesssim_{D} \frac{\pi}{\pi - 2}\Vert {Z}%
\Vert_{q}^2,
\end{equation*}
thus establishing the desired upper bound.
\end{proof}

\begin{proof}[Proof of Theorem \protect\ref{Thm-GLM-WPF}]
Under null hypothesis $H_0,$ the training samples $(X_1,Y_1),%
\ldots,(X_n,Y_n) $ are produced from the logistic regression model with
parameter $\beta_\ast$. As $\beta_\ast$ minimizes the expected
log-exponential loss $l(x,y;\beta),$ the corresponding optimality condition
is $\E[h(X,Y;\beta_\ast)] = \mathbf{0}, $ where 
\begin{align*}
h(x,y;\beta_\ast) = \frac{-yx}{1+\exp(y\beta_\ast x)}.
\end{align*}
As $\E\Vert h(X,Y;\beta_\ast)\Vert_2^2 \leq \E\Vert X \Vert_2^2$ is finite,
Assumption A2') is satisfied. Let $I_d$ denote $d \times d$ identity matrix.
While 
\begin{align*}
D_xh(x,y;\beta_\ast) = \frac{-yI_d}{1+\exp(y\beta_\ast^Tx)} +
  \frac{x\beta_\ast^T}{(1+\exp(y\beta_\ast^Tx))(1+\exp(-y\beta_\ast^Tx))} 
\end{align*}
is continuous (as a function of $x$) for every $y,$ it is also true that 
\begin{align*}
\Pr \left( \left\Vert \xi^T D_xh(X,Y;\beta_\ast)\right\Vert_p = 0 \right) = \Pr
\left( Y\big( 1+\exp(-Y\beta_\ast^T X)\big) \xi = (\xi^TX)\beta \right) = 0,
\end{align*}
for any $\xi \neq \mathbf{0},$ thus satisfying Assumption A4'). As all the
conditions required for the convergence in distribution in Lemma \ref%
{Lem-Reg-RWP-Eq-1} are satisfied, we obtain, 
\begin{align*}
\sqrt{n}R_n(\beta_\ast) \Rightarrow \sup_{\xi \in A} \xi^TZ,
\end{align*}
where $Z \sim \mathcal{N}(\mathbf{0}, \E[XX^T/(1+\exp(Y\beta_\ast^TX))^2])$
as a consequence of Lemma \ref{Lem-Reg-RWP-Eq-1}. Here, the set $A = \{ \xi
\in \mathbb{R}^d: \text{ess}\ \text{sup}\Vert \xi^TD_xh(X,Y;\beta_\ast)
\Vert \leq 1\}.$

\noindent \textit{Proof of the stochastic upper bound in Theorem \ref%
{Thm-GLM-WPF}:} First, we claim that $A$ is a subset of the norm ball $\{
\xi \in \mathbb{R}^d: \Vert \xi \Vert_p \leq 1\}.$ To establish this, we
observe that, 
\begin{align}
\left \Vert \xi^T D_xh(X,Y;\beta_\ast) \right \Vert_p &\geq \left \Vert 
\frac{-Y\xi}{1 + \exp(Y\beta_\ast^TX)} \right \Vert_p - \left \Vert \frac{%
(\xi^TX)\beta_\ast}{\big( 1 + \exp(Y\beta_\ast^TX) \big)\big( 1 +
\exp(Y\beta_\ast^TX)\big)} \right \Vert_p  \notag \\
&\geq \left( \frac{1}{1 + \exp(Y\beta_\ast^TX)} - \frac{\Vert X \Vert_q
\Vert \beta_\ast \Vert_p}{\big(1 + \exp(Y\beta_\ast^TX)\big)\big(1 +
\exp(-Y\beta_\ast^TX)\big)}\right)\Vert \xi \Vert_p,  \label{Inter-Bnd-LogR}
\end{align}
because $Y \in \{+1,-1\},$ and due to H\"{o}lder's inequality $\vert \xi^TX
\vert \leq \Vert \xi \Vert_p \Vert X \Vert_q.$ If $\xi \in \mathbb{R}^d$ is
such that $\Vert \xi \Vert_p = (1-\epsilon)^{-2} > 1$ for a given $\epsilon
> 0,$ then following (\ref{Inter-Bnd-LogR}), $\Vert \xi^T D_xh(X,Y)\Vert_p >
1, $ whenever 
\begin{align*}
(X,Y) \in \Omega_\epsilon := \left\{ (x,y): \frac{\Vert x \Vert_q \Vert
\beta_\ast \Vert_p}{1+\exp(-y\beta_\ast^Tx)} \leq \frac{\epsilon}{2}, \ 
\frac{1}{1+\exp(y\beta_\ast^Tx)} \geq 1 - \frac{\epsilon}{2} \right\}.
\end{align*}
Since $X$ has positive density almost everywhere, the set $\Omega_\epsilon$
has positive probability for every $\epsilon > 0.$ Thus, if $\Vert \xi
\Vert_p > 1,$ $\| \xi^TD_xh(X,Y;\beta_\ast)\|_p > 1$ with positive
probability. Therefore, $A$ is a subset of $\{ \xi : \Vert \xi \Vert_p \leq
1\}.$ Consequently, 
\begin{align*}
L_3 := \sup_{\xi \in A} \xi^T Z \ \overset{D}{\leq} \ \sup_{\xi: \Vert \xi
\Vert_p \leq 1} \xi^TZ = \Vert Z \Vert_q.
\end{align*}
If we let $\tilde{Z} \sim \mathcal{N}(\mathbf{0}, \E[XX^T]),$ then $\text{Cov}%
[\tilde{Z}] - \text{Cov}[Z]$ is positive definite. As a result, $L_3$ is
stochastically dominated by $L_4 := \Vert \tilde{Z} \Vert_q,$ thus verifying
the desired stochastic upper bound in the statement of Theorem \ref%
{Thm-GLM-WPF}.
\end{proof}

\begin{proof}[\textbf{Proof of Theorem \protect\ref{Thm-RWP-UB-d-Growing}}]
Instead of characterizing the exact weak limit, we will find a stochastic
upper bound for $R_n(\beta_\ast).$ The RWP function, as in the proof of
Theorem \ref{Thm-WPF}, admits the following dual representation (see %
\eqref{RWP-Dual-Lin-Log-Reg}): 
\begin{align*}
R_n(\beta_\ast) &= \sup_{\lambda} \left\{ -\frac{1}{n} \sum_{i=1}^n
\sup_{x^{\prime }} \left\{ \lambda^T(Y_i - \beta_\ast^Tx^{\prime })x^{\prime
}-\Vert x^{\prime }- X_i \Vert_\infty^2\right\}\right\} \\
&= \sup_{\lambda} \left\{ -\lambda^T \frac{Z_n}{\sqrt{n}} - \frac{1}{n}%
\sum_{i=1}^n \sup_{\Delta} \bigg\{ e_i \lambda^T\Delta -
(\beta_\ast^T\Delta)(\lambda^T X_i) - \big( \Vert \Delta \Vert_\infty^2 +
(\beta_\ast^T\Delta)(\lambda^T\Delta)\big)\bigg\} \right\},
\end{align*}
where $Z_n = n^{-1/2} \sum_{i=1}^n e_iX_i,$ $e_i = Y_i - \beta_\ast^T X_i.$
In addition, we have changed the variable from $x^{\prime }- X_i = \Delta.$
If we let $\zeta = \sqrt{n}\lambda,$ then 
\begin{align*}
n&R_n(\beta_\ast) = \sup_{\zeta} \left\{ -\zeta^TZ_n - \frac{1}{\sqrt{n}}
\sum_{i=1}^n \sup_{\Delta}\bigg\{ e_i \zeta^T\Delta -
(\beta_\ast^T\Delta)(\zeta^T X_i) - \big( \sqrt{n} \Vert \Delta
\Vert_\infty^2 + (\beta_\ast^T\Delta)(\zeta^T\Delta)\big) \bigg\} \right\} \\
&\leq \sup_{\zeta} \left\{ -\zeta^TZ_n - \frac{1}{\sqrt{n}} \sum_{i=1}^n
\sup_{\Vert\Delta\Vert_\infty}\bigg\{ \left\Vert e_i \zeta^T - (\zeta^T X_i)
\beta_\ast^T\right\Vert_1 \Vert \Delta \Vert_\infty - \sqrt{n}\left( 1 + 
\frac{\Vert\beta_\ast\Vert_1 \Vert \zeta \Vert_1}{\sqrt{n}} \right) \Vert
\Delta \Vert_\infty^2\bigg\} \right\},
\end{align*}
where we have used H\"{o}lder's inequality thrice to obtain the upper bound.
If we solve the inner supremum over the variable $\Vert \Delta \Vert,$ we
obtain, 
\begin{align*}
nR_{n}(\beta _{\ast }) & \leq\sup_{\zeta }\left\{ -\zeta ^{T}Z_{n}-\frac{1}{%
\sqrt{n}}\sum_{i=1}^{n} \frac{\left\Vert e_i \zeta - (\zeta^T X_i)
\beta_\ast\right\Vert_1^2}{4\sqrt{n} \left(1 + \Vert\beta_\ast\Vert_1 \Vert
\zeta \Vert_1n^{-1/2}\right)} \right\} \\
& \leq \sup_{a \geq 0}\sup_{\zeta: \Vert \zeta \Vert_1 = 1} \left\{
-a\zeta^TZ_n - \frac{a^2}{4 \left(1 + a\Vert\beta_\ast\Vert_1 n^{-1/2}\right)%
} \frac{1}{n} \sum_{i=1}^n \left\Vert e_i \zeta - (\zeta^T X_i)
\beta_\ast\right\Vert_1^2\right\},
\end{align*}
where we have split the optimization into two parts: one over the
magnitude (denoted by $a$), and another over all unit vectors $\zeta.$
Further, due to H\"{o}lder's inequality, we have
$\vert \zeta^T Z_n \vert \leq \Vert Z_n \Vert_\infty$ as
$\Vert \zeta \Vert_1 = 1.$ Therefore, letting
$c_1(n) = \Vert Z_n \Vert_\infty,$
$c_2(n) = \inf_{\zeta: \Vert \zeta \Vert_1 = 1} \frac{1}{n}
\sum_{i=1}^n \left\Vert e_i \zeta - (\zeta^T X_i)
  \beta_\ast\right\Vert_1^2$ and
$c_3(n) = 1 + a\Vert \beta \Vert_1^2n^{-1/2},$ observe that
\begin{align*}
  nR_n(\beta_\ast) \leq \sup_{a \geq 0} \left\{c_1(n)a -
  \frac{c_2(n)}{4c_3(n)} a^2
  \right\} = \frac{c_1^2(n)}{c_2(n)}(1+o(1)) = \frac{\Vert Z_n\Vert_\infty^2(1+o(1))}{ \inf_{\{\zeta:   \Vert \zeta \Vert_1 = 1\}}\frac{1}{n} \sum_{i=1}^n \left\Vert e_i \zeta - (\zeta^T
  X_i) \beta_\ast\right\Vert_1^2}.
\end{align*}
Since
$\left\Vert e_i\zeta - (\zeta^T X_i) \beta_\ast\right\Vert_1^2 \geq
\left( \left\vert e_{i}\right\vert \left\Vert {\zeta } \right\Vert
  _{1}-\left\vert \zeta ^{T}X_{i}\right\vert \left\Vert \beta _{\ast
    }\right\Vert _{1}\right) ^{2},$ the denominator, $c_2(n),$ can be
lower bounded as follows:
\begin{align*}
  c_2(n) &:= \inf_{\zeta: \Vert \zeta \Vert_1 = 1}\E_{\Pr_n}\left\Vert e\zeta - (\zeta^T X)
\beta_\ast\right\Vert_1^2 \geq \inf_{\zeta: \Vert \zeta \Vert_1 = 1}
           \E_{\Pr_n}\left[ \left( \vert e \vert - \vert \zeta^TX\vert
           \Vert \beta_\ast \Vert_1\right)^2\right]\\
  &\geq \E_{\Pr_n}\left[ \inf_{\zeta: \Vert \zeta \Vert_1 = 1}\E_{\Pr_n}\left[\left( \vert e \vert - \vert \zeta^TX\vert
           \Vert \beta_\ast \Vert_1\right)^2 \,\vert\, X \right]\right] \geq
    \E_{\Pr_n}\left[ \inf_{c \in \mathbb{R}}\E_{\Pr_n}\left[\left( \vert e
    \vert - c\right)^2 \,\vert\, X \right]\right]. 
\end{align*}
Since $e_i$ and $X_i$ are independent and
$\min_{c} \E[(Z-c)^2] = \text{Var}[Z]$ for any random variable $Z,$ we
obtain that $c_2(n) \geq \text{Var}_n\vert e \vert.$
Therefore
$nR_{n}(\beta _{\ast })\leq {\left\Vert Z_{n}\right\Vert
  _{\infty}^{2}(1+o(1))}/{ \text{Var}_n\left\vert e \right\vert}.$ 
\end{proof}

\section{Strong duality for the linear semi-infinite program resulting
  from the RWP function}
\label{AppSec-Str-Duality} In the main body of the paper, we have
utilized strong duality of linear semi-infinite
programs 
to derive a dual representation of the RWP function in order to
perform asymptotic analysis (see Proposition
\ref{Prop-RWP-Duality}).  
Establishing strong duality in this context relies on the following
well-known result on problem of moments (\cite{Isii1962,newey_higher_2004}).

\noindent \textbf{The problem of moments.} Let $\Omega$ be a nonempty Borel
measurable subset of $\mathbb{R}^m,$ which, in turn, is endowed with the
Borel sigma algebra $\mathcal{B}_{\Omega}.$ Let $X$ be a random vector
taking values in the set $\Omega,$ and $f = (f_1,\ldots,f_k): \Omega
\rightarrow \mathbb{R}^k$ be a vector of moment functionals. Let $\mathcal{P}%
_\Omega$ and $\mathcal{M}^+_\Omega$ denote, respectively, the set of
probability and non-negative measures, respectively on $(\Omega, \mathcal{B}_\Omega)$ such
that the Borel measurable functionals $\phi, f_1, f_2,\ldots,f_k,$ defined
on $\Omega,$ are all integrable. Given a real vector $q = (q_1,\ldots,q_k),$
the objective of the problem of moments is to find the worst-case bound, 
\begin{align}
v(q) := \sup \big\{ \E_\mu[\phi(X)]: \E_\mu[f(X)] = q, \ \mu \in \mathcal{P}%
_\Omega \big\}.  \label{Mom-Prob}
\end{align}
If we let $f_0 = \mathbf{1}_\Omega,$ it is convenient to add the constraint, 
$\E_\mu[f_0(X)] = 1,$ by appending $\tilde{f} = (f_0, f_1,\ldots, f_k), \ 
\tilde{q} = (1,q_1,\ldots,q_k),$ and consider the following reformulation of
the above problem: 
\begin{align}
v(q) := \sup \left\{ \int \phi(x) d\mu(x): \int \tilde{f}(x)d\mu(x) = \tilde{%
q}, \ \mu \in \mathcal{M}^+_\Omega \right\}.  \label{Mom-Prob-II}
\end{align}
Then, under the assumption that a certain Slater's type of condition is
satisfied, one has the following equivalent dual representation for the
moment problem \eqref{Mom-Prob-II}. See Theorem 1 (and the discussion of
Case [I] following Theorem 1) in \cite{Isii1962} for a proof of the
following result:

\begin{proposition}
\label{Prop-Moment-Prob} Let $\mathcal{Q}_{\tilde{f}} = \big\{ \int \tilde{f}%
(x)d\mu(x): \mu \in \mathcal{M}^+_\Omega\big\}.$ If $\tilde{q} =
(1,q_1,\ldots,q_k)$ is an interior point of $\mathcal{Q}_{\tilde{f}},$ then 
\begin{align*}
v(q) = \inf \left\{ \sum_{i=0}^k a_iq_i: \ a_i \in \mathbb{R}, \
\sum_{i=0}^k a_i \tilde{f}_i(x) \geq \phi(x) \text{\ for all } x \in
\Omega \right\}.
\end{align*}
\end{proposition}

In the rest of this section, we recast the dual reformulation of RWP
function (in \eqref{Prop-RWP-Duality}) and the dual reformulation of the
distributional representation in Proposition \ref{Prop-Duality} as
particular cases of the dual representation of the problem of moments in
Proposition \ref{Prop-Moment-Prob}.

\noindent \textbf{Dual representation of RWP function.} Recall from Section %
\ref{Sec-Dual-RWP} that $W$ is a random vector taking values in $\mathbb{R}%
^m $ and $h(\cdot,\theta)$ is Borel measurable.

\begin{proof}[Proof of Proposition \protect\ref{Prop-RWP-Duality}]
For simplicity, we do not write the dependence on parameter $\theta$ in $%
h(u,\theta)$ and $R_n(\theta)$ in this proof; nevertheless, we should keep
in mind that the RWP function is a function of parameter $\theta.$ Given
estimating equation $\E[h(W)] = \mathbf{0},$ recall the definition of the
corresponding RWP function, 
\begin{align*}
R_n &:= \inf \left\{ D_c(\Pr,\Pr_n): \ \E_\Pr\big[ h(W) \big] = \mathbf{0}\right\}
\\
&= \inf \left\{ \E_\pi\big[ c(U,W) \big]: \ \E_\pi \big[ h(U) \big] = \mathbf{0%
}, \ \pi_{_W} = \Pr_n, \ \pi \in \mathcal{P}(\mathbb{R}^m \times \mathbb{R}^m)
\right\},
\end{align*}
where $\pi_{_W}$ denotes the marginal distribution of $W$ and
$\mathbb{P}_n$ is the empirical distribution formed from distinct
samples $\{W_1,\ldots,W_n\}.$ To recast this as a problem of moments
as in \eqref{Mom-Prob}, let $\Omega = \{(u,w) \in \mathbb{R}%
^m \times \{W_1,\ldots,W_n\}: c(u,w) < \infty\},$
\begin{align*}
f(u,w) &= 
\begin{bmatrix}
\mathbf{1}_{_{\{w=W_1\}}}(u,w) \\ 
\mathbf{1}_{_{\{w=W_2\}}}(u,w) \\ 
\vdots \\ 
\mathbf{1}_{_{\{w=W_n\}}}(u,w) \\ 
h(u) \\ 
\end{bmatrix}
\quad \quad \text{ and } \quad \quad q = 
\begin{bmatrix}
1/n \\ 
1/n \\ 
\vdots \\ 
1/n \\ 
\mathbf{0} \\ 
\end{bmatrix}%
.
\end{align*}
Further, let $\phi(u,w) = -c(u,w),$ for all $(u,w) \in \Omega.$ Then,
\begin{align*}
R_n = -\sup\left\{ \E_\pi\big[\phi(U,W)\big]: \ \E_\pi \big[ f(U,W) \big] = q,
\ \pi \in \mathcal{P}_\Omega \right\},
\end{align*}
is of the same form as \eqref{Mom-Prob}. Since the constraints
$\mathbb{E}_\pi[\mathbf{1}_{\{w = W_i\}}(U,W)] = 1/n,$ for
$i = 1,\ldots,n,$ together specify that $\mathbb{P}_\pi(\Omega) = 1,$
the constraint that $E_\pi[\mathbf{1}_\Omega(U,W)] = 1$ is
redundant. Moreover, as $\{\mathbf{0\}}$ lies in the interior of
convex hull of the range $\{h(u): (u,w) \in \Omega \},$ observe that
the set
$\mathcal{Q}_{f} := \{ \int fd\mu: \mu \in \mathcal{M}^+_{\Omega}\}$
is simply $\mathbb{R}^n_+ \times \mathbb{R}.$ Then it is immediate
that the Slater's condition $q \in \text{int}(\mathcal{Q}_f)$ is
satisfied for the moment problem,
\begin{align*}
  R_n =  -\sup\left\{ \int \phi(u,w) d\mu(u,w): \int f(u,w)d\mu(u,w) = q, \mu \in \mathcal{M}_{\Omega}^+\right\}. 
\end{align*}
Consequently, we obtain the following dual representation of $R_n$ due
to Proposition \ref{Prop-Moment-Prob}:
\begin{eqnarray*}
R_n &=& -\inf_{a_i \in \mathbb{R}} \left\{ \frac{1}{n} \sum_{i=1}^n
        a_i:  \ \sum_{i=1}^na_i\mathbf{1}_{_{\{ w = W_i\}}}(u,w)
+ \sum_{i=n+1}^k a_i h_i(u) \geq -c(u,w), \text{ for all } (u,w) \in \Omega
\right\} \\
&=& -\inf_{a_i \in \mathbb{R}} \left\{ \frac{1}{n} \sum_{i=1}^n a_i: \
a_i \geq \sup_{u: c(u,W_i) < \infty} \left\{-c(u,W_i) - \sum_{i=n+1}^k
a_i h_i(u) \right\}\right\}.
\end{eqnarray*}
As the inner supremum is not affected even if we take supremum over $\{ u:
c(u,W_i) = \infty\},$ after letting $\lambda = (a_{n+1},\ldots,a_k)$ for
notational convenience, we obtain 
\begin{align}
R_n &= \sup_{\lambda} \left\{ \frac{1}{n} \sum_{i=1}^n \inf_{u \in \mathbb{R}%
^m} \left\{ c(u,W_i) + \lambda^T h(u) \right\}\right\}.  \label{RWP-Dual-Rep}
\end{align}
As $\lambda$ is a free variable, we flip the sign of $\lambda$ to arrive at
the statement of Proposition \ref{Prop-RWP-Duality}. This completes the
proof.
\end{proof}

\section{Exchange of sup and inf in the DRO formulation
  \eqref{RWPI-DRO}}
\label{Sec-App-Sup-Inf}
The inf-sup exchange in Proposition \ref{Prop_min_max} below is
obtained by suitably modifying the inf-sup exchange in \cite[Theorem
2]{Confidence_prep} and its proof to accommodate more relaxed
assumptions than in \cite{Confidence_prep}. The sequence of steps in
the proof of Proposition \ref{Prop_min_max} is similar to that of
\cite[Theorem 2]{Confidence_prep} and is given here for completeness.

\begin{proposition}
  \label{Prop_min_max} For a given probability distribution
  $\mathbb{Q},$ define
  \[g(\beta) :=\sup_{\Pr:\, \mathcal{D}_c(\Pr,\mathbb{Q}) \leq
      \,\delta}\E_{\Pr}\left[ l\big( %
      X,Y;\beta \big)\right],\] for $\beta \in \mathbb{R}^d.$ Suppose
  that $g( \cdot) $ is real-valued and the level set
  $\{\beta \in \mathbb{R}^d:g\left( \beta \right) \leq b\}$ is bounded
  for every $b\in \mathbb{R}$.  In addition, suppose that
  $\E_{\Pr}\left[ l\big(X,Y;\beta \big)\right] $ is convex and lower
  semicontinuous in the variable $\beta,$ for every
  $\Pr \in \mathcal{U}_\delta(\mathbb{Q}) := \{\Pr:\,
  \mathcal{D}_c(\Pr,\mathbb{Q}) \leq \,\delta \}.$ 
  Then,
\begin{equation*}
\inf_{\beta \in \mathbb{R}^{d}}\sup_{\Pr: \mathcal{D}_{c}(\Pr,\mathbb{Q})\leq \delta
}\E_{\Pr}\left[ l\big(X,Y;\beta \big)\right] =\sup_{\Pr: \mathcal{D}_{c}(\Pr,\mathbb{Q}) %
\leq \delta }\inf_{\beta \in \mathbb{R}^{d}}\E_{\Pr}\left[ l\big(X,Y;\beta \big)%
\right] .
\end{equation*}
\end{proposition}

\begin{proof}
  We begin by defining the sequence of approximation problems,
  \begin{align*}
    g_N(\beta) := \sup_{\Pr \in \mathcal{U}_\delta^N(\mathbb{Q})}
    \E_{\Pr}\left[ l\big(X,Y;\beta \big)\right],
  \end{align*}
  where $N = 1,2,\ldots,$ and
  \begin{align*}
    \mathcal{U}_\delta^N(\mathbb{Q}) = \left\{ \Pr \in \mathcal{P}(\mathcal{K}_N):
    \mathcal{D}_c(\Pr,\mathbb{Q}) \leq \delta \right\},
  \end{align*}
  with $\mathcal{P}(\mathcal{K}_N)$ denoting the set of probability
  distributions over the set
  $\mathcal{K}_N := \left\{ x : \Vert x \Vert_2 \leq N\right\}.$ Then,
  due to the compactness of the set
  $\mathcal{U}^N_\delta(\mathbb{Q}),$ we obtain
  \begin{align*}
    \inf_{\beta \in \mathbb{R}^d} g_N(\beta) = \inf_{\beta \in \mathbb{R}^d} 
    \sup_{\Pr \in \mathcal{U}_\delta^N(\mathbb{Q}) }
    \E_{\Pr}\left[ l\big(X,Y;\beta \big)\right] 
    = \sup_{\Pr \in \mathcal{U}_\delta^N(\mathbb{Q}) } \inf_{\beta \in \mathbb{R}^d}
    \E_{\Pr}\left[ l\big(X,Y;\beta \big)\right], 
  \end{align*}
  as a consequence of Sion's minimax theorem
  \cite{sion_general_1958}. Therefore, with $g_N(\cdot)$ being an
  increasing sequence of functions, we have
  \begin{align}
    \lim_{N \rightarrow \infty} \inf_{\beta \in \mathbb{R}^d} g_N(\beta) 
    &= \sup_{N \geq 1} \inf_{\beta \in \mathbb{R}^d} g_N(\beta) 
      = \sup_{N \geq 1}\sup_{\Pr \in \mathcal{U}_\delta^N(P_n) } \inf_{\beta \in \mathbb{R}^d}
      \E_{\Pr}\left[ l\big(X,Y;\beta \big)\right] \nonumber\\
    &\leq \sup_{\Pr \in \mathcal{U}_\delta (\mathbb{Q}) } \inf_{\beta \in \mathbb{R}^d}
      \E_{\Pr}\left[ l\big(X,Y;\beta \big)\right]
      \leq  \inf_{\beta \in \mathbb{R}^d} \sup_{\Pr \in \mathcal{U}_\delta (\mathbb{Q})}
      \E_{\Pr}\left[ l\big(X,Y;\beta \big)\right] \label{minmax-inter}\\
    &= \inf_{\beta \in \mathbb{R}^d} g(\beta). 
  \nonumber
  \end{align}
  The rest of the proof is divided into three technical steps:
  
  \textbf{Step 1:} In this step, we show that the sequence of
  functions $\{g_{_N}(\cdot): N \geq 1\}$ converges pointwise to the
  function $g(\cdot),$ as $N \rightarrow \infty.$ Since $g_N(\beta)$
  is increasing in $N,$ we have that $g_N(\beta)$ converges as
  $N \rightarrow \infty,$ for every $\beta.$ Let the function
  $g^\ast(\cdot)$ denote the pointwise limit, 
  $g^\ast(\cdot) = \lim_{N \rightarrow \infty} g_N(\cdot).$ With
  $g_N(\cdot) \leq g(\cdot)$ for every $N,$ we have
  $g^\ast(\beta) \leq g(\beta).$ Since $g(\cdot)$ is real-valued, we
  consequently have $g^\ast(\beta) \leq g(\beta) < +\infty,$ for every
  $\beta \in
  \mathbb{R}^d.$ 

  To show that $g^\ast(\beta)$ necessarily equals $g(\beta)$ for every
  $\beta,$ we argue via contradiction as follows: Suppose that
  $\varepsilon := g(\beta) - g^\ast(\beta) > 0$ for some
  $\beta \in \mathbb{R}^d.$ Consider any
  $\mathbb{P}^\prime \in \mathcal{U}_\delta(P_n)$ such that
  $\E_{\Pr^\prime}\left[ l(X,Y;\beta)\right] \in (g(\beta) -
  \varepsilon/2,g(\beta)].$ With $g(\beta)$ being finite, there exists
  $N_0$ sufficiently large such that
  \begin{align*}
    \E_{\Pr^\prime}\left[ l(X,Y;\beta) \mathbb{I}(\Vert X \Vert_2 > N)\right]
    < \varepsilon/4  \quad \text{ and }  \quad
    \left[ 1 - \Pr^\prime(\mathcal{K}_N)\right]
    \E_{\mathbb{Q}} \left[ l(X,Y;\beta)\mathbb{I}(\Vert X \Vert_2 \leq N)\right]
    > -\varepsilon/4,
  \end{align*}
  for all $N > N_0.$ From $\mathbb{P}^\prime,$ we construct a measure
  $\mathbb{P}_N^\prime \in \mathcal{U}_\delta^N(\mathbb{Q})$ by
  letting,
  \begin{align*}
    \mathbb{P}_N^\prime(\cdot) = \Pr^\prime(\cdot) + \left[ 1 - \Pr^\prime(\mathcal{K}_N)\right]
    \frac{\mathbb{Q}(\cdot)}{\mathbb{Q}(\mathcal{K}_N)},
  \end{align*}
  for all $N$ large enough such that $\mathbb{Q}(\mathcal{K}_N) > 0.$ Then,
  \begin{align*}
    g^\ast(\beta) \geq g_N(\beta) \geq \E_{\Pr^\prime_N}\left[ l(X,Y;\beta)\right]
    > \E_{\Pr^\prime}\left[ l(X,Y;\beta)\right] - \varepsilon/2,
  \end{align*}
  for all $N > N_0.$ With
  $\E_{\Pr^\prime}\left[ l(X,Y;\beta)\right] \in (g(\beta) -
  \varepsilon/2,g(\beta)],$ we then have
  $g^\ast(\beta) > g(\beta) - \varepsilon,$ which leads to a
  contradiction to the assumption that
  $\varepsilon := g(\beta) - g^\ast(\beta) > 0.$ This verifies that
  the pointwise limit $g^\ast(\cdot) = g(\cdot).$

  \textbf{Step 2:} In this next step, we show that the sequence of
  functions $\{g_N(\cdot): N \geq 1\}$ epiconverges to the function
  $g(\cdot),$ as $N \rightarrow \infty.$ See, for example,
  \cite[Definition 7.1]{rockafellar2009variational} for a definition
  of epiconvergence. To accomplish this step, we first see that for
  every sequence $\{\beta_N: N \geq 1\}$ satisfying
  $\beta_N \rightarrow \beta \in \mathbb{R}^d,$ 
  \begin{align*}
    \liminf_{N \rightarrow \infty} g_N(\beta_N) \geq
    \liminf_{N \rightarrow \infty} g_M(\beta_N) \geq g_M(\beta),
  \end{align*}
  for any positive integer $M.$ Indeed, this is because $g_N(\cdot)$
  is an increasing sequence of functions and $g_M(\cdot),$ being
  pointwise maxima of lower semicontinuous functions, is
  lower semicontinuous. Letting $M \rightarrow \infty,$ we then have
  \begin{align*}
    \liminf_{N \rightarrow \infty} g_N(\beta_N) \geq g(\beta), 
  \end{align*}
  due to the pointwise convergence concluded in Step 1. Next, for any
  $\beta \in \mathbb{R}^d,$ if we pick the sequence $\beta_N = \beta,$
  we have
  $\lim_{N \rightarrow \infty} g_N(\beta_N) = \lim_{N \rightarrow
    \infty} g_N(\beta) = g(\beta).$
  We therefore have from the epiconvergence characterization in
  \cite[Proposition 7.1]{rockafellar2009variational} that the sequence
  $\{g_N : N \geq 1\}$ epiconverges to the function $g(\cdot).$

  \textbf{Step 3:} In this final step, we show that the optimal values
  $\inf_{\beta \in \mathbb{R}^d} g_N(\beta)$ converge to
  $\inf_{\beta \in \mathbb{R}^d} g(\beta),$ as $N \rightarrow \infty.$
  With $\E_{\Pr}[l(X,Y;\beta)]$ being convex in the variable $\beta,$
  we have that the pointwise maximum $g(\cdot)$ is convex. Combining
  this observation with the level-boundedness of the limiting function
  $g(\cdot),$ we have from \cite[Exercise
  7.32(c)]{rockafellar2009variational} that the sequence
  $\{g_N(\beta): N \geq 1\}$ is eventually level-bounded. Further,
  since the functions $g_N(\cdot), g(\cdot)$ are lower semicontinuous
  and proper, we obtain the desired optimal value convergence,
  \[\inf_{\beta \in \mathbb{R}^d} g_N(\beta) \rightarrow \inf_{\beta
      \in \mathbb{R}^d} g(\beta),\] as a consequence of \cite[Theorem
  7.33]{rockafellar2009variational}.

  The conclusion in Step 3 forces the inequalities in
  (\ref{minmax-inter}) to be equalities, thus rendering the desired
  inf-sup interchange in the statement of Proposition
  \ref{Prop_min_max}.
\end{proof}

\begin{proof}[Proof of Lemma \protect\ref{Lem-inf-sup-exchange-apps}]
  Let us consider linear regression loss function first. Under the
  null hypothesis, $\E\Vert X \Vert_2^2 < \infty$ and
  $\E[e^2] < \infty.$ Therefore, for any $\beta \in \mathbb{R}^d,$
  $\E[l(X,Y;\beta)] = \E[(Y-\beta^TX)^2] < \infty.$ Further, as the
  loss function $l(x,y;\beta)$ is a convex and continuous in the
  variable $\beta,$ we have that $\E_\Pr[l(X,Y;\beta)]$ is convex and
  lower semicontinuous for any $\Pr \in \mathcal{U}_\delta(\Pr_n).$
  Next, the distributionally robust representation in Theorem
  \ref{Cor-lp-good-equiv},
\begin{align*}
g(\beta) = \sup_{\Pr \in \ \mathcal{U}_\delta(\Pr_n)} \E_\Pr[l(X,Y;\beta)] = \left( 
\sqrt{\E_{\Pr_n}[(Y-\beta^TX)^2]} + \sqrt{\delta}\Vert \beta \Vert_p\right)^2
\end{align*}
allows us to conclude that $g(\beta)$ is finite for every
$\beta \in \mathbb{R}^d.$ Further, as $g(\beta) \rightarrow \infty$
when $\Vert \beta \Vert_p \rightarrow \infty$ and $g(\beta)$ is convex
and continuous in $\mathbb{R}^d,$ the level sets
$\{ \beta: g(\beta) \leq b\}$ are compact and nonempty for every
$b > (\sqrt{\E_{\Pr_n}[(Y-\beta_\ast^TX)^2]} + \sqrt{\delta}\Vert
\beta_\ast \Vert)^2.$ This verifies the level-boundedness requirement
in the statement of Proposition \ref{Prop_min_max}.  As all the
conditions in Proposition \ref{Prop_min_max} are satisfied, the sup
and inf in the DRO formulation (\ref{RWPI-DRO}) can be exchanged in
the linear regression example as a consequence of Proposition
\ref{Prop_min_max}. Exactly similar reasoning applies for logistic
regression loss function when $\E\Vert X \Vert_2^2$ is finite.
\end{proof}

\end{document}